\documentclass{article}
\usepackage{amssymb,amscd,amsmath,amsthm,xcolor}
\usepackage{graphics}
\usepackage[dvips]{graphicx}
\usepackage{hyperref}
\usepackage[noabbrev,capitalise]{cleveref}
\usepackage{url}
\usepackage{tcolorbox}
\usepackage{adjustbox}

\usepackage{mathtools}

\def\N{\mathcal{N}}
\def\NN{\mathbb{N}}
\usepackage{helvet}
\def\TT{\mathbb{T}}
\def\PP{\mathbb{P}}
\def\DD{\mathbb{D}}

\def\WW{\mathbb{W}}
\def\MM{\mathbb{M}}
\def\F{\mathcal{F}}
\def\U{\mathcal{U}}
\def\C{\mathcal{C}}

\def\D{\mathcal{D}}
\def\V{\mathcal{V}}
\def\A{\mathcal{A}}
\def\B{\mathcal{B}}
\def\P{\mathcal{P}}
\def\M{\mathcal{M}}

\def\PP{\mathbb{P}}

\def\Psf{\mathsf{P}}
\def\Qsf{\mathsf{Q}}

\usepackage{algorithm2e}
\RestyleAlgo{ruled}
\usepackage{babel}[french]
\usepackage{tikz}
\renewcommand{\L}[0]{\mathcal{L}}

\newcommand{\cs}{2^{\NN}}

\newcommand{\bstr}{2^{<\NN}}

\newcommand{\converge}{\!\!\downarrow}
\newcommand{\diverge}{\!\!\uparrow}

\usepackage{xcolor}

\newcommand{\ISig}{\mathsf{I}\Sigma^0}
\newcommand{\Robinson}{\mathsf{Q}}
\newcommand{\IPi}{\mathsf{I}\Pi^0}
\newcommand{\IDelta}{\mathsf{I}\Delta^0}

\newcommand{\BSig}{\mathsf{B}\Sigma^0}

\newcommand{\CSig}{\mathsf{C}\Sigma}
\newcommand{\WF}{\mathsf{WF}}

\newcommand{\card}{\operatorname{card}}

\newcommand{\qvdash}{\operatorname{{?}{\vdash}}}
\newcommand{\nqvdash}{\operatorname{{?}{\nvdash}}}
\newcommand{\uh}{\upharpoonright}
\newcommand{\Sub}[1]{{#1}\mbox{-}\mathsf{Subset}}
\newcommand{\RCA}{\mathsf{RCA}}
\newcommand{\WKL}{\mathsf{WKL}}
\newcommand{\COH}{\mathsf{COH}}

\newcommand{\RT}{\mathsf{RT}}
\newcommand{\GST}{\mathsf{GST}_1}

\def\qt#1{``#1''}%

\newcommand{\finsub}{\subseteq_{\mathtt{fin}}}

\newtheorem{theorem}{Theorem}
\numberwithin{theorem}{section}
\newtheorem{maintheorem}[theorem]{Main Theorem}
\newtheorem{lemma}[theorem]{Lemma}
\newtheorem{proposition}[theorem]{Proposition}
\theoremstyle{definition}
\newtheorem{definition}[theorem]{Definition}
\newtheorem{corollary}[theorem]{Corollary}

\newtheorem{remark}[theorem]{Remark}
\newtheorem{statement}[theorem]{Statement}
\newtheorem{question}[theorem]{Question}

\makeatletter
\newtheorem*{rep@theorem}{\rep@title}
\newcommand{\newreptheorem}[2]{%
\newenvironment{rep#1}[1]{%
 \def\rep@title{#2 \ref{##1}}%
 \begin{rep@theorem}}%
 {\end{rep@theorem}}}
\makeatother

\newreptheorem{maintheorem}{Main Theorem}

\author{Quentin Le Hou\'erou \and Ludovic Levy Patey \and Ahmed Mimouni}

\title{The reverse mathematics of the pigeonhole hierarchy}

\begin{document}

\maketitle

\begin{abstract}
The infinite pigeonhole principle for $k$ colors ($\RT^1_k$) states, for every  $k$-partition $A_0 \sqcup \dots \sqcup A_{k-1} = \NN$, the existence of an infinite subset~$H \subseteq A_i$ for some~$i < k$. This seemingly trivial combinatorial principle constitutes the basis of Ramsey's theory, and plays a very important role in computability and proof theory. In this article, we study the infinite pigeonhole principle at various levels of the arithmetical hierarchy from both a computability-theoretic and reverse mathematical viewpoint. We prove that the hierarchy of pigeonhole principles induced by restricting instances to levels of the arithmetic hierarchy is strict over~$\RCA_0$ using an elaborate iterated jump control construction, and study its first-order consequences.
This is part of a large meta-mathematical program studying the computational content of combinatorial theorems.
\end{abstract}

\section{Introduction}

The infinite pigeonhole principle for $k$ colors ($\RT^1_k$) states, for every $k$-partition $A_0 \sqcup \dots \sqcup A_{k-1} = \NN$, the existence of an infinite subset~$H \subseteq A_i$ for some~$i < k$. In particular, for~$k = 2$, it is equivalent to the existence, for every set $A \subseteq \NN$, of an infinite set~$H \subseteq A$ or $H \subseteq \overline{A}$. This statement is the basis of Ramsey's theorem, and more generally of Ramsey's theory. Recall that Ramsey's theorem for $n$-tuples states, for every $k$-coloring of $[\NN]^n$, the existence of an infinite set~$H \subseteq \NN$ such that the coloring is monochromatic on~$[H]^n$. Here, $[H]^n$ denotes the set of all subsets of~$H$ of size~$n$. The classical proof of Ramsey's theorem consists of reducing a $k$-coloring of $[\NN]^{n+1}$ to a $k$-coloring of~$[\NN]^n$ by $\omega$ many successive applications of the pigeonhole principle. In the base case, Ramsey's theorem for singletons is nothing but $\RT^1_k$.

In reverse mathematics, the pigeonhole principle plays an important role, both for theoretical and technical reasons. Reverse mathematics is a foundational program at the intersection of computability theory and proof theory, whose goal is to find the optimal axioms to prove theorems from \qt{ordinary mathematics}. The study of Ramsey's theorem for pairs received a particular attention as its strength escaped the empirical structural observation of reverse mathematics~\cite{hirschfeldt2015slicing}. As it turns out, the computability-theoretic study of Ramsey's theorem for pairs is closely related to the study of the pigeonhole principle. Technically, all the computability-theoretic and proof-theoretic subtleties of Ram\-sey's theorem for pairs are already present in the study of $\RT^1_2$ \cite{cholak2001strength,dzhafarov2009ramsey,liu2012rt22,liu2015cone}.

In this article, we study the pigeonhole principle viewed both as a statement in second-order arithmetic and as a mathematical problem, formulated in terms of \emph{instances} and \emph{solutions}. An instance of~$\RT^1_2$ is a set~$A$, and a solution is an infinite subset~$H \subseteq A$ or $H \subseteq \overline{A}$. We study a hierarchy of pigeonhole principles based on the complexity of their instances, and use the frameworks of computability theory and reverse mathematics to give lower and upper bounds to the complexity of a solution. We extend several known results to higher levels of the arithmetic hierarchy, using a unifying framework of iterated jump control for the pigeonhole principle initially developed by Monin and Patey~\cite{monin201pigeons,monin2021weakness}. This answers multiple open questions from Benham et al.~\cite{benham2024ginsburg}. In particular, we prove that every $\Sigma^0_2$ instance of $\RT^1_2$ admits a solution whose jump is computable in any PA degree over $\emptyset'$. This result is obtained by an asymmetric construction whose dividing line is new in computability theory. This can be put in contrast with weak K\"onig's lemma, where the restriction to $\Delta^0_1$ trees is strictly weaker than the restriction to $\Sigma^0_2$ trees over $\RCA_0$, and whose restriction to $\Delta^0_n$ trees is strictly weaker than $\Sigma^0_n$ trees over computable reducibility.

\subsection{Reverse mathematics}

Reverse mathematics is a foundational program started by Harvey Friedman~\cite{friedmanSystemsSecondOrder1975,friedmanSystemsSecondOrder1976}.
This program can be considered as a partial realization of Hilbert's program~\cite{simpson1988partial}, and as an answer to the crisis of foundations. It uses the framework of subsystems of second-order arithmetic, with a base theory, $\RCA_0$, capturing \qt{computable mathematics}. More precisely, here is a formal description of its axioms.

\emph{Robinson arithmetic} $\Robinson$ (Peano arithmetic without induction) is composed of the following axioms:
\newlength \Colsep
\setlength \Colsep{10pt}

\noindent \begin{minipage}{\textwidth}
\begin{minipage}[t] [3cm] [c]{\dimexpr0.4 \textwidth-0.4 \Colsep \relax}
\begin{enumerate}
\item [(1)] $x + 1 \neq 0$
\item [(2)] $x = 0 \vee \exists y\ (x = y + 1)$
\item [(3)] $x + 1 = y + 1 \rightarrow x = y$
\item [(4)] $x + 0 = x$
\end{enumerate}
\end{minipage} \hfill
\begin{minipage}[t] [3cm] [c]{\dimexpr0.6 \textwidth-0.6 \Colsep \relax}

\begin{enumerate}
\item [(5)] $x + (y + 1) = (x+y)+1$
\item [(6)] $x \times 0 = 0$
\item [(7)] $x \times (y+1) = (x \times y) + x$
\item [(8)] $x < y \leftrightarrow \exists z\ (z \neq 0 \wedge x + z = y)$
\end{enumerate}
\end{minipage}
\end{minipage}

A second-order formula is \emph{arithmetic} if it contains only first-order quantifiers (with second-order parameters allowed).  Arithmetic formulas admit a natural classification based on the number of alternations of their quantifiers in prenex normal form. A formula is $\Sigma^0_0$ (or $\Pi^0_0$) if it contains only bounded first-order quantifiers, that is, quantifiers of the form $\forall x \leq t$ or $\exists x \leq t$ where~$x$ is a first-order variable and $t$ is an arithmetic term. A formula is $\Sigma^0_n$ if it is of the form $\exists x_0 \forall x_1 \dots Q x_{n-1} \varphi(x_0, \dots, x_{n-1})$, where~$\varphi$ is a $\Sigma^0_0$-formula and $Q$ is $\forall$ if $n$ is even, and $\exists$ otherwise. $\Pi^0_n$ formulas are defined accordingly, starting with a universal first-order quantifier. A set~$A \subseteq \NN$ is $\Sigma^0_n$ or $\Pi^0_n$ if it is definable by a $\Sigma^0_n$ or $\Pi^0_n$ formula, respectively. It is $\Delta^0_n$ if it is simultaneously $\Sigma^0_n$ and $\Pi^0_n$. By a theorem from G\"odel, $\Delta^0_1$ sets are exactly the computable ones.
The \emph{$\Delta^0_1$-comprehension scheme} is defined for every $\Sigma^0_1$-formula $\varphi(x)$ and every $\Pi^0_1$-formula $\psi(x)$ as 
$$
\forall x(\varphi(x) \leftrightarrow \psi(x)) \to \exists Z \forall x(x \in Z \leftrightarrow \varphi(x))
$$
The left-hand part of the implication ensures that the predicate defined by $\varphi(x)$ is $\Delta^0_1$, while the right-hand part is the classical comprehension axiom for~$\varphi(x)$.
Based on the correspondence between $\Delta^0_1$ and computable sets, the $\Delta^0_1$-comprehension scheme restricts the comprehension scheme to sets which can be obtained computably from their parameters.

Last, the $\Sigma^0_1$-induction scheme is defined for every $\Sigma^0_1$-formula $\varphi(x)$ as
$$
[\varphi(0) \wedge \forall x(\varphi(x) \rightarrow \varphi(x+1))] \to \forall y \varphi(y)
$$
In general, the induction scheme for $\Sigma^0_n$-formulas is equivalent to a bounded version of the $\Sigma^0_n$-comprehension scheme, that is, the existence of every initial segment of the $\Sigma^0_n$-set (see H\'ajek and Pudl\'ak~\cite{hajek2017metamathematics}). Restricting the induction therefore corresponds to restrict\-ing the complexity of the finite sets of the model.
$\RCA_0$ is composed of Robinson arithmetic, together with the $\Delta^0_1$-comprehension scheme, and the $\Sigma^0_1$-induction scheme.

Models of second order arithmetic are of the form $\M = (M, S, +, \times, <, 0, 1)$, where $M$ is the first-order part, representing the integers in the model, and $S \subseteq \P(M)$ is the second-order part, representing the sets of integers. An \emph{$\omega$-model} is a structure $\M = (M, S, +, \times, <, 0, 1)$ whose first-order part~$M$ consists of the standard integers~$\omega$, together with the standard operations $+$, $\times$ and the natural order~$<$. Thus, $\omega$-models are fully specified by their second-order part~$S$. In particular, $\omega$-models of $\RCA_0$ admit a nice characterization in terms of Turing ideals. A \emph{Turing ideal}~$S \subseteq \P(\omega)$ is a non-empty collection of sets which is downward-closed under Turing reducibility, and closed under the effective join $X \oplus Y = \{ 2n : n \in X \} \cup \{ 2n+1 : n \in Y \}$. 
In particular, $\RCA_0$ admits a minimal $\omega$-model (with respect to inclusion) whose second-order part is exactly the collection of all computable sets.

Among models of second-order arithmetic, $\omega$-models are of particular interest due to their connection with classical computability theory, and most proofs of non-implications over~$\RCA_0$ consist of building a Turing ideal satisfying the left-hand side but not the right-hand side. This will also be the case in this article.

Since the beginning of reverse mathematics, many theorems from the core of mathematics have been studied, and some empirically observed structural phenomena emerged: there exist four main subsystems of second-order arithmetic, linearly ordered by logical strength, such that the vast majority of mathematics is either equivalent to one of these systems over~$\RCA_0$, or already provable over~$\RCA_0$. See Simpson~\cite{simpson2009subsystems} or Dzhafarov and Mummert~\cite{dzhafarov2022reverse} for an introduction to reverse mathematics and these main systems. Most of the theorems studied in reverse mathematics are statements of the form $(\forall X)[\phi(X) \rightarrow (\exists Y)\psi(X, Y)]$. Any such statement can be seen as a problem $\Psf$, whose instances are sets~$X$ such that $\phi(X)$, and whose solutions to~$X$ are any set~$Y$ such that~$\psi(X,Y)$ holds. These statements are called \emph{$\Pi^1_2$-problems}.

\subsection{Hierarchy of pigeonhole principles}

The structure property observed in reverse mathematics, and known as the \qt{Big Five}, admits a few counter-examples, mostly coming from Ramsey theory. Historically, the first natural theorem escaping this phenomenon is Ramsey's theorem for pairs ($\RT^2_2$), which was proven not to be even linearly ordered with the above-mentioned main subsystems. The computability-theoretic and proof-theoretic study of Ramsey's theorem for pairs raised many long-standing open questions, each of them requiring the development of new techniques and breakthroughs in computability theory~\cite{seetapun1995strength,cholak2001strength,liu2012rt22,monin2021srt}. The computability-theoretic study of combinatorial theorems from Ramsey theory is still currently the most active branch of research in reverse mathematics. See Hirschfeldt~\cite{hirschfeldt2015slicing} for an introduction to the reverse mathematics of combinatorial principles.

Beyond the clear combinatorial link between Ramsey's theorem for pairs and the pigeonhole principle, the formal computability-theoretic relation between these two principles was emphasized by the decomposition of $\RT^2_2$ into a cohesive\-ness principle ($\COH$) and the pigeonhole principle for $\Delta^0_2$ instances ($\Sub{\Delta^0_2}$).\footnote{This principle is also known as $\mathsf{D}^2_2$ in the reverse mathematical literature~\cite{cholak2001strength}.}

Given an infinite sequence of sets~$\vec{R} = R_0, R_1, \dots$, an infinite set~$C$ is \emph{$\vec{R}$-cohesive} if for every~$s \in \NN$, $C \subseteq^{*} R_s$ or $C \subseteq^{*} \overline{R}_s$, where~$\subseteq^{*}$ denotes inclusion up to finite changes. One can think of the sequence $\vec{R}$ as an infinite sequence of instances of $\RT^1_2$, and an $\vec{R}$-cohesive set as an infinite set which is almost a solution to every instance simultaneously.

\begin{statement}[Cohesiveness]
$\COH$ is the statement \qt{Every infinite sequence of sets has a cohesive set}.
\end{statement}

Informally, $\Sub{\Sigma^0_n}$ and $\Sub{\Delta^0_n}$ are the restriction of the pigeonhole principle for 2-colorings of $\Sigma^0_n$ and $\Delta^0_n$ sets, respectively. Technically, $\Sigma^0_n$ sets do not necessarily belong to models of weak arithmetic, and therefore are manipulated through formulas.

\begin{statement}Fix~$n \geq 1$.
\begin{itemize}
    \item $\Sub{\Sigma^0_n}$ is the statement \qt{For every~$\Sigma^0_n$ formula~$\varphi(x)$, there is an infinite set~$H \subseteq \NN$ such that either $\forall x \in H\ \varphi(x)$ or $\forall x \in H\ \neg \varphi(x)$.}
    \item $\Sub{\Delta^0_n}$ is the statement \qt{For every $\Sigma^0_n$ formula~$\varphi(x)$ and every~$\Pi^0_n$ formula~$\psi(x)$ such that $\forall x(\varphi(x) \leftrightarrow \psi(x))$ holds, there is an infinite set~$H \subseteq \NN$ such that either $\forall x \in H\ \varphi(x)$ or $\forall x \in H\ \neg \varphi(x)$.}
\end{itemize}
\end{statement}

Cholak, Jockusch and Slaman~\cite{cholak2001strength} and Mileti~\cite{mileti2004partition} proved the equivalence of~$\RT^2_2$ and $\COH + \Sub{\Delta^0_2}$ over~$\RCA_0 + \IDelta_2$, where $\mathsf{I}\Gamma$ denotes the induction scheme for $\Gamma$-predicates~\footnote{Over~$\RCA_0$, $\IDelta_n$ is equivalent to the better-known collection principle for $\Sigma^0_n$-formulas ($\BSig_n$) for~$n \geq 2$.}. Later, Chong, Lempp and Yang~\cite{chong2010role} got rid of the use of $\IDelta_2$, yielding the following equivalence:

\begin{theorem}[\cite{cholak2001strength,mileti2004partition,chong2010role}]
$\RCA_0 \vdash \RT^2_2 \leftrightarrow \COH + \Sub{\Delta^0_2}$.
\end{theorem}

The cohesiveness principle is very weak from a reverse mathematical viewpoint. It has the same first-order consequences as $\RCA_0$~\cite{cholak2001strength} and preserves all the first-jump control properties studied in reverse mathematics, such as cone avoidance, PA avoidance, DNC avoidance, among others. From a more abstract perspective, $\COH$ is equivalent over~$\RCA_0 + \IDelta_2$ to the statement \qt{Every $\Delta^0_2$ infinite binary tree admits an infinite $\Delta^0_2$ path}\cite{belanger2022conservation} and Towsner~\cite{townser2015maximum} proved that the $\Delta^0_2$-sets are indistinguishable from arbitrary sets from the viewpoint of~$\RCA_0$. Thus, a statement of the existence of a $\Delta^0_2$-approximation of a set does not add any proof-theoretic strength to $\RCA_0$. It follows from these considerations that the whole reverse mathematical complexity of~$\RT^2_2$ is already contained in~$\Sub{\Delta^0_2}$.

More recently, Benham et al.~\cite{benham2024ginsburg} revealed a surprising connection between a theorem of topology and $\RT^1_2$ for $\Sigma^0_2$ sets. The Ginsburg-Sands theorem~\cite{ginsburg1979minimal} states that every infinite topological space has an infinite subspace homeomorphic to exactly one of the following five topologies on $\NN$: indiscrete, discrete, initial segment, final segment, and cofinite. When restricted to $T_1$-spaces, it states that every infinite topological space has an infinite subspace homeomorphic to either the discrete or the cofinite topology on~$\NN$. Benham et al.~\cite{benham2024ginsburg} proved that the Ginsburg-Sands theorem for $T_1$ spaces is equivalent over~$\RCA_0$ to $\COH + \Sub{\Sigma^0_2}$. The higher levels of the pigeonhole hierarchy are related to Ramsey-type hierarchies such as the rainbow Ramsey and free set theorems~\cite{csima2009strength,wang2014some} in terms of computability-theoretic techniques used to analyze them. In particular, all three hierarchies admit strong cone avoidance~\cite{dzhafarov2009ramsey,wang2014some} and for every $n \geq 2$, there is a $\Delta^0_n$ instance with no $\Sigma^0_{n-1}$ solution~\cite{jockuschRamseyTheoremRecursion1972,cholakFreeSetsReverse2001,csima2009strength}.
The strictness of the rainbow Ramsey and free set hierarchies remains an open question.

\subsection{Strictness of the hierarchy}

The main contributions of this article are the strictness of the hierarchy of pigeonhole principles over~$\RCA_0$,
and conservation theorems at various levels of the induction hierarchy. Separating a problem~$\Psf$ from another problem~$\Qsf$ over~$\RCA_0$ usually consists in finding an invariant computability-theoretic weakness property such that every weak instance of~$\Psf$ admits a weak solution, while there is a weak instance of~$\Qsf$ with no weak solution. Then, a simple iterated construction yields an $\omega$-model of~$\RCA_0 + \Psf$ which is not a model of~$\Qsf$.

The most natural weakness properties are the levels of the arithmetic hierarchy, but these are not invariant, as if a set~$X$ is $\Delta^0_2(Y)$ and $Y$ is $\Delta^0_2$, then $X$ is not in general $\Delta^0_2$. Lowness properties are a strengthening of the levels of the arithmetic hierarchy providing the desired invariant. A set~$X$ is of \emph{low} degree if $X' \leq_T \emptyset'$. More generally, a set~$X$ is of \emph{low${}_n$} degree if $X^{(n)} \leq_T \emptyset^{(n)}$, where $X^{(n)}$ denotes the $n$-fold Turing jump of~$X$. If a set~$X$ is low${}_n$ over~$Y$ and $Y$ is low${}_n$, then $X$ is low${}_n$, as $X^{(n)} \leq_T Y^{(n)} \leq_T \emptyset^{(n)}$. The low${}_n$ degrees form a subclass of the $\Delta^0_{n+1}$ degrees.

\begin{definition}
Fix~$n \geq 1$. A problem $\Psf$ \emph{admits a low${}_n$ basis} if every computable instance of~$\Psf$ admits a solution of low${}_n$ degree. 
\end{definition}

Building a solution~$G$ of low${}_n$ degree consists of deciding the $\Sigma^0_n(G)$ properties of the generic object through a $\emptyset^{(n)}$-computable process. This technique is called \emph{$n$th-jump control}. Second and higher jump controls are usually significantly more complicated than first-jump control, as the forcing relation involves density properties. Thankfully, in many cases, $(n+1)$th-jump control can be obtained by an $n$th-jump control with PA degrees. A set $P$ is of \emph{PA degree} over~$X$ if every infinite $X$-computable binary tree admits an infinite $P$-computable path.

\begin{definition}
Fix~$n \geq 1$. A problem $\Psf$ \emph{admits a weakly low${}_n$ basis} if to every computable instance of~$\Psf$, and for every set~$Q$ of PA degree over~$\emptyset^{(n)}$, there is a solution $Y$ such that $Y^{(n)} \leq_T Q$.
\end{definition}

Clearly, if a problem admits a low${}_n$ basis, then it admits a weakly low${}_n$ basis.
On the other hand, given $n \geq 1$, by the low basis theorem for $\Pi^0_1$-classes~\cite{jockusch1972classes}, there is a set~$Q$ of PA degree over~$\emptyset^{(n)}$ such that $Q' \leq_T \emptyset^{(n+1)}$. Thus, if $Y$ is a solution such that $Y^{(n)} \leq_T Q$, then $Y^{(n+1)} \leq_T Q' \leq_T \emptyset^{(n+1)}$, hence $Y$ is of low${}_{n+1}$ degree. It follows that if a problem admits a weakly low${}_n$ basis, then it admits a low${}_{n+1}$ basis. We therefore have the following implications
\begin{center}
low basis $\rightarrow$ weakly low basis $\rightarrow$ low${}_2$ basis $\rightarrow$ weakly low${}_2$ basis $\rightarrow \dots$ 
\end{center}

Due to its strong connections with Ramsey's theorem for pairs, the statement $\Sub{\Delta^0_2}$ was thoroughly studied in reverse mathematics. Low${}_n$ and weakly low${}_n$ basis theorems for $\Sub{\Delta^0_2}$ were studied in particular by Cholak, Jockusch and Slaman~\cite{cholak2001strength} and Downey, Hirschfeldt, Lempp and Solomon~\cite{downey2001delta2}. The problem $\Sub{\Delta^0_n}$ for $n \geq 2$ was studied by Monin and Patey~\cite{monin2021weakness}, and the problem $\Sub{\Sigma^0_2}$ was more recently studied by Benham et al.~\cite{benham2024ginsburg}. The following table summarizes the known literature on the subject.

\begin{figure}[htbp]
\begin{center}
\bgroup
\def\arraystretch{1.5}%
\begin{tabular}{|c|c|c|c|} \hline
Problem & Non-basis & Previous basis & New basis\\ \hline
$\Sub{\Sigma^0_{n+1}}$ & low${}_n$~\cite{downey2001delta2} & weakly low${}_{n+1}$~\cite{monin2021weakness} & weakly low${}_n$ \\
$\Sub{\Delta^0_{n+1}}$ & low${}_n$~\cite{downey2001delta2} & weakly low${}_n$~\cite{monin2021weakness} & \\
$\Sub{\Sigma^0_2}$ & low~\cite{downey2001delta2} & low${}_2$~\cite{benham2024ginsburg} & weakly low \\
$\Sub{\Delta^0_2}$ & low~\cite{downey2001delta2} & weakly low~\cite{cholak2001strength} & \\
$\Sub{\Sigma^0_1}$ & & computable solutions & \\ \hline
\end{tabular}
\egroup
\end{center}
\caption{Summary table of the previously known bounds and the new bounds in terms of low basis theorems.
The new basis theorem proven in this article completes the table with tight bounds.}
\end{figure}

Our first main theorem is a weakly low${}_n$ basis theorem for $\Sub{\Sigma^0_{n+1}}$, disproving a conjecture of Benham et al.~\cite{benham2024ginsburg}.

\begin{maintheorem}\label{thm:main-weakly-low-basis}
Fix~$n \geq 1$. For every $\Sigma^0_{n+1}$ set~$A$ and every set~$Q$ of PA degree over~$\emptyset^{(n)}$, there is an infinite set~$H \subseteq A$ or $H \subseteq \overline{A}$ such that $H^{(n)} \leq_T Q$.
\end{maintheorem}

It follows that for every~$n \geq 1$, there is an $\omega$-model of $\Sub{\Sigma^0_n}$ with only low${}_n$ sets.
By Downey et al.~\cite{downey2001delta2}, there is a computable instance of $\Sub{\Delta^0_{n+1}}$ with no low${}_n$ solution. Thus, $\RCA_0 + \Sub{\Sigma^0_n} \nvdash \Sub{\Delta^0_{n+1}}$ for every~\mbox{$n \geq 1$}.

Separating $\Sub{\Delta^0_n}$ from $\Sub{\Sigma^0_n}$ is more complicated, as these princi\-ples satisfy the same lowness basis. A function $f : \NN \to \NN$ \emph{dominates} $g : \NN \to \NN$ if $\forall x(f(x) \geq g(x))$. A function $f : \NN \to \NN$ is \emph{$X$-hyperimmune} if it is not dominated by any $X$-computable function. If $X$ is computable, then we simply say that $f$ is hyperimmune. Benham et al.~\cite{benham2024ginsburg} separated $\Sub{\Delta^0_2}$ from $\Sub{\Sigma^0_2}$ by designing a very elaborate invariant in terms of preservation of hyperimmunities and $\emptyset'$-hyperimmunities simultaneously. We simplify their argument and generalize it to higher levels of the pigeonhole hierarchy.

\begin{definition}
Fix a multiset~$I \subseteq \NN$. A problem $\Psf$ \emph{preserves hyperimmunity at levels~$I$} if for every computable instance~$X$ of~$\Psf$ and every family of functions $(f_n)_{n \in I}$ such that $f_n$ is $\emptyset^{(n)}$-hyperimmune, there is a solution~$Y$ to~$X$ such that for every~$n \in I$, $f_n$ is $Y^{(n)}$-hyperimmune.
\end{definition}

Benham et al.~\cite{benham2024ginsburg} essentially proved that $\Sub{\Sigma^0_2}$ does not preserve hyper\-immunity at levels~$\{0,1\}$ while $\Sub{\Delta^0_2}$ does. A direct relativization of their proof yields that $\Sub{\Sigma^0_n}$ does not preserve hyperimmunity at levels~$\{n-2,n-1\}$ for every~$n \geq 2$, but the positive preservation theorem requires a non-trivial generalization. Our second main theorem generalizes and simplifies the proof of Benham et al.~\cite{benham2024ginsburg} by stating that $\Sub{\Delta^0_n}$ preserves hyperimmunity at levels~$\{n-2,n-1\}$ for every~$n \geq 2$.

\begin{maintheorem}\label[maintheorem]{thm:main-preservation-hyps}
Fix~$n \geq 2$. For every $\Delta^0_n$ set~$A$, every $\emptyset^{(n-2)}$-hyperimmune function $f : \NN \to \NN$ and every $\emptyset^{(n-1)}$-hyperimmune function $g : \NN \to \NN$, there is an infinite set~$H \subseteq A$ or $H \subseteq \overline{A}$ such that $f$ is $H^{(n-2)}$-hyperimmune and $g$ is $H^{(n-1)}$-hyperimmune.
\end{maintheorem}

It follows that $\RCA_0 + \Sub{\Delta^0_n} \nvdash \Sub{\Sigma^0_n}$ for every~$n \geq 2$. Thus, the hierarchy of pigeonhole principles is strict, answering a question of Benham et al.~\cite{benham2024ginsburg}. The separation being witnessed by $\omega$-models, these separations also hold when adding the full induction scheme to~$\RCA_0$.

\begin{maintheorem}
Over~$\RCA_0$, the following hierarchy is strict
$$\Sub{\Sigma^0_1} < \Sub{\Delta^0_2} < \Sub{\Sigma^0_2} < \Sub{\Delta^0_3} < \dots $$
\end{maintheorem}

\subsection{Conservation theorems}

The \emph{first-order part} of a second-order theory~$T$ is the set of its first-order theorems.
Characterizing the first-order part of ordinary second-order theorems is an important subject of study in reverse mathematics and is closely related to Hilbert's finitistic reductionism program~\cite{simpson1988partial}.

The first-order parts of the main subsystems studied in reverse mathematics are well-understood.
In particular, the first-order part of $\RCA_0 + \ISig_n$ for~$n \geq 1$ corresponds to
$\Robinson + \mathsf{I}\Sigma_n$, where~$\mathsf{I}\Sigma_n$ denotes the induction scheme restricted to $\Sigma_n$-formulas, that is, with no second-order parameter allowed~\cite{friedmanSystemsSecondOrder1976}. 

A good way to calibrate the first-order part of second-order theories is to reduce it to existing benchmark theories through conservation. Given a family of sentences $\Gamma$, a theory $T_2$ is \emph{$\Gamma$-conservative} over~$T_1$ if every $\Gamma$-sentence provable over~$T_2$ is provable over~$T_1$. A formula is \emph{$\Pi^1_1$} if it starts with a universal second-order quantifier, followed by an arithmetic formula. If a theory $T_2$ is $\Pi^1_1$-conservative over~$T_1$, then the first-order part of~$T_2$ is included in the first-order part of~$T_1$.

Characterizing the first-order part of $\Sub{\Delta^0_2}$ is one of the most important questions in reverse mathematics. Chong, Lempp and Yang~\cite{chong2010role} proved that  $\RCA_0 \vdash \Sub{\Delta^0_2} \to \IDelta_2$. On the other hand, Cholak, Jockusch and Slaman~\cite{cholak2001strength} proved that $\RCA_0 + \ISig_2 + \Sub{\Delta^0_2}$ is $\Pi^1_1$-conservative over~$\RCA_0 + \ISig_2$ and Chong, Slaman and Yang~\cite{chong2017inductive} proved that $\RCA_0 + \Sub{\Delta^0_2} \nvdash \ISig_2$. Thus, the first-order part of~$\RCA_0 + \Sub{\Delta^0_2}$ is above $\Robinson + \mathsf{I}\Delta_2$ and strictly below $\Robinson + \mathsf{I}\Sigma_2$.

Benham et al.~\cite{benham2024ginsburg} asked whether $\RCA_0 + \ISig_2 + \Sub{\Sigma^0_2}$ is $\Pi^1_1$-conservative over~$\RCA_0 + \ISig_2$. We answer positively at every level of the hierarchy.

\begin{maintheorem}\label{thm:main-conservation}
Fix~$n \geq 2$. Then $\RCA_0 + \ISig_n + \Sub{\Sigma^0_n}$ is $\Pi^1_1$-conservative over~$\RCA_0 + \ISig_n$.
\end{maintheorem}

It follows that the Ginsburg-Sands theorem for $T_1$-spaces is $\Pi^1_1$-conservative over~$\RCA_0 + \ISig_2$.
On the other hand, it remains open whether $\RCA_0 + \Sub{\Sigma^0_n} \vdash \ISig_2$ for any~$n \geq 2$
or even whether~$\RCA_0 + \IDelta_n + \Sub{\Sigma^0_n}$ is $\Pi^1_1$-conservative over~$\RCA_0 + \IDelta_n$.


\subsection{Notation}

We assume the reader is familiar with notations from classical computability theory. See Cooper~\cite{cooper2004computability} or Soare~\cite{soare2016turing} for a reference. In particular, $\leq_T$ denotes Turing reducibility, and $\Phi_0, \Phi_1, \dots$ is a fixed enumeration of all the Turing functionals.
We write $\Phi_e^A(x)\converge$ to say that the $e$th program with oracle~$A$ halts on input~$x$, and $\Phi_e^A(x)\diverge$ otherwise.

\emph{Integers}. We let $\NN$ denote the set of non-negative integers. When working in models of weak arithmetic, we distinguish the formal set~$\NN$ representing the integers in the model, from the set~$\omega$ of standard integers, which corresponds to the integers in the meta-theory. In particular, $\omega$ is always a (proper or not) initial segment of~$\NN$.

\emph{Binary strings}. We let $2^{<\NN}$ denote the set of all finite binary strings, and $2^\NN$ denote the class of all infinite binary sequences. Note that $2^\NN$ is in bijection with $\P(\NN)$ and both are usually identified.
Finite binary strings are written with greek letters~$\sigma, \tau, \mu, \dots$. The length of a binary string $\sigma$ is written~$|\sigma|$, and the concatenation of two binary strings $\sigma$ and $\tau$ is written $\sigma \cdot \tau$. Given an infinite binary sequence~$X \in 2^\NN$ and $n \in \NN$, we write $X \uh_n$ for its initial segment of length~$n$. A string~$\sigma$ is a prefix of a string~$\tau$, written $\sigma \preceq \tau$, if there is some~$\mu$ such that $\sigma \cdot \mu = \tau$.

\emph{Finite sets}. Finite binary strings, seen as finite characteristic functions, are often identified with finite sets, that is, $\sigma$ is identified with $\{ n < |\sigma| : \sigma(n) = 1 \}$. Based on this correspondence, we extend the set-theoretic notations to binary strings, and let for example $\sigma \cup \rho$ denote the binary string of length~$\max(|\sigma|, |\rho|)$, and such that $(\sigma \cup \rho)(n) = 1$ if and only if $\sigma(n) = 1$ or $\rho(n) = 1$.
In particular, one shall distinguish the cardinality $\card \sigma = \card \{ n < |\sigma| : \sigma(n) = 1 \}$ of a string seen as a set, from its length~$|\sigma|$.

\emph{Trees}. A \emph{binary tree} is a set~$T \subseteq 2^{<\NN}$ closed under prefix. A \emph{path} through a binary tree~$T$ is an infinite binary sequence~$X \in 2^\NN$ such that $X \uh_n \in T$ for every~$n \in \NN$. We let $[T]$ denote the class of all paths through~$T$. Every closed class in the Cantor space can be written of the form $[T]$ for a binary tree~$T \subseteq 2^{<\NN}$. A class~$\C \subseteq 2^\NN$ is \emph{$\Pi^0_1(X)$} if it is of the form $[T]$ for an $X$-co-c.e.\ tree~$T$, or equivalently an $X$-computable tree~$T$, or even a primitive $X$-recursive tree~$T$. A set $P$ is of \emph{PA degree} over~$X$ if for every~$X$-computable infinite binary tree~$T \subseteq 2^{<\NN}$, $P$ computes an infinite path. Note that if $P$ is of PA degree over~$X$, then $P$ computes~$X$.

\emph{Formulas}.
A formula is $\Sigma^0_n(X_1, \dots, X_k)$ (resp.\ $\Pi^0_n(X_1, \dots, X_k)$) if it is $\Sigma^0_n$ (resp.\ $\Pi^0_n$) with parameters $X_1, \dots, X_k$. Here, $X_1, \dots, X_k$ are either free set variables, or actual set parameters, in which case we implicitly consider an augmented language.
Given a family of sets~$\M \subseteq \P(\NN)$, a formula is $\Sigma^0_n(\M)$ (resp.\ $\Pi^0_n(\M)$) if it is $\Sigma^0_n(X_1, \dots, X_k)$ (resp.\ $\Pi^0_n(X_1, \dots, X_k)$) for some~$X_1, \dots, X_k \in \M$.

\emph{Ideals}.
Recall that a non-empty family of sets $\M \subseteq \P(\NN)$ is a \emph{Turing ideal} if for every~$X \in \M$ and $Y \leq_T X$, then $Y \in \M$, and for every~$X, Y \in \M$, $X \oplus Y \in \M$. A Turing ideal~$\M$ is a \emph{Scott ideal} if furthermore, for every infinite binary tree~$T \in \M$, there is an infinite path~$P \in [T] \cap \M$.
A countable Turing ideal~$\M = \{Z_0, Z_1, \dots\}$ is \emph{coded} by a set~$M \subseteq \NN$ if $M = \bigoplus_n Z_n = \{ \langle x, n \rangle : x \in Z_n \}$. Here, $\langle \cdot, \cdot \rangle : \NN^2 \to \NN$ denotes the usual Cantor bijection. An \emph{$M$-index} of an element~$X \in \M$ is an integer~$n$ such that $X = Z_n$. A \emph{Scott code} of a Scott ideal~$\M$ is a code~$M$ such that the operations $(m, n) \mapsto p$ such that $Z_p = Z_m \oplus Z_n$, and $(e, n) \mapsto p$ such that $Z_p$ is a completion of~$\Phi_e^{Z_n}$, are computable. A collection of sets~$\M$ is \emph{topped} by a set~$X$ if it is of the form $\{ Z \in \cs : Z \leq_T X \}$. Note that every topped collection is a Turing ideal, and that no Scott ideal is topped. 

\emph{$\omega$-structure}. An $\omega$-structure $\M = (\omega, S)$ is fully specified by its second-order part~$S$. Thus, we identify both notions. In particular, we say that $\M$ is topped if so is $S$. As mentioned, $\M \models \RCA_0$ if and only if $S$ is a Turing ideal.

\emph{Mathias forcing}. A \emph{Mathias condition} is an ordered pair $(\sigma, X)$, where $\sigma$ is a finite binary string, and $X \subseteq \NN$ is an infinite set with $\min X > |\sigma|$. A condition $(\tau, Y)$ \emph{extends} another condition $(\sigma, X)$ if $\sigma \preceq \tau$, $Y \subseteq X$, and $\tau \setminus \sigma \subseteq X$. The set $X$ of a Mathias condition $(\sigma, X)$ is considered as a \emph{reservoir} of elements which are allowed to be later added to the initial segment~$\sigma$. By abuse of notation, when~$\min X \leq |\sigma|$, we shall write $(\sigma, X)$ to denote the Mathias condition $(\sigma, X \setminus \{0, \dots, |\sigma|\})$.

\subsection{Organization of the paper}

In \Cref{sect:iterated-jump-control}, we present the big picture of the iterated-jump control techniques used to achieve the main theorems. In particular, we emphasize the role of the so-called \qt{forcing question} in the preservation of computability-theoretic weakness properties.

PA degrees play an essential role in computability theory, and are often involved as intermediary objects to obtain a good iterated jump control construc\-tion. Let \emph{weak K\"onig's lemma} ($\WKL$) be the problem whose instances are infinite binary trees and whose solutions are paths. In \Cref{sect:forcing-trees}, we prove that $\WKL$ preserves hyperimmunity at every level of the arithmetic hierarchy simultaneously, that is, if the considered functions are $\emptyset^{(n)}$-hyperimmune, they will remain $G^{(n)}$-hyperimmune, where $G$ is the generic object. This both serves as a gentle example to iterated-jump control and a preliminary construction necessary to prove our main theorems.

Solutions to combinatorial theorems from Ramsey theory are often constructed using variants of Mathias forcing. However, Mathias forcing does not behave well with respect to iterated-jump control. In \Cref{sect:largeness}, we introduce the fundamental concepts of largeness, partition regularity, minimal and cohesive class, which enable to define a refinement of Mathias forcing with the appropriate iterated-jump control. Then, in \Cref{core-forcing}, we introduce the common combinatorial core of all the notions of forcing used in this article.

In \Cref{sect:iterated-lowness}, we define two notions of forcing, namely, main forcing and witness forcing, to build low${}_{n+1}$ and weakly low${}_n$ solutions to $\Sub{\Sigma^0_{n+1}}$, and use it to separate $\Sub{\Sigma^0_n}$ from $\Sub{\Delta^0_{n+1}}$ over~$\omega$-models. Then, in \Cref{sect:preservation-hyps}, we introduce a disjunctive notion of forcing to preserve multiple hyperimmunities simultaneously, and use this framework to separate $\Sub{\Delta^0_{n+1}}$ from $\Sub{\Sigma^0_{n+1}}$ over $\omega$-models. We formalize the constructions of \Cref{sect:iterated-lowness} over weak models of arithmetic in \Cref{sect:conservation} to prove that $\RCA_0 + \ISig_{n+1} + \Sub{\Sigma^0_{n+1}}$ is $\Pi^1_1$-conservative over $\RCA_0 + \ISig_{n+1}$.

Last, in \Cref{sect:open-questions}, we state some remaining open questions and research directions.

\section{Iterated jump control and forcing question}\label{sect:iterated-jump-control}

The main theorems of this article are proven by effective forcing, with an iterated jump control.
This technique consists of making the constructed set inherit computability-theoretic weaknesses of the ground model by translating arithmetical properties of the generic set into absolute arithmetical formulas of the same complexity.

In what follows, we shall consider an arbitrary notion of forcing $(\PP, \leq)$, together with an \emph{interpretation function} $[\cdot] : \PP \to \P(\cs)$ such that if $d \leq c$, then $[d] \subseteq [c]$.
Intuitively, a condition~$c$ is an approximation of the constructed set, and its interpretation is the class of all candidate sets which satisfy the approximation. If~$d \leq c$, then the approximation~$d$ is more precise than $c$, thus $[d] \subseteq [c]$. In all the notions of forcing we shall consider, the interpretation will be a closed class in the Cantor space, and for every sufficiently generic filter~$\F$, the class $\bigcap_{c \in \F} [c]$ will be a singleton~$\{G_\F\}$. A condition~$c$ \emph{forces} a formula $\varphi(G)$ if $\varphi(G_\F)$ holds for every sufficiently generic filter~$\F$ containing~$c$.

The computability-theoretic weaknesses of the generic set are closely related to the existence of a so-called forcing question with a good definitional complexity.

\begin{definition}\label[definition]{def:abstract-forcing-question}
Let $\Gamma$ be a family of formulas. A \emph{forcing question} for $\Gamma$-formulas is a relation $\qvdash \subseteq \PP \times \Gamma$ such that for every~$c \in \PP$ and $\varphi(G) \in \Gamma$,
\begin{enumerate}
    \item If $c \qvdash \varphi(G)$, then there is an extension~$d \leq c$ forcing~$\varphi(G)$ ;
    \item If $c \nqvdash \varphi(G)$, then there is an extension~$d \leq c$ forcing~$\neg \varphi(G)$.
\end{enumerate}
\end{definition}

Given a formula~$\varphi(G)$, the set $\PP$ can be divided into three categories: the conditions forcing $\varphi(G)$, the conditions forcing~$\neg \varphi(G)$, and the conditions forcing neither of those. 
A forcing question can be thought of as a dividing line within the third category. There are therefore two canonical relations satisfying \Cref{def:abstract-forcing-question}:
\begin{itemize}
    \item $c \qvdash_0 \varphi(G)$ if there is some~$d \leq c$ forcing $\varphi(G)$. This relation will give the same answers to the third category and the first one.
    \item $c \qvdash_1 \varphi(G)$ if there is no $d \leq c$ forcing $\neg \varphi(G)$. This relation will give the same answers to the third category and the second one.
\end{itemize}
In some cases, however, there exist intermediary forcing questions with a better definitional complexity.

Note that a forcing question for $\Sigma^0_n$ formulas induces a forcing question for $\Pi^0_n$ formulas by negating the relation, thus we shall only consider forcing questions for~$\Sigma^0_n$ formulas. The notion of forcing question was introduced by Monin and Patey~\cite[Section 2]{monin201pigeons} who proved two abstract theorems. We recall them for the sake of completeness.

\begin{definition}
Let $\Gamma$ be a family of formulas. A forcing question is \emph{$\Gamma$-preserving} if for every~$c \in \PP$
and every formula $\varphi(G, x) \in \Gamma$, the relation $c \qvdash \varphi(G, x)$ is in~$\Gamma$ uniformly in~$x$, that is, the predicate $P(x) := c \qvdash \varphi(G, x)$ is a $\Gamma$-formula.
\end{definition}

The first abstract theorem concerns the preservation of the arithmetic hierarchy. It is used to prove cone avoidance, or its iterated versions.

\begin{theorem}[\cite{monin201pigeons}]\label{thm:abstract-preserving-definitions}
Let $(\PP, \leq)$ be a notion of forcing with a $\Sigma^0_n$-preserving forcing question.
For every non-$\Sigma^0_n$ set~$C$ and every sufficiently generic filter~$\F$, $C$ is not $\Sigma^0_n(G_\F)$.
\end{theorem}
\begin{proof}
For every~$e \in \NN$, let $\D_e \subseteq \PP$ be the set of all conditions forcing $W_e^{G^{(n-1)}} \neq C$.
We claim that~$\D_e$ is dense. Indeed, given $c \in \PP$, consider the following set:
$$
U = \{ x \in \NN : c \qvdash x \in W_e^{G^{(n-1)}} \}
$$
Since the forcing question is $\Sigma^0_n$-preserving, the set~$U$ is $\Sigma^0_n$, thus $U \neq C$.
Let $x \in U \Delta C = (U \setminus C) \cup (C \setminus U)$. Suppose first that $x \in U \setminus C$. By Property~1 of the forcing question, there is an extension~$d \leq c$ forcing $x \in W_e^{G^{(n-1)}}$.
Suppose now that $x \in C \setminus U$. By Property~2 of the forcing question, there is an extension~$d \leq c$ forcing $x \not\in W_e^{G^{(n-1)}}$.
In both cases, the extension $d \leq c$ forces $W_e^{G^{(n-1)}} \neq C$, so the set $\D_e$ is dense. This proves our claim. Thus, for every sufficiently generic filter~$\F$, $\F$ is $\{\D_e : e \in \NN\}$-generic, hence $C$ is not $\Sigma^0_n(G_\F)$.
\end{proof}

Many forcing questions, when answering positively a $\Sigma^0_n$ question, can actually find a finite set of witnesses for the outermost existential quantifier. This can be seen as a form of compactness.

\begin{definition}
A forcing question is \emph{$\Sigma^0_n$-compact} if for every $c \in \PP$ and every $\Sigma^0_n$ formula $\varphi(G, x)$, if $c \qvdash \exists x \varphi(G, x)$, then there is some~$k \in \NN$ such that $c \qvdash (\exists x < k)\varphi(G, x)$.
\end{definition}

The existence of a $\Sigma^0_n$-compact forcing question is closely related to the ability to compute fast-growing functions. Recall that a function~$f : \NN \to \NN$ is \emph{$A$-hyperimmune} if it is not dominated by any $A$-computable function.

\begin{theorem}[\cite{monin201pigeons}]\label{thm:abstract-compact-hyperimmunity}
Let $(\PP, \leq)$ be a notion of forcing with a $\Sigma^0_n$-compact, $\Sigma^0_n$-preserving forcing question.
For every $\emptyset^{(n-1)}$-hyperimmune function~$f : \NN \to \NN$ and every sufficiently generic filter~$\F$, $f$ is $G_\F^{(n-1)}$-hyperimmune.
\end{theorem}
\begin{proof}
For every~$e \in \NN$, let $\D_e \subseteq \PP$ be the set of all conditions forcing $\Phi_e^{G^{(n-1)}}$ not to dominate~$f$. More precisely, $\D_e$ is the set of all conditions forcing either $\Phi_e^{G^{(n-1)}}$ to be partial, or $\Phi_e^{G^{(n-1)}}(x)\converge < f(x)$ for some~$x \in \NN$. We claim that~$\D_e$ is dense.
Suppose first that $c \nqvdash \Phi_e^{G^{(n-1)}}(x)\converge$ for some~$x \in \NN$. Then by Property~2 of the forcing question, there is an extension~$d \leq c$ forcing $\Phi_e^{G^{(n-1)}}(x)\diverge$, hence to be partial.
Suppose now that for every~$x \in \NN$, $c \qvdash \exists v \Phi_e^{G^{(n-1)}}(x)\converge = v$. By $\Sigma^0_n$-compactness of the forcing question, for every~$x \in \NN$, there is some bound $k_x \in \NN$ such that $c \qvdash (\exists v < k_x)\Phi_e^{G^{(n-1)}}(x)\converge = v$. Let~$h : \NN \to \NN$ be the function which on input~$x$, looks for some~$k_x \in \NN$ such that $c \qvdash (\exists v < k_x)\Phi_e^{G^{(n-1)}}(x)\converge = v$, and outputs~$k_x$. Such a function is total by hypothesis, and $\emptyset^{(n-1)}$-computable by $\Sigma^0_n$-preservation of the forcing question. Since~$f$ is $\emptyset^{(n-1)}$-hyperimmune, $h(x) < f(x)$ for some~$x \in \NN$. By Property~1 of the forcing question, there is an extension~$d \leq c$ forcing $(\exists v < k_x)\Phi_e^{G^{(n-1)}}(x)\converge = v$. Since~$f(x) \geq k_x$, $d$ forces $\Phi_e^{G^{(n-1)}}(x)\converge < f(x)$.
This proves our claim. Thus, for every sufficiently generic filter~$\F$, $\F$ is $\{\D_e : e \in \NN\}$-generic, hence $f$ is $G_\F^{(n-1)}$-hyperimmune.
\end{proof}

Additional structural properties on the forcing question, such as the ability to find simultaneous answers to independent questions, yield PA or DNC avoidance, as in Monin and Patey~\cite{monin2021srt}.

The forcing question plays an important role in conservation theorems as well. Indeed, given a $\Pi^1_2$-problem~$\Psf$, proving that $\RCA_0 + \ISig_n + \Psf$ is $\Pi^1_1$-conservative over~$\RCA_0 + \ISig_n$ consists in starting with a countable model $\M = (M, S)$ of $\RCA_0 + \ISig_n$, and given an instance~$X \in S$ of $\Psf$, constructing a solution~$G \subseteq M$ such that $\M \cup \{G\} \models \ISig_n$, where~$\M \cup \{G\} = (M, S \cup \{G\})$.

\begin{definition}
Given a notion of forcing $(\PP, \leq)$ and some~$n \in \NN$, a forcing question is \emph{$(\Sigma^0_n, \Pi^0_n)$-merging}
if for every~$c \in \PP$ and every pair of $\Sigma^0_n$ formulas $\varphi(G), \psi(G)$ such that $c \qvdash \varphi(G)$ 
but $c \nqvdash \psi(G)$, then there is an extension~$d \leq c$ forcing $\varphi(G) \wedge \neg \psi(G)$.
\end{definition}

The existence of a $(\Sigma^0_n, \Pi^0_n)$-merging forcing question enables to preserve $\Sigma^0_n$-induction.

\begin{theorem}\label{thm:abstract-preservation-isig1}
Let~$\M = (M, S) \models \ISig_n$ be a countable model and let~$(\PP, \leq)$ be a notion of forcing with a $\Sigma^0_n$-preserving $(\Sigma^0_n,\Pi^0_n)$-merging forcing question. For every sufficiently generic filter~$\F$, $\M \cup \{G_\F\} \models \ISig_n$. 
\end{theorem}
\begin{proof}
For every $\Sigma^0_n$-formula~$\varphi(x, G)$, let $\D_\varphi \subseteq \PP$ be the set of all conditions forcing either $\forall x \varphi(x, G)$,
or $\neg \varphi(0, G)$, or $\varphi(a-1,G) \wedge \neg \varphi(a, G)$, for some~$a \in M$ with $a > 0$. We claim that~$\D_\varphi$ is dense.

Let~$c \in \PP$ be a condition. If $c$ forces $\forall x \varphi(x, G)$, then we are done. Otherwise, there is an extension~$d \leq c$ and some~$b \in M$ such that $d$ forces $\neg \varphi(b, G)$.
Let $A = \{ x \in M : d \qvdash \varphi(x, G) \}$.
Since the forcing question is $\Sigma^0_n$-preserving, the set $A$ is $\Sigma^0_n(\M)$.
Moreover, $d$ forces $\neg\varphi(b, G)$, so by definition of the forcing question, $d \nqvdash \varphi(b, G)$, hence $b \not \in A$.
Since $\M \models \ISig_n$, and $A \neq M$, either $0 \not \in A$, or there is some~$a \in M$ with $a > 0$ such that $a \not \in A$, and $a-1 \in A$.
In the first case, by definition of the forcing question, there is an extension of~$d$ forcing $\neg \varphi(0, G)$.
Otherwise, since the forcing question is $(\Sigma^0_n, \Pi^0_n)$-merging, there is an extension of~$d$
forcing $\varphi(a-1, G) \wedge \neg \varphi(a, G)$. This proves our claim.
Thus, for every sufficiently generic filter~$\F$, $\F$ is $\{\D_\varphi \}$-generic, hence $\M \cup \{G_\F\} \models \ISig_n$.
\end{proof}

The most natural way to define a forcing question consists in defining $c \qvdash \varphi(G)$ to hold if there exists some~$d \leq c$ forcing~$\varphi(G)$. There exists an inductive syntactic definition of the forcing relation, and when the partial order is computable, this definition yields a $\Sigma^0_n$-preserving forcing question. However, in most cases, the partial order is not computable, and one uses a custom $\Gamma$-preserving forcing relation to obtain a $\Gamma$-preserving forcing question.

\begin{definition}\label{def:forcing-relation}
Let $\Gamma$ be a family of formulas. A \emph{forcing relation} for $\Gamma$ is a relation $\Vdash \subseteq \PP \times \Gamma$ such that for every~$c \in \PP$ and $\varphi(G) \in \Gamma$,
\begin{enumerate}
    \item If $c \Vdash \varphi(G)$, then $c$ forces $\varphi(G)$ ;
    \item The set of conditions~$c$ such that $c \Vdash \varphi(G)$ or $c \Vdash \neg \varphi(G)$ is dense ;
    \item If $c \Vdash \varphi(G)$ and $d \leq c$ then $d \Vdash \varphi(G)$.
\end{enumerate}
\end{definition}

The first property, known as \qt{forcing implies truth}, states the soundness of the relation, while the second property states is completeness. The definition is equivalent to the statement \qt{for every sufficiently generic filter~$\F$ and every~$\varphi(G) \in \Gamma$, $\varphi(G_\F)$ holds if and only if $c \Vdash \varphi(G)$ for some condition~$c \in \F$.}

\section{Forcing with trees}\label{sect:forcing-trees}

PA degrees play an essential role in the computability-theoretic analysis of the pigeonhole principle. Indeed, the notion of forcing used to build solutions to the pigeonhole principle is a variant of Mathias forcing whose reservoirs belong to a Scott ideal. Therefore, to prove an avoidance or preservation property for the pigeonhole principle, one must first prove a basis theorem for $\Pi^0_1$ classes with respect to the same property.

PA degrees admit several characterizations and therefore form a robust notion. A function $f : \NN \to \NN$ is \emph{diagonally non-$X$-computable} ($X$-DNC) if $f(e) \neq \Phi_e^X(e)$ for every~$e \in \NN$. A set is of PA degree over~$X$ if and only if it computes a diagonally non-$X$-computable $\{0,1\}$-valued function, or equivalently if, given an enumeration $\psi_0, \psi_1, \dots$ of all $\Pi^0_1(X)$ formulas, it computes a function $g : \NN^2 \to \{0,1\}$ such that $(\forall e_0, e_1)[(\psi_{e_0} \vee \psi_{e_1}) \rightarrow \psi_{e_{g(e_0, e_1)}}]$. Furthermore, there exists an $X$-computable infinite binary tree whose paths are exactly the $\{0,1\}$-valued $X$-DNC functions, so there exists a maximally difficult tree whose paths are all of PA degree. It follows that any computability-theoretic result about PA degrees can be stated equivalently over PA degrees or over members of $\Pi^0_1$ classes.

For our purpose, we will need to prove the existence of PA degrees which preserve multiple hyperimmunities relative to various levels of the arithmetic hierarchy. For example, if $f$ is hyperimmune and $g$ is $\emptyset'$-hyperimmune, one wants to prove the existence of a set~$P$ of PA degree such that $f$ is $P$-hyperimmune and $g$ is $P'$-hyperimmune. There exist, among others, two well-known basis theorems which serve a large majority of the purposes: the low and the computably dominated basis theorem~\cite{jockusch1972classes}. A set~$X$ is of \emph{computably dominated} degree if every total $X$-computable function is dominated by a total computable function, or equivalently if it does not compute any hyperimmune function.

\begin{theorem}[Jockusch and Soare~\cite{jockusch1972classes}]\label[theorem]{thm:low-compdom-pi01-basis-theorems}
Let $\C \subseteq 2^\NN$ be a non-empty $\Pi^0_1$ class.
\begin{enumerate}
    \item There exists a member~$X \in \C$ of low degree ($X' \leq_T \emptyset'$).
    \item There exists a member~$X \in \C$ of computably dominated degree.
\end{enumerate}
\end{theorem}

Given~$f$ and $g$ as in the discussion above (before \Cref{thm:low-compdom-pi01-basis-theorems}), by the computably dominated basis theorem, there exists a set~$P$ of PA and computably dominated degree. In particular, $f$ is $P$-hyperimmune. On the other hand, by the low basis theorem, there is a set~$Q$ of PA and low degree, so since~$g$ is $\emptyset'$-hyperimmune and $Q' \leq_T \emptyset'$, $g$ is $Q'$-hyperimmune. The difficulty is to preserve both hyperimmunities simultaneously, as every low degree is hyperimmune, so no set can be simultaneously of low and computably dominated degree.

In order to prove the preservation of multiple hyperimmunities simultaneously, we use a notion of forcing with primitive recursive trees introduced by Wang \cite{wang2016definability}, who showed the existence of a forcing question with good definability properties at every level. For the sake of completeness, we state the properties of his notion of forcing, and use it to prove the existence of PA degrees which preserve multiple hyperimmunities simultaneously.

\begin{definition}
    Let~$\TT$ be the notion of forcing whose \emph{conditions} are infinite primitive recursive trees, partially ordered by the inclusion relation. We write $S \leq T$ for the inclusion and say that $S$ \emph{extends} $T$. 
\end{definition}

A non-empty class~$\P \subseteq \cs$ is $\Pi^0_1$ if and only if there is a co-c.e. pruned (i.e. with no leaves) tree~$T \subseteq \bstr$ such that $\P = [T]$. Using a delaying trick, for every co-c.e. tree~$T$, there is a primitive recursive tree~$S$ such that $[T] = [S]$ (see \cite[Lemma 3.8]{wang2016definability}).
It follows that there is a primitive recursive tree whose paths are exactly the $\{0,1\}$-valued DNC functions. 

The partial order $(\TT, \leq)$ being non-computable, the usual inductive definition of the forcing relation
does not have the right definitional complexity already for deciding $\Sigma^0_1(G)$-properties of the generic object. We shall therefore define a custom forcing relation for $\Sigma^0_1$ and $\Pi^0_1$ formulas. At higher levels, the complexity of the partial order is absorbed by the complexity of the forced formula, and therefore one can use the standard inductive definition of the forcing relation.

\begin{definition}
    We define inductively a forcing relation for arithmetic formulas as follows:\\
    For $\phi(G,x)$ a $\Pi_0^0$ formula:
    \begin{itemize}
        \item $T \Vdash (\exists x)\phi(G,x)$ if $ (\exists \ell)(\forall \sigma \in 2^{\ell} \cap T)(\exists x \leq \ell) \phi(\sigma,x)$,
        \item $T \Vdash (\forall x) \neg \phi(G,x)$ if $(\forall \sigma \in T)(\forall x) \neg \phi(\sigma,x)$.
    \end{itemize}
    For $\phi(G,x)$ a $\Pi_n^0$ formula for $n \geq 1$:
    \begin{itemize}
         \item $T \Vdash (\exists x) \phi(G,x)$ if $ T \Vdash \phi(G,a)$ for some $a \in \NN$,
        \item $T \Vdash(\forall x) \neg \phi(G,x)$ if $(\forall S \leq T) S \not \Vdash (\exists x) \phi(G,x)$.
    \end{itemize}
\end{definition}

The following lemma corresponds to~\cite[Lemma 3.13]{wang2016definability}, and states that the forcing relation has the same definitional complexity as the formulas it forces.

\begin{lemma}\label{lem:tree-forcing-relation-preserving}
    Let $n \in \NN$, $T \in \TT$ and $\phi(G,x)$ be a $\Pi_n^0$ formula. The formula $T \Vdash (\exists x)\phi(G,x)$ is $\Sigma_{n+1}^0$, and the formula  $T \Vdash (\forall x)\neg \phi(G,x)$ is $\Pi_{n+1}^0$.
\end{lemma}

The following lemma corresponds to \cite[Lemma 3.12]{wang2016definability}
and states that the forcing relation is sound and complete.

\begin{lemma}
Let $\F$ be a sufficiently generic filter. Then $\bigcap_{T \in \F} [T]$ contains a single element $G_{\F}$ and, for every arithmetic formula $\phi(G)$, $\phi(G_{\F})$ holds if and only if there is a condition $T \in \F$ forcing $\phi(G)$.
\end{lemma}

We can now define a $\Sigma^0_n$-preserving forcing question for $\Sigma^0_n$ formulas based on
the previously defined forcing relation.

\begin{definition}[Forcing question]\label{def:tree-forcing-question}
We define the $\Sigma_n^0$-forcing question for $n \geq 1$ as follows: Let $\phi(G,x)$ be a $\Pi_{n-1}^0$ formula
\begin{itemize}
    \item If $n = 1$, $T \qvdash (\exists x)\phi(G,x)$ holds if $T \Vdash (\exists x)\phi(G,x)$.
    \item If $n > 1$, $T \qvdash (\exists x) \phi(G,x)$ holds if $T \not \Vdash (\forall x) \neg \phi(G,x)$.
\end{itemize}   
\end{definition}

The following lemma states that \Cref{def:tree-forcing-question} meets the specifications of a forcing question. 
Moreover, by \Cref{lem:tree-forcing-relation-preserving}, this forcing question is $\Sigma^0_n$-preserving. 

\begin{lemma}\label[lemma]{lem:tree-question-find-extension}
    Let $n \geq 1$, let  $\phi(G,x)$ be a $\Pi_{n-1}^0$ formula, and let $T \in \TT$.
    \begin{itemize}
        \item If $T \qvdash (\exists x)\phi(G,x)$, then $\exists S \leq T$ such that $S \Vdash (\exists x)\phi(G,x)$.
        \item If $T \nqvdash (\exists x) \phi(G,x)$, then $\exists S \leq T$ such that $S \Vdash (\forall x)\neg \phi(G,x)$.
    \end{itemize}
\end{lemma}
\begin{proof}
If $T \qvdash (\exists x)\phi(G,x)$, there are two cases:
\begin{itemize}
    \item If $n = 1$, then $T \Vdash (\exists x)\phi(G,x)$.
    \item If $n > 1$, then $T \not \Vdash (\forall x) \neg \phi(G,x)$, hence there exists some $S \leq T$ such that $S \Vdash (\exists x) \phi(G,x)$.
\end{itemize}
If $T \nqvdash (\exists x)\phi(G,x)$, there are two cases:
\begin{itemize}
    \item If $n = 1$, then $T \not \Vdash (\exists x)\phi(G,x)$, hence $(\forall \ell)(\exists \sigma \in 2^{\ell} \cap T)(\forall x \leq \ell) \neg \phi(\sigma,x)$. Then the tree $S = \{\sigma \in T : (\forall x \leq |\sigma|) \neg \phi(\sigma,x) \}$ is an infinite primitive recursive subtree of $T$ such that $S \Vdash (\forall x) \neg \phi(G,x)$.
    \item If $n > 1$, then $T \Vdash (\forall x) \neg \phi(G,x)$.
\end{itemize}
\end{proof}

Since we are interested in preservation of iterated hyperimmunity, the following lemma states that the forcing question is $\Sigma^0_n$-compact.

\begin{lemma}\label[lemma]{lem:tree-question-compact}
    For all $n \leq 1$, the $\qvdash$ relation for $\Sigma_n^0$ formulas is compact, i.e., if $T \qvdash (\exists x)\phi(G,x)$, then there exists some bound~$k$ such that $T \qvdash (\exists x < k)\phi(G,x)$.
\end{lemma}
\begin{proof}
Assume $T \qvdash (\exists x)\phi(G,x)$ for some $\Sigma_n^0$ formula $(\exists x)\phi(G,x)$. There are two cases:
\begin{itemize}
    \item If $n = 1$, then $T \Vdash (\exists x)\phi(G,x)$ and there exists some $\ell$ such that $(\forall \sigma \in 2^{\ell} \cap T)(\exists x \leq \ell) \phi(\sigma,x)$, hence $T \Vdash (\exists x \leq \ell)\phi(G,x)$ and $T \qvdash (\exists x \leq \ell)\phi(G,x)$.
    \item If $n > 1$, then $T \not \Vdash (\forall x) \neg \phi(G,x)$, hence there exists some $S \leq T$ such that $S \Vdash (\exists x) \phi(G,x)$. Fix such a tree $S$, there exists some $k$ such that $S \Vdash \phi(G,k)$, hence $T \qvdash (\exists x \leq k) \phi(G,x)$.
\end{itemize}

\end{proof}

We are now ready to prove the main theorem of this section, based on the abstract framework of forcing questions.

\begin{theorem}\label[theorem]{thm:wkl-hyperimmunity-preservation}
Let $\{(f_s,n_s)\}_{s \in \NN}$ be a family such that for every $s\in \NN$, $f_s$ is $\emptyset^{(n_s)}$-hyperimmune. Let $T \subseteq 2^{<\NN}$ be an infinite computable tree. There exists some path $P$ in $[T]$ such that $f_s$ is $P^{(n_s)}$-hyperimmune for every $s \in \NN$.
\end{theorem}

\begin{proof}
By \Cref{lem:tree-question-compact,lem:tree-question-find-extension,lem:tree-forcing-relation-preserving}, the notion of forcing $(\PP, \leq)$ admits a $\Sigma^0_n$-compact, $\Sigma^0_n$-preserving forcing question, so apply \Cref{thm:abstract-compact-hyperimmunity}.

\end{proof}

\begin{remark}
Wang~\cite{wang2016definability} proved, given a family $\{(C_s, n_s)\}_{s \in \NN}$ such that for every~$s \in \NN$, $C_s$ is not $\Sigma^0_{n_s}$, the existence of a set~$P$ of PA degree such that for every~$s \in \NN$, $C_s$ is not $\Sigma^0_{n_s}(P)$. Downey et al.~\cite{downey2022relationships} studied the relationships between notions of preservations and avoidance, and proved in particular that preservation of hyperimmunity is equivalent to preservation of non-$\Sigma^0_1$ definitions. Their proof relativizes to iterated jumps, but it is not known to be equivalent when working with levels of the hierarchy simultaneously. Therefore, the main theorem of this section (\Cref{thm:wkl-hyperimmunity-preservation}) is not a consequence of Wang's result.
\end{remark}

As mentioned, the degree-theoretic study of PA degrees and members of $\Pi^0_1$ classes coincide, as there exists a maximal $\Pi^0_1$ class containing only sets of PA degree. The following well-known proposition shows that the degrees of Scott codes coincide with PA degrees, hence it is not more complicated to compute hierarchies of PA degrees than a single one. Recall that a \emph{Scott code} of a countable Scott ideal~$\M = \{ Z_i : i \in \NN \}$ is a set $M = \bigoplus_{i} Z_i$ such that the basic operations (Turing reducibility, effective join, taking a path through a binary tree) on the $M$-indices are computable.

\begin{proposition}[Scott~\cite{scott1962algebra}]\label{prop:pa-to-scott}
For every set~$X$, there exists a non-empty $\Pi^0_1(X)$ class~$\C(X)$ containing only Scott codes of Scott ideals containing~$X$.
\end{proposition}
\begin{proof}
Let $\C(X)$ be the class of all $\bigoplus_n Z_n$ such that for every~$a, b \in \NN$, $Z_{\langle 0, a, b\rangle} = Z_a \oplus Z_b$ and for every~$e, a \in \NN$, $Z_{\langle 1, e, a\rangle}$ is a completion of the partial function $\Phi_e^{Z_a}$. Let $M = \bigoplus_i Z_i \in \C(X)$ and $\M = \{ Z_i : i \in \NN \}$. By construction, $\M$ is closed under effective join. We claim that $\M$ downward-closed under Turing reducibility. Let $Z_a \in \M$ and $Y \leq_T Z_a$. Then there is a Turing functional $\Phi_e$ such that $\Phi_e^{Z_a} = Y$, so $Z_{\langle 1, e, a \rangle} = Y \in \M$.

We now claim that $\M$ is a Scott ideal. Let $Z_a \in \M$, and let $\Phi_e$ be a Turing functional such that for every set $X$, every $x \in \NN$ and $i < 2$, $\Phi_e^{X}(x)\converge = 1-i$ if and only if $\Phi^X_x(x)\converge = i$. Then any completion of $\Phi_e^X$ is a $\{0,1\}$-valued $X$-DNC function, hence of PA degree over~$X$. It follows that $Z_{\langle 1, e, a\rangle} \in \M$ is of PA degree over~$Z_a$.
\end{proof}

\section{Largeness and partition regularity}\label{sect:largeness}

Solutions to problems from Ramsey theory are often constructed using variants of Mathias forcing, that is, with conditions consisting of a finite stem and an infinite reservoir. Even in the case of computable Mathias forcing, where the reservoirs are computable, the $\Sigma^0_2(G)$ and $\Pi^0_2(G)$ properties of the generic object~$G$ are generally more complex than the $\Sigma^0_2$ and $\Pi^0_2$ properties of the ground model. In particular, every sufficiently generic set for computable Mathias forcing is of high degree. Recall that a function $f : \NN \to \NN$ is \emph{dominant} if it eventually dominates every computable function. By Martin's domination theorem~\cite{martin1966classes}, a degree is high if and only if it computes a dominant function. The \emph{principal function} of an infinite set~$X = \{x_0 < x_1 < \dots \}$ is the function $p_X : \NN \to \NN$ defined by~$n \mapsto x_n$.

\begin{proposition}[Folklore]
Let~$\F$ be a sufficiently generic filter for computable Mathias forcing.
Then the principal function of $G_\F$ is dominant.
\end{proposition}
\begin{proof}
Given a total computable function~$g : \NN \to \NN$ and a computable Mathias condition $(\sigma, X)$, one can computably thin out the reservoir~$X$ to obtain a computable reservoir~$Y$ such that $p_{\sigma \cup Y}$ eventually dominates~$g$.
The condition $(\sigma, Y)$ is an extension of~$(\sigma, X)$ forcing the principal function of $G_\F$ to dominate~$g$.
\end{proof}

A condition can be seen as an invariant property that is preserved along the construction of an infinite object: given a mathematical approximation satisfying some structural properties, one can apply one step of the construction, and obtain another mathematical approximation satisfying the same structural properties.

In the construction of solutions to the pigeonhole principle, Mathias forcing over Scott ideals is an appropriate invariant for a good first-jump control, but it is an over-generalization preventing from having a good second-jump control: the forcing relation for $\Pi^0_2(G)$ properties is a density statement about an infinite collection of $\Sigma^0_1(G)$ properties. It requires guaranteeing some \emph{positive} information about the future, while a reservoir forces some \emph{negative} information, as it restricts the candidate integers that can be added to the generic set. One must therefore use a \qt{reservoir of reservoirs}, which will restrict the possible choices of reservoirs, hence will restrict the future negative information, which is a way of forcing positive information.

This \qt{reservoir of reservoirs} must still allow the necessary operations on the reservoirs to ensure a good first-jump control. Looking at the combinatorics of a first-jump control of the pigeonhole principle, the only operations on the reservoirs are finite truncation, and splitting based on a 2-partition. This naturally yields the notion of partition regular class.
Partition regularity is a generalization of the notion of infinity.

\begin{definition}
A class $\A \subseteq 2^{\NN}$ is \textit{partition regular} if :
\begin{itemize}
    \item $\A$ is non-empty,
    \item for all $X \in \A$, if $ X \subseteq Y$, then $Y \in \A$,
    \item for every $X \in \A$, for every $2$-cover $Y_0 \cup Y_1 \supseteq X$, there exists $i < 2$ such that $Y_i \in \A$.
\end{itemize}
\end{definition}

By iterating the splitting, if $\A$ is partition regular, then for every integer $k$, for every $X \in \A$, and every $k$-cover $Y_1, Y_2, \dots Y_k$ of $X$, there exists some $i \leq k$ such that $Y_i \in \A$.
By the infinite pigeonhole principle, the class of all infinite sets is partition regular. We will be interested in partition regular classes having only infinite sets. These classes are called \emph{non-trivial}. Equivalently, a partition regular class is non-trivial if every set has at least 2 elements.
For a set $X$, let $\L_X$ be the $\Pi_2^0(X)$ partition regular class containing all the sets having an infinite intersection with $X$. 

Given a partition regular class~$\A \subseteq \cs$, one can construct solutions to the pigeonhole principle with a good first-jump control, using a variant of Mathias forcing whose reservoirs belong to~$\A$. Cholak, Jockusch and Slaman~\cite{cholak2001strength} and Dorais~\cite{dorais2012variant} first used variants of Mathias forcing with reservoirs in partition regular classes to build generic sets of non-high degree. The technique was then developed by Monin and Patey~\cite{monin201pigeons,monin2021weakness,monin2021srt,monin2022partition} to prove several basis theorems about the pigeonhole principle.

\subsection{Large classes}

One should expect from a notion of largeness that it is closed upwards under inclusion. The collection of all partition regular classes is not closed upwards: for example, letting~$X$ be any infinite and co-infinite set, $\L_X$ is partition regular, but $\L_X \cup \{\overline{X}\}$ is not. The following notion of largeness is more convenient to work with, and closely related to partition regularity.

\begin{definition}
A class $\A \subseteq 2^{\NN}$ is \textit{large} if :
\begin{itemize}
    \item for all $X \in \A$, if $ X \subseteq Y$, then $Y \in \A$,
    \item for every integer $k$, for every $k$-cover $Y_1, Y_2, \dots Y_k$ of $\NN$, there exists $i \leq k$ such that $Y_i \in \A$.
\end{itemize}
\end{definition}

The notion of largeness was introduced and studied by Monin and Patey~\cite{monin201pigeons}. They proved that a class is large if and only if it contains a partition regular subclass. Furthermore, every large class contains a maximal partition regular subclass for inclusion, which admits an explicit syntactic definition.

\begin{definition}
Given a large class $\A \subseteq \cs$, let
$$\L(\A) = \{X \in \A: \forall k \forall X_0 \cup \dots \cup X_{k-1} \supseteq X~\exists i < k~X_i \in \A\}$$
\end{definition}

Monin and Patey~\cite{monin201pigeons} proved that if $\A$ is large, then $\L(\A)$ is the maximal partition regular subclass of~$\A$. Large classes satisfy a very useful combinatorial property that we shall use all over the article:

\begin{lemma}[\cite{monin201pigeons}]\label[lemma]{lem:decreasing-sequence-large}
Let $\A_0 \supseteq \A_1 \supseteq \dots$ a decreasing sequence of large classes, then $\bigcap_{i \in \NN} \A_i$ is large.  
\end{lemma}

The contrapositive of \Cref{lem:decreasing-sequence-large} has some compactness flavor. Indeed, if an intersection $\bigcap_{i} \A_i$ of a collection of classes $\A_0, \A_1, \dots$ is not large, then there is some~$n \in \NN$
such that $\bigcap_{i < n} \A_i$ is not large. We shall be interested only in G${}_\delta$ large classes, that is, intersections of open large classes. For this, we consider $W_0, W_1, \dots \subseteq \bstr$ as an effective enumeration of all c.e.\ sets of strings, and let $\U_0, \U_1, \dots$ be defined by
$$
\U_e = \{ X \in \cs : \exists \rho \subseteq X\ \rho \in W_e \}
$$
Thus, $\U_0, \U_1, \dots$ is a uniform enumeration of all upward-closed $\Sigma^0_1$ classes.
By an immediate relativization, we let $\U_0^Z, \U_1^Z, \dots$ be a uniform enumeration of all upward-closed $\Sigma^0_1(Z)$ classes.
From now on, fix a Scott ideal $\M = \{Z_0,Z_1,\dots\}$ with Scott code~$M$ (in other words, $M = \bigoplus_{i} Z_i$ and the basic operations on the $M$-indices are computable).
Given a set~$C \subseteq \NN^2$, we let 
$$\U_C^{\M} = \bigcap_{(e,i) \in C} \U_e^{Z_i}$$
Thanks to \Cref{lem:decreasing-sequence-large}, largeness of an arbitrary intersection of $\Sigma^0_1$ classes can be reduced to checking largeness of a finite intersection of $\Sigma^0_1$ classes, which is a $\Pi^0_2$ statement. The following lemma gives the relativized complexity of the general statement:

\begin{lemma}[{\cite{monin2021weakness}}]\label[lemma]{lem:complexity-largeness}
Let $C \subseteq \NN^2$ be a set, the statement “$\U_C^{\M}$ is large” is $\Pi_1^0(C \oplus M')$.
\end{lemma}
\begin{proof}[Proof sketch]
By \Cref{lem:decreasing-sequence-large} and by compactness, this statement can be rephrased as \qt{for every $k$ and every finite subset $E$ of $C$, there exists some $n$ such that for every $k$-partition of $\{0,\dots, n\}$, one of its parts belongs to $\bigcap_{(i,e) \in E} \U_e^{Z_i}$.}
\end{proof}

Given a large $\Sigma^0_1$ class~$\U$, its largest partition regular subclass $\L(\U)$ is $\Pi^0_2$. Still by \Cref{lem:decreasing-sequence-large}, the largest partition regular subclass of a large $\Pi^0_2$ class is again $\Pi^0_2$. One can therefore switch from largeness to partition regularity with no additional cost:

\begin{lemma}[\cite{monin2021weakness}]
Let $C \subseteq \NN^2$ be a set, then $\L(\U_C^{\M}) = \bigcap_{F \finsub C} \L(\U_F^\M)$ is $\Pi_1^0(C \oplus M')$ and there exists a set $D$ computable uniformly in $C$ such that $\L(\U_C^{\M}) = \U_D^{\M}$.
\end{lemma}

\subsection{$\M$-minimal classes}

As mentioned above, the notion of forcing for constructing solutions to the pigeonhole principle with a good first-jump control is a variant of Mathias forcing whose conditions belong to a Scott ideal. To obtain a good second-jump control, one must restrict the reservoirs to some well-chosen partition regular class. 

Given the computability-theoretic nature of the $\Sigma^0_2(G)$ and $\Pi^0_2(G)$ statements that need to be forced, the appropriate partition regular class does not admit a nice explicit combinatorial definition. One can either decide to start with the simplest partition regular class of all the infinite sets, and refine this class over the construction by considering partition regular subclasses which will ensure stronger positive information about the reservoirs, or build once and for all the most restrictive partition regular class, in other words, the partition regular class which will maintain as much positive information about the reservoirs as possible. We adopt the latter approach.

Seeing a partition regular class as a \qt{reservoir of reservoirs}, if~$\A \subseteq \B$ are two partition regular classes, $\A$ will impose more restrictions on the possible choice of reservoirs than~$\B$. Considering that a reservoir forces negative information about the set, $\A$ will force more positive information than~$\B$. Therefore, minimal partition regular classes will ensure as much positive information as possible.

\begin{definition}
A large class $\A$ is \emph{$\M$-minimal} if for every $X \in \M$ and $e \in \NN$, either $\A \subseteq \U_e^X$ or $\A \cap \U_e^X$ is not large.
\end{definition}

Every large class containing a partition regular subclass, every $\M$-minimal large class of the form $\U^\M_C$ is also partition regular.
There is a natural greedy algorithm to build a set~$C \subseteq \NN^2$ such that $\U_C^\M$ is non-trivial and $\M$-minimal.

\begin{lemma}\label{lem:computation-m-minimal}
Let~$\M$ be a Scott ideal with Scott code~$M$ and let $D \subseteq \NN^2$ be a set of indices such that $\U^\M_D$ is large and contains only infinite sets. Then $(D \oplus M')'$ computes a set~$C \supseteq D$ such that $\U^\M_C$ is $\M$-minimal.
\end{lemma}
\begin{proof}
By the padding lemma, there is a total computable function $g : \NN^2 \to \NN$ such that for every $e, s \in \NN$ and every set~$X$, $\U^X_{g(e, s)} = \U^X_e$ and $g(e,s) > s$. By uniformity of the properties of a Scott code, there is another total computable function $h : \NN^2 \to \NN$ such that for every $e, s \in \NN$ and every Scott code~$M$, $h(e,s)$ and $e$ are both $M$-indices of the same set, and $h(e,s) > s$.

We build a $(D \oplus M')'$-computable sequence of $D$-computable sets $C_0 \subseteq C_1 \subseteq \dots$ such that, letting $C = \bigcup_s C_s$, $\U^M_C$ is $\M$-minimal and for every~$s$, $C \uh s = C_s \uh s$.
Start with $C_0 = D$.
Then, given a set~$C_s \subseteq \NN^2$ such that $\U_{C_s}^\M$ is large, and a pair $(e, i)$, define $C_{s+1} = C_s \cup \{(g(e,s),h(i,s))\}$ if $\U_{C_s}^\M \cap \U_e^{Z_i}$ is large, and $C_{s+1} = C_s$ otherwise. The set~$C = \bigcup_s C_s$ is the desired set. Note that by choice of~$g$ and $h$, in the former case, $\U_{C_{s+1}}^\M = \U_{C_s}^\M \cap \U_e^{Z_i}$.
By \Cref{lem:complexity-largeness}, the statement \qt{$\U_{C_s}^\M \cap \U_e^{Z_i}$ is large} is $\Pi^0_1(C_s \oplus M')$, so it can be decided $(D \oplus M')'$-computably since $C_s \leq_T D$. The use of $g$ and $h$ ensures that $C_{s+1} \uh s = C_s \uh s$.
\end{proof}

One can apply \Cref{lem:computation-m-minimal} with $D = \{ (e_s, i) : s \in \NN \}$ where $\U_{e_s}^{Z_i} = \{ Y \in \cs : \operatorname{card} Y \geq s \}$ to obtain a set~$C \leq_T M''$ such that $\U^M_C$ is $\M$-minimal.
However, being $M''$-computable is too complex for our purpose. Thankfully, one does not need to explicitly have access to the set of indices of the $\M$-minimal class, but only to be able to check that a class is compatible with it. This yields the notion of $\M$-cohesive class.



\subsection{$\M$-cohesive classes}

In the previous algorithm for constructing an $\M$-minimal class, the order in which one considers the pairs $(e, i)$ matters. Indeed, if $\A$ is large and $\B, \C \subseteq \A$ are two large subclasses, then $\B \cap \C$ is not necessarily large. Therefore, there exist many $\M$-minimal classes, depending on the ordering of the pairs. The notion of $\M$-cohesiveness is a way of choosing an $\M$-minimal class without explicitly giving its set of indices.

\begin{definition}
    A large class $\A$ is \emph{$\M$-cohesive} if for every $X \in \M$, either $\A \subseteq \L_X$ or $\A \subseteq \L_{\overline{X}}$.
\end{definition}

It follows from the definition that for every infinite set $X \in \M$, if $X \in \A$, then $\A \subseteq \L_X$.
Indeed, otherwise, by $\M$-cohesiveness, $X \in \A \subseteq \L_{\overline{X}}$, yielding a contradiction.
The cohesiveness terminology comes from the cohesiveness principle ($\COH$), which states for every infinite sequence of sets~$R_0, R_1, \dots$, the existence of an infinite set $H \subseteq \NN$ such that for every~$n \in \NN$, either $H \subseteq^* R_n$ or $H \subseteq^* \overline{R}_n$. Such a set~$H$ is said to be \emph{cohesive} for the sequence. There exists an immediate correspondence between the cohesiveness principle and the existence of $\M$-cohesive classes. Indeed, given an infinite set~$H$ which is cohesive for the sequence $\M = \{Z_0, Z_1, \dots \}$,
the class~$\L_H$ is partition regular and $\M$-cohesive. 

The following lemma shows that an $\M$-cohesive class already contains the information of an $\M$-minimal class, in the sense that in the greedy algorithm to build an $\M$-minimal class from an $\M$-cohesive one, the ordering on the pairs does not matter.

\begin{lemma}[\cite{monin2021weakness}]\label[lemma]{lem:cohesive-compatibility}
    Let $\U_C^{\M}$ be an $\M$-cohesive class. Let $\U_D^{\M}$ and  $\U_E^\M$ be such that $\U_C^\M \cap \U_D^\M$ and  $\U_C^\M \cap \U_E^\M$ are both large. Then so is $\U_C^\M \cap \U_D^\M \cap  \U_E^\M$.
\end{lemma}

It follows that every $\M$-cohesive class admits a unique $\M$-minimal large subclass.

\begin{lemma}[\cite{monin2021weakness}]
For every $\M$-cohesive class $\U_C^{\M}$, there exists a unique $\M$-minimal large subclass:

$$\langle \U_C^{\M} \rangle = \bigcap_{e \in \NN, X \in \M} \{\U_e^{X} : \U_C^{\M} \cap \U_e^{X} \textit{is large}\}$$
\end{lemma}

Contrary to $\M$-minimal classes, one can build a set~$C \subseteq \NN^2$ such that $\U_C^\M$ is $\M$-cohesive computably in any PA degree over~$M'$. There are two possible constructions: either using the correspondence with the cohesiveness principle, knowing that any PA degree over $M'$ computes the jump of an infinite set~$H \subseteq \NN$ cohesive of~$\M = \{Z_0, Z_1, \dots \}$, and computing in $H'$ the set~$C$, or directly building the set $C$ by deciding, given a set~$C_s \subseteq \NN^2$ and a set $Z_i$, whether $\U_{C_s}^{\M} \cap \L_{Z_i}$ or $\U_{C_s}^{\M} \cap \L_{\overline{Z}_i}$ is large. We prove it formally with the latter approach.

\begin{lemma}\label{lem:computation-m-cohesive}
Let~$\M$ be a Scott ideal with Scott code~$M$ and let $D \subseteq \NN^2$ be a set of indices such that $\U^M_D$ is large and contains only infinite sets. Then any PA degree over~$D \oplus M'$ computes a set~$C \supseteq D$ such that $\U^\M_C$ is $\M$-cohesive.
\end{lemma}
\begin{proof}
Fix $P$ a PA degree over $D \oplus M'$. Recall that $P$ is able to choose, among two $\Pi^0_1(D \oplus M')$ formulas such that at least one is true, a valid one.

First, consider two $M$-computable enumerations of sets $(E_n)_{n\in \NN}$ and  $(F_n)_{n\in \NN}$ such that for every $n \in \NN$, $\U_{E_n}^{Z_n} = \L_{Z_n}$ and $\U_{F_n}^{Z_n} = \L_{\overline{Z_n}}$. By the padding lemma, one can suppose that $\min E_n, \min F_n \geq n$. The set $C$ will be defined as $\bigcup_{n \in \NN} C_n$ for $C_0 \subseteq C_1 \subseteq \dots$ a $P$-computable sequence of $M \oplus D$-computable sets satisfying:
\begin{itemize}
    \item $C_{0} = D$,
    \item $\U_{C_k}^{\M}$ is large for every $k \in \NN$,
    \item $C_k \uh k = C \uh k$ for every $k \in \NN$, and thus $C$ will be $P$-computable.
\end{itemize}

Let $C_{0} = D$, then, by assumption, $\U_{C_{0}}^{\M}$ is large.

Assume $C_k$ has been defined for some $k \in \NN$. Then, as $\U_{C_k}^{\M}$ is large, one of the two following $\Pi_1^0(D \oplus M')$ statements must hold: $\qt{\U_{C_k}^{\M}\ \cap \L_{Z_k} \mbox{is large}}$ or $\qt{\U_{C_k}^{\M}\ \cap \L_{\overline{Z_k}} \mbox{is large}}$. Hence, $P$ is able to choose one that is true. If $\U_{C_k}^{\M}\ \cap \L_{Z_k}$ is large, let $C_{k+1} = C_k \cup E_k$, and if $\U_{C_k}^{\M}\ \cap \L_{\overline{Z_k}}$ is large, let $C_{k+1} = C_k \cup F_k$. By our assumption that $\min E_n, \min F_n \geq n$ for all $n$, the value of $C_k \uh k$ will be left unchanged in the rest of the construction.
\end{proof}

The above construction of~$C$ carries the information, given an element $X \in \M$, whether $X \in \U_C^\M$ or not (or equivalently whether~$X \in \langle \U_C^\M \rangle$ or not, or again whether $\U_C^\M \subseteq \L_X$ or not). The following lemmas shows that this information can be recovered by any such set~$C$ independently of this construction, with the help of~$M'$.

\begin{lemma}\label[lemma]{lem:complexity-finding-among-partition}
    Let $\U_C^\M$ be an $\M$-cohesive class. 
    $C \oplus M'$ computes a function $f : \NN \to 2$ such that for every~$a \in \NN$, $f(a) = 1$ if and only if $Z_a \in \U_C^\M$.
\end{lemma}
\begin{proof}
By $\M$-cohesiveness, $Z_a \in \U_C^{\M}$ if and only if $\U_C^\M \subseteq \L_{Z_a}$. The statement $\U_C^\M \subseteq \L_{Z_a}$ is $\Pi^0_1(C \oplus M')$ (as it is equivalent to $\U_C^\M \cap \L_{Z_a}$ large
). This statement is also $\Sigma^0_1(C \oplus M')$, since it is equivalent to $\U_C^\M \cap \L_{\overline{Z_a}}$ not large, hence it is $\Delta^0_1(C \oplus M')$ and therefore decidable by $C \oplus M'$.
\end{proof}

The following lemma shows that, in some sense, the construction above of an $\M$-cohesive class can be done without loss of generality.

\begin{lemma}
    Let $\U_C^\M$ be an $\M$-cohesive class, then $C \oplus M'$ is of PA degree over $X'$ for every~$X \in \M$.
\end{lemma}

\begin{proof}
    Let $X \in \M$. Fix an $X$-c.e. enumeration $E_0 \subseteq E_1 \subseteq \dots$ of~$X'$.
    Given~$e \in \NN$, let $Y_e$ be the set of all~$n$ such that $\Phi_e^{E_n}(e)[n]\converge = 1$. Note that if~$\Phi_e^{X'}(e)\converge = 1$ then $Y_e$ is cofinite, and if $\Phi_e^{X'}(e)\converge = 0$ then $Y_e$ is finite. It follows that if $\Phi_e^{X'}(e)\converge$, then $Y_e \in \U_C^\M$ if and only if $\Phi_e^{X'}(e) = 1$. Let $g : \NN \to \NN$ be the computable function which to~$e$ associates an $M$-index for~$Y_e$ and let $f : \NN \to 2$ be the $C \oplus M'$-computable function of \Cref{lem:complexity-finding-among-partition}. Then $e \mapsto 1-f(g(e))$ is a  $C \oplus M'$-computable $\{0,1\}$-valued $X'$-DNC function, hence $C \oplus M'$ is of PA degree over~$X'$.
\end{proof}

Note that in the case where~$M$ is of low degree, then $C \oplus M'$ is of PA degree over~$M'$.
It is not clear at first sight that the notions of $\M$-cohesiveness and $\M$-minimality do not coincide, at least for classes of the form $\U_C^\M$. The following proposition shows that the two notions are always distinct. It is not of direct use for the remainder of this article, but of independent interest.

\begin{proposition}
For every countable Turing ideal~$\M$, there exists a set $C$ such that $\U_C^\M$ is $\M$-cohesive but not $\M$-minimal.
\end{proposition}

\begin{proof}
Consider the following $\Pi_2^0$ class $\P = \{X : (\forall n)(\exists x) |[x,2^x) \cap X| \geq n\}$.
\smallskip

\emph{Claim 1. The class $\P$ is partition regular.} First, $\NN \in \P$. Let~$X_0 \cup X_1 \in \P$ for some sets $X_0,X_1$, if $X_0,X_1 \notin \P$, then for all $i < 2$, there exists some $n_i \in \NN$ such that for all $x$, and  $|[x,2^x) \cap X_i| < n_i$. Thus, for all $x \in \NN$, $|[x,2^x) \cap (X_0 \cup X_1)| < n_0 + n_1$, contradicting our assumption that $X_0 \cup X_1 \in \P$. This proves our claim.
\smallskip

Let $\U_D^\M$ be an $\M$-cohesive partition regular subclass of~$\P$, which exists as $\P$ is $\Pi^0_2$.
Let $Z_0, Z_1, \dots$ be the list of all sets~$X$ in $\M$ such that $\U_D^\M \subseteq \L_X$ and
let $\U_C^\M = \bigcap_n \L_{Z_n}$.
By $\M$-cohesiveness of $\U_D^\M$, for every $X \in \M$, either $\U_D^\M \subseteq \L_X$ or $\U_D^\M \subseteq \L_{\overline{X}}$, so $\U_C^\M$ is $\M$-cohesive. Moreover, $\U_D^\M \subseteq \U_C^\M$.
\smallskip

\emph{Claim 2. $\U_C^\M \not \subseteq \U_D^\M$}
Let $X = \{x_0, x_1, \dots \}$ defined as follows: let $x_n$ be the smallest element of $\bigcap_{k \leq n} Z_k$ bigger than $2^{x_{n-1}}$ (or bigger than $0$ if $n = 0$). Note that $\bigcap_{k \leq n} Z_k$ is infinite, since $\{ \overline{Z}_k : k \leq n \} \cup \{\bigcap_{k \leq n} Z_k\}$ is a cover of~$\NN$ and none of $\overline{Z}_k$ belongs to $\U_C^\M$, so $\bigcap_{k \leq n} Z_k \in \U_C^\M \subseteq \L_\NN$. By construction $X$ cannot be an element of $\P$ as $|X \cap [x,2^x)| \leq 1$ for every $x$, and $X \in \U_C^\M$ as $X \cap Z_k$ is infinite for every $k \in \NN$.
It follows that $\U_C^\M$ is $\M$-cohesive, but not $\M$-minimal.

\end{proof}

\subsection{Scott towers, largeness towers}

In order to obtain notions of forcing with a good iterated jump control, we shall often work with hierarchies of Scott ideals with good computational properties.

\begin{definition}
Fix~$n \geq 0$. A \emph{Scott tower} of height~$n$ is a sequence of Scott ideals $\M_0, \dots, \M_n$
with Scott codes $M_0, \dots, M_n$, respectively, such that for every~$i < n$, $M_i' \in \M_{i+1}$.
\end{definition}

Intuitively, a Mathias-like notion of forcing with a good iterated-jump control will decide $\Sigma^0_{k+1}$-properties with a question with parameters in~$\M_k$.

\begin{remark}\label{rem:index-translation}
Note that for every Scott tower $\M_0, \dots, \M_n$ and $i < n$, there exists a computable function $g : \NN \to \NN$
translating $M_i$-indices into $M_{i+1}$-indices, that is, such that for every~$a \in \NN$, $g(a)$ is an $M_{i+1}$-index of the set of $M_i$-index $a$. Indeed, there is a computable function $h : \NN \to \NN$ such that $\Phi^{M_i}_{h(a)}$ is a set of $M_i$-index~$a$ for every~$a \in \NN$. For every~$a \in \NN$, let $Z_a$ be the set of $M_{i+1}$-index~$a$. By our definition of a Scott code, there exists a total computable function $q : \NN^2 \to \NN$ such that $Z_{q(e, a)}$ is a completion of $\Phi_e^{Z_a}$. Let $b$ be an $M_{i+1}$-index of~$M_i$ and let $g(a) = q(h(a), b)$. By definition, $Z_{g(a)}$ is a completion of $\Phi^{M_i}_{h(a)}$, hence is the set of~$M_i$-index~$a$.
\end{remark}

The following proposition will be useful throughout this article.

\begin{proposition}\label[lemma]{lem:scott-tower}
Fix~$n \geq 0$. There exists a Scott tower $\M_0, \dots, \M_n$ with Scott codes $M_0, \dots, M_n$, such that for every~$i \leq n$, $M_i$ is low over $\emptyset^{(i)}$, that is, $(M_i \oplus \emptyset^{(i)})' \leq_T \emptyset^{(i+1)}$.
\end{proposition}
\begin{proof}
For every set $X$, consider the $\Pi^0_1(X)$ class $\C(X)$ defined in \Cref{prop:pa-to-scott}.
By a relativized version of the low basis theorem~\cite{jockusch1972classes}, for every $i \leq n$, there is a Scott ideal~$\M_i$ containing $\emptyset^{(i)}$ with Scott code~$M_i \in \C(\emptyset^{(i)})$ of degree low over $\emptyset^{(i)}$. Then, for every $i < n$, $M_i' \equiv_T \emptyset^{(i+1)}$, and thus $M_i' \in \M_{i+1}$.
\end{proof}

Given a Scott tower $\M_0, \dots, \M_n$, we shall define some sets $C_0, \dots, C_{n-1}$ such that $\U^{\M_i}_{C_i}$ is $\M_i$-cohesive, and work with notions of forcing with multiple reservoirs, such that the reservoir~$X_i$ at level~$i$ belongs to $\M_i \cap \U^{\M_i}_{C_i}$. The combinatorics of the notions of forcing will require in particular that the various cohesive classes are compatible, that is, $\bigcap_i \U^{\M_i}_{C_i}$ is large. This is not the case in general, as given a bi-infinite set~$X$, $\L_X$ and $\L_{\overline{X}}$ are partition regular, but $\L_X \cap \L_{\overline{X}}$ is not even large. We will therefore need to ensure structurally this compatibility between the cohesive classes. This yields the notion of largeness tower.

\begin{definition}
Fix~$n \geq 0$. A \emph{largeness tower} of height~$n$ is a Scott tower $\M_0, \dots, \M_n$
together with a sequence of sets  $C_0, \dots, C_{n-1}$ such that for every~$i < n$:
\begin{enumerate}
    \item $\U_{C_i}^{\M_i}$ is an $\M_i$-cohesive large class containing only infinite sets ;
    \item $C_i \in \M_{i+1}$ ;
    \item $\U_{C_{i+1}}^{\M_{i+1}} \subseteq \langle \U_{C_i}^{\M_i} \rangle$ if $i < n-1$.
\end{enumerate}
\end{definition}

The following proposition shows that the compatibility requirements between the various notions of largeness
do not impose any constraints on the notion of Scott tower, and thus that the two constructions can done separately.

\begin{proposition}[Monin and Patey~\cite{monin2021weakness}]\label[lemma]{lem:largeness-tower}
Every Scott tower can be completed into a largeness tower.
\end{proposition}
\begin{proof}
Let $\M_0, \dots, \M_n$ be a Scott tower with Scott codes $M_0, \dots, M_n$. 
The Scott ideal $\M_1$ contains $M_0'$ hence also contains a set~$X_0$ that is PA over $M_0'$. By \Cref{lem:computation-m-cohesive}, $X_0$ computes a set $C_0$ such that $\U_{C_0}^{\M_0}$ is an $\M_0$-cohesive large class. Note that $C_0 \in \M_1$ since a Scott ideal is stable by Turing reduction.

Assume $C_k$ has already been defined for some $k < n - 1$. In particular, $\U_{C_{k}}^{\M_{k}}$ is $\M_k$-cohesive, so by \Cref{lem:computation-m-minimal}, $(C_k \oplus M_k')'$ computes a set $D_{k+1}$ such that $\U_{D_{k+1}}^{\M_k} = \langle \U_{C_{k}}^{\M_{k}} \rangle$. By \Cref{rem:index-translation}, there is a set~$E_{k+1} \leq_T D_{k+1}$ such that $\U_{E_{k+1}}^{\M_{k+1}} = \U_{D_{k+1}}^{\M_k}$. Note that $C_k \oplus M_k' \in M_{k+1}$, so $E_{k+1} \leq_T D_{k+1} \leq_T M_{k+1}'$. By definition of a largeness tower, $M_{k+1}' \in \M_{k+2}$ so there exists a set~$X_{k+1} \in \M_{k+2}$ of PA degree over~$M_{k+1}'$. By \Cref{lem:computation-m-cohesive}, $X_{k+1}$ computes a set $C_{k+1} \supseteq E_{k+1}$ such that the class $\U_{C_{k+1}}^{\M_{k+1}}$ is an $\M_{k+1}$-cohesive class. In particular, $C_{k+1} \in \M_{k+2}$.
\end{proof}

\section{Core forcing}\label{core-forcing}

Monin and Patey~\cite{monin2021weakness} designed a notion of forcing which will serve as the common combinatorial core to all the notions of forcing introduced in this article. We define it and state its main lemmas and re-prove them for the sake of completeness. 
For the remainder of this section, fix~$n \geq 0$, and let $\M_0, \dots, \M_{n+1}$ be Scott ideals with Scott codes $M_0, \dots, M_{n+1}$, respectively, and let $C_0, \dots, C_n$ be forming a largeness tower, that is, for every $i \leq n$:
\begin{itemize}
    \item $\U_{C_i}^{\M_i}$ is an $\M_i$-cohesive large class containing only infinite sets ;
    \item $C_i,M_i' \in \M_{i+1}$ ;
    \item $\U_{C_{i+1}}^{\M_{i+1}} \subseteq \langle \U_{C_i}^{\M_i} \rangle$ if $i < n$.
\end{itemize}
This hierarchy is defined abstractly because of its multiple uses in \Cref{sect:iterated-lowness} and \Cref{sect:preservation-hyps}. 

Following the intuition on $\M$-minimal partition regular classes, $\langle \U_{C_0}^{\M_0} \rangle$ is a collection of reservoirs with a maximum amount of positive information, that is, satisfying a maximum amount of $\Sigma^0_1$ properties. More generally, $\langle \U_{C_n}^{\M_n} \rangle$ satisfies a maximum amount of $\Sigma^0_{n+1}$ properties. Moreover, this hierarchy of minimal classes is ordered under inclusion, so that any reservoir in $\langle \U_{C_n}^{\M_n} \rangle$ will satisfy simultaneously all these $\Sigma^0_{k+1}$ properties for $k \leq n$. The core forcing is a refinement of Mathias forcing in which the reservoirs are required to maintain as much positive information as possible.

\begin{definition}[Condition]
For any set~$A \in \langle \U_{C_n}^{\M_n} \rangle$, let~$\PP^A_n$ be the notion of forcing whose conditions are Mathias conditions $(\sigma, X_n)$ such that
 \begin{itemize}
    \item $\sigma \subseteq A$ ;
    \item $X_n \in \M_n \cap \langle \U_{C_n}^{\M_n} \rangle$.
\end{itemize}
\end{definition}

Note that for every set~$A$, either $A$ or $\overline{A} \in \langle \U_{C_n}^{\M_n} \rangle$, hence either $\PP^A_n$ or $\PP^{\overline{A}}_n$ is a valid notion of forcing, and it might be the case for both. Actually, by Monin and Patey~\cite[Proposition 2.7]{monin2022partition}, the measure of sets~$A$ such that both $A$ and $\overline{A}$ belong to $\langle \U_{C_n}^{\M_n} \rangle$ is~1.

The definition of a condition slightly differs from the original notion~\cite{monin2021weakness} by two aspects: First, the reservoir is required to belong to~$\M_n$, while the original definition did not impose any computability-theoretic constraint on it. Second, the stem~$\sigma$ is required to be a subset of~$A$. None of these variations will affect the combinatorial properties of the notion of forcing, but they will be very useful for the study of $\Sub{\Sigma^0_n}$ and $\Sub{\Delta^0_n}$.

\begin{definition}
The set $\PP^A_n$ is partially ordered using the Mathias extension relation, that is, a $\PP^A_n$-condition $(\tau,Y_n)$ \emph{extends} a $\PP^A_n$-condition $(\sigma,X_n)$ (and we write $(\tau,Y_n) \leq (\sigma,X_n)$) if $Y_n \subseteq X_n$ and $\sigma \preceq \tau \subseteq \sigma \cup X_n$.
\end{definition}

Given a $\PP^A_n$-condition~$(\sigma, X_n)$, requiring that the stem~$\sigma$ is included in~$A$ might be an issue, since the reservoir~$X_n$ might have empty intersection with~$A$. The following lemma shows that not only $A \cap X_n$ is infinite, but $X_n$ and $X_n \cap A$ also satisfy the same large $\Sigma^0_{k+1}$ properties forced by $\langle \U_{C_n}^{\M_n} \rangle$ for every~$k \leq n$.

\begin{lemma}\label[lemma]{lem:core-reservoir-intersection}
    Let $(\sigma, X_n) \in \PP^A_n$ be a condition. Then $A \cap X_n \in \langle \U_{C_n}^{\M_n} \rangle$.
\end{lemma}
\begin{proof}
    By partition regularity of $\langle \U_{C_n}^{\M_n} \rangle$, since $A \in \langle \U_{C_n}^{\M_n} \rangle$ and $A = (A \cap X_n) \cup (A \cap \overline{X_n})$, either $A \cap X_n$ or $A \cap \overline{X_n}$ belongs to $\langle \U_{C_n}^{\M_n} \rangle$. Since $X_n \in \M_n \cap \langle \U_{C_n}^{\M_n} \rangle$ and $\U_{C_n}^{\M_n}$ is $\M_n$-cohesive, $\langle \U_{C_n}^{\M_n} \rangle \subseteq \L_{X_n}$ and therefore, $A \cap \overline{X_n} \notin \langle \U_{C_n}^{\M_n} \rangle$, hence $A \cap X_n \in \langle \U_{C_n}^{\M_n} \rangle$.
\end{proof}

Every filter~$\F \subseteq \PP^A_n$ induces a set~$G_\F$ whose characteristic function is the limit of $\{\sigma : (\sigma, X_n) \in \F\}$, that is, $x \in G_\F$ is $\sigma(x) = 1$ for some~$(\sigma, X_n) \in \F$.
It is often convenient to see a forcing condition~$c \in \F$ as an approximation of the constructed object~$G_\F$. The following notion of cylinder gives the class of \qt{candidate} objects associated to a condition.

\begin{definition}[Cylinder]
The \textit{cylinder} under a $\PP^A_n$-condition $(\sigma,X_n)$ is the class 
$$[\sigma, X_n] = \{G : \sigma \subseteq G \subseteq \sigma \cup (X_n \cap A)\}$$
\end{definition}

Given a filter~$\F$, using the \qt{candidate} interpretation of a cylinder, one should expect the resulting object~$G_\F$ to belong to the cylinder of each condition of~$\F$. We shall see in \Cref{prop:core-forcing-generic-set} that this is the case.
Moreover, still following this intuition, if a condition $d$ extends another condition~$c$, then $d$ is a more precise approximation than~$c$, so there should be less candidate objects associated to~$d$ than to~$c$. This is indeed the case: if $(\tau, Y_n) \leq (\sigma, X_n)$, then $[\tau, Y_n] \subseteq [\sigma, X_n]$. 

Monin and Patey~\cite{monin2021weakness} designed the following forcing question for the core forcing. In order to obtain a forcing question with sufficiently good definitional properties, the question does not directly involve the reservoir of the condition, but over-approximates it by asking whether the collection of reservoirs with the desired property is large. This results in a forcing question depending only on the stem of the condition. Despite this over-approximation, the resulting question for~$\Sigma^0_{n+1}$-formulas is not $\Sigma^0_{n+1}$, but rather $\Pi^0_1(\M_{n+1})$. Because of this, we shall refine this notion of forcing in the later sections to obtain better definability properties.

\begin{definition}[{Forcing question, \cite[Definition 3.3]{monin2021weakness}}]
    Let $\sigma \in 2^{<\MM}$ be a finite string.

    \begin{itemize}
        \item For $\phi(G,x)$ a $\Pi_0^0$ formula, let $\sigma \qvdash (\exists x)\phi(G,x)$ hold if: 
        $$\U_{C_0}^{\M_0} \cap \{X : (\exists \rho \subseteq X)(\exists x)\phi(\sigma \cup \rho, x) \} \textit{ is large}.$$
        \item 
    For $1 \leq k \leq n$ and $\phi(G,x)$ a $\Pi_k^0$ formula, we define inductively the relation $\sigma \qvdash (\exists x)\phi(G,x)$ to hold if: $$\U_{C_{k}}^{\M_{k}} \cap \{X : (\exists \rho \subseteq X)(\exists x) \sigma \cup \rho \nqvdash \neg \phi(G,x) \} \textit{ is large}.$$
    \end{itemize}
\end{definition}

As mentioned above, the forcing question does not have the appropriate definability property.

\begin{lemma}\label[lemma]{lem:core-complexity-question}
The statement $\sigma \qvdash \phi(G)$ is $\Pi_1^0(\M_{k+1})$ if $\phi(G)$ is a $\Sigma_{k+1}^0$ formula for $k \leq n$.
\end{lemma}
\begin{proof}
By an immediate induction using \Cref{lem:complexity-largeness}.
\end{proof}

The forcing question is defined by induction over the syntactical definition of a formula. It is therefore sensitive to the presentation of a property. In particular, logically equivalent formulas might yield a different answer through the forcing question.

Every notion of forcing induces a forcing relation by letting~$c$ force~$\varphi(G)$ if and only if $\varphi(G_\F)$ for every sufficiently generic filter~$\F$ containing~$c$. However, in many situations, it is more convenient to work with a custom syntactical forcing relation. 

\begin{definition}[{Forcing relation, \cite[Definition 3.5]{monin2021weakness}}]\label[definition]{def:core-forcing-relation}
Let $c = (\sigma,X_n)$ be a $\PP_n^A$-condition, we define the forcing relation $\Vdash$ for $\Sigma_k^0$ and $\Pi_k^0$ formulas for every $0 < k \leq n+1$ as follows: For $\phi(G,x)$ a $\Pi_0^0$ formula:
\begin{itemize}
    \item $c \Vdash (\exists x)\phi(G,x)$ if $\phi(\sigma,a)$ for some $a \in \NN$;
    \item $c \Vdash (\forall x)\neg\phi(G,x)$ if $(\forall \rho \subseteq X_n)(\forall x \in \NN) \neg\phi(\sigma \cup \rho,x)$.
\end{itemize}
For $0 < k \leq n$ and $\phi(G,x)$ a $\Pi_{k}^0$ formula:
\begin{itemize}
    \item $c \Vdash (\exists x)\phi_e(G,x)$ if $c \Vdash \phi(G,a)$ for some~$a \in \NN$;
    \item $c \Vdash (\forall x)\neg\phi(G,x)$ if $(\forall \rho  \subseteq X_n)(\forall x \in \NN) \sigma \cup \rho \qvdash \neg \phi(G,x)$.
\end{itemize}
\end{definition}

Intuitively, a $\Pi^0_{k+1}$-formula $(\forall x)\neg \phi(G, x)$ can be seen as a countable collection $\{\neg \phi(G, x) : x \in \NN \}$ of~$\Sigma^0_k$-formulas. Assuming that the forcing question meets its requirements, the forcing relation for $\Pi^0_{k+1}$ formulas is a density statement, saying that for every $\Sigma^0_k$-formula $\neg \phi(G, x)$, whatever the extension of the stem~$\sigma$ into~$\sigma \cup \rho$, since $\sigma \cup \rho \qvdash \neg \phi(G,x)$, there is an extension forcing $\neg \phi(G, x)$. Thus, if~$\F$ is a sufficiently generic filter containing a condition~$c = (\sigma, X_n)$ such that $\sigma \qvdash (\forall x)\neg \phi(G,x)$, then for every $x \in \NN$, there will be a condition in~$\F$ forcing $\neg \phi(G, x)$.

Also note that the forcing relation for $\Pi^0_k$ formulas might seem too strong, as one would expect to require $\rho$ to range over~$X_n \cap A$ instead of~$X_n$. We shall see, thanks to \Cref{lem:core-question-find-extension}, that it is still dense to force either a formula or its negation.
Together with \Cref{lem:core-forcing-closure,prop:core-forcing-imply-truth}, this shows that the forcing relation satisfies the axioms of~\Cref{def:forcing-relation}.

\begin{lemma}\label{lem:core-forcing-closure}
    The forcing relations are closed downwards.
\end{lemma}

\begin{proof}
    Let $c = ( \sigma,X_n)$ and $d = (\tau,Y_n)$ be two $\PP_n^A$-conditions, such that $d \leq c$.
Let $\phi(G, x)$ be a $\Pi^0_0$-formula.
    \begin{itemize}
        \item If $c \vdash (\exists x)\phi(G, x)$, then $\phi(\sigma,a)$ holds for some $a \in \NN$ hence $\phi(\tau,a)$ holds and thus $d \Vdash (\exists x)\phi(G, x)$.
        \item If $c \vdash (\forall x)\neg \phi(G, x)\varphi(G)$, then $(\forall \rho \subseteq X_n)(\forall x \in \NN)\neg \phi(\sigma \cup \rho,x)$, hence in particular, letting $\mu$ be such that $\sigma \cdot \mu = \tau$,  $(\forall \rho \subseteq Y_n)(\forall x \in \NN)\neg \phi(\sigma \cup \mu \cup \rho,x)$, hence $d \Vdash (\forall x)\neg \phi(G, x)$.
    \end{itemize}
Now, let $\phi(G, x)$ be a $\Pi^0_k$-formula for $1 \leq k \leq n$:
    \begin{itemize}
        \item If $c \vdash (\exists x)\phi(G, x)$, then $c \Vdash \phi(G,a)$ for some $a \in \NN$. By induction hypothesis, $d \Vdash \phi(G,a)$, and thus $d \Vdash (\exists x)\phi(G, x)$.
         \item If $c \vdash (\forall x)\neg \phi(G, x)$, then $(\forall \rho \subseteq X_n)(\forall x \in \NN)\sigma \cup \rho \qvdash \neg \phi(G,x)$, hence in particular, letting $\mu$ be such that $\sigma \cdot \mu = \tau$, $(\forall \rho \subseteq Y_n)(\forall x \in \NN) \sigma \cup \mu \cup \rho \qvdash \neg \phi(G,x)$, hence $d \Vdash (\forall x)\neg \phi(G, x)$.
    \end{itemize}
\end{proof}

In order to prove iterated lowness basis theorems, one will need to construct generic filters effectively. For this, it is necessary to fix a finite representation of $\PP^A_n$-condition to talk about effectivity.
Recall that an \emph{$M_n$-index} of a set $X \in \M_n$ is an integer~$a \in \NN$ such that $X = Z_a$. Note that $M_n$-indices are not unique.

\begin{definition}
An \emph{index} of a $\PP^A_n$-condition $(\sigma, X_n)$ is a pair $\langle \sigma, a\rangle$ where~$a$ is an $M_n$-index of~$X_n$.
\end{definition}

The following lemma states that the forcing question meets its specifications, that is, each case is witnessed by an extension forcing the answer.

\begin{lemma}[{\cite[Lemma 3.8]{monin2021weakness}}]\label[lemma]{lem:core-question-find-extension}
Let $c = (\sigma, X_n)$ be a $\PP^A_n$-condition and $\phi(G,x)$ be a $\Pi_k^0$ formula for $k \leq n$.
\begin{itemize}
    \item If $\sigma \qvdash (\exists x) \phi(G,x)$, then there exists $d \leq c$ such that $d \Vdash (\exists x) \phi(G,x)$.
    \item If $\sigma \nqvdash (\exists x) \phi(G,x)$, then there exists $d \leq c$ such that $d \Vdash (\forall x) \neg \phi(G,x)$.
\end{itemize}
Deciding which case holds requires $M_{k+1}'$, then given the answers to the forcing question, finding an index of the extension can be done uniformly in an index of~$c$ using $A \oplus M_{n+1}$.
\end{lemma}

\begin{proof}
Suppose $k=0$.
Let $\U(\sigma,\phi)$ denote $\{X : (\exists \rho \subseteq X )(\exists x)\phi(\sigma \cup \rho, x) \}$. 
\begin{itemize}
    \item If $c \qvdash (\exists x) \phi(G,x)$, then the class $\U_{C_{0}}^{\M_{0}} \cap \U(\sigma,\phi)$ is large. By definition of the $\langle \cdot \rangle$ operator, this yields that  $\U_{C_{0}}^{\M_{0}} \cap \U(\sigma,\phi) \supseteq \langle \U_{C_{0}}^{\M_{0}} \rangle \supseteq \langle \U_{C_n}^{\M_n} \rangle$.

    By \Cref{lem:core-reservoir-intersection}, $A \cap X_n \in \langle \U_{C_n}^{\M_n}\rangle$, hence there exists some $\rho \subseteq A \cap X_n$ such that $(\exists x)\phi(\sigma \cup \rho, x)$ holds. The $\PP^A_n$-condition $d = (\sigma \cup \rho, X_n \setminus \{0, \dots, |\rho|-1\})$ is such that $d \leq c$ and $d \Vdash (\exists x)\phi(G,x)$.

    \item 
    If $c \nqvdash (\exists x) \phi(G,x)$, then the class $\U_{C_{0}}^{\M_{0}} \cap \U(\sigma,\phi)$ is not large. 
 This yields the existence of a finite set $F \subseteq C_0$ such that $\U_F^{\M_0} \cap \U(\sigma,\phi)$ is not large. As such, there is some $\ell \in \NN$ and some $(\ell+1)$-partition $Z_0, \dots Z_\ell$ of $\NN$ such that for all $i \leq \ell$, $Z_i \not \in \U_F^{\M_0} \cap \U(\sigma,\phi)$. Consider the $\Pi_1^0(\M_0)$ class of such partitions. Since $\M_0$ is a Scott ideal, there is such a partition in $\M_0$. Consider such a partition $Z_0, \dots Z_\ell$ for $\ell \in \NN$. Since $\langle \U_{C_n}^{\M_n} \rangle$ is partition regular, and as $X_n \in \langle \U_{C_n}^{\M_n} \rangle$, there exists $i \leq \ell$ such that $X_n \cap Z_i \in  \langle \U_{C_n}^{\M_n} \rangle $, and thus $Z_i \in \langle \U_{C_n}^{\M_n} \rangle $.  Moreover, since $\langle \U_{C_n}^{\M_n} \rangle \subseteq \U_{C_0}^{\M_0} \subseteq \U_F^{\M_0}$, we have that $Z_i \not \in \U(\sigma,\phi)$, and thus $X_n \cap Z_i \not \in \U(\sigma,\phi)$.  The $\PP^A_n$-condition $d = (\sigma, X_n \cap Z_i)$ is such that $d \leq c$ and $d \Vdash (\forall x) \neg \phi(G,x)$.
\end{itemize}
Suppose $0 < k \leq n$.
Let $\U(\sigma,\phi)$ denote $\{X : (\exists \rho \subseteq X )(\exists x)\sigma \cup \rho \nqvdash \neg \phi(G,x)  \}$. 
\begin{itemize}
     \item If $c \qvdash (\exists x) \phi(G,x)$, then the class $\U_{C_{k}}^{\M_{k}} \cap \U(\sigma,\phi)$ is large. By definition of the $\langle \cdot \rangle$ operator, this yields that  $\U_{C_{k}}^{\M_{k}} \cap \U(\sigma,\phi) \supseteq \langle \U_{C_{k}}^{\M_{k}} \rangle \supseteq \langle \U_{C_n}^{\M_n} \rangle $.

     By \Cref{lem:core-reservoir-intersection}, $A \cap X_n \in \langle \U_{C_n}^{\M_n}\rangle$, hence there exists some $\rho \subseteq X_n \cap A$ and some $a \in \NN$ such that $\sigma \cup \rho \nqvdash \neg \phi(G,a)$. By induction hypothesis, there exists $ d \leq (\sigma \cup \rho, X_n \setminus \{0, \dots, |\rho|-1\}) \leq c$  such that $d \Vdash \phi(G,a)$, hence, $d \Vdash (\exists x) \phi(G,x)$. 

    \item 
    This is proved the exact same way as in the $k =0$ case, after redefining $\U(\sigma,\phi)$.
\end{itemize}
Regarding the complexity of finding such $d$, the main difficulty is finding, given an $(\ell+1)$-partition $Z_0, \dots, Z_\ell$ in $\M_n$, some $i\leq \ell$ such that $Z_i \in \langle \U_{C_n}^{\M_n} \rangle$. By \Cref{lem:complexity-finding-among-partition}, this can be done uniformly in $C_n \oplus M_n'$.
\end{proof}

Genericity is usually defined with respect to a collection of dense sets of conditions. We define a hierarchy of notions of genericity depending on the formulas which are decided by the filter.

\begin{definition}[Generic filter]
We say that a $\PP^A_n$-condition $c$ \emph{decides} a formula~$\varphi(G)$ if either~$c \Vdash \varphi(G)$, or $c \Vdash \neg \varphi(G)$. A filter~$\F$ \emph{decides} a formula~$\varphi(G)$ if it contains a $\PP^A_n$-condition deciding $\varphi(G)$.

A filter~$\F$ is \emph{$k$-generic} for $1 \leq k \leq n+1$ if it decides every $\Sigma^0_1, \dots, \Sigma^0_k$ formulas.
\end{definition}

Thanks to \Cref{lem:core-question-find-extension}, for every $k \leq n$ and every $\Sigma^0_{k+1}$-formula~$\varphi(G)$, it is dense to either force~$\varphi(G)$ or force $\neg \varphi(G)$. It follows that every sufficiently generic filter is $(n+1)$-generic. Given a filter $\F$, we write $[\F]$ for~$\bigcap_{(\sigma, X_n) \in \F} [\sigma,X_n]$.

\begin{proposition}\label[proposition]{prop:core-forcing-generic-set}
Let $\F$ be a $1$-generic filter. There exists a unique set $G_{\F}$ belonging to $[\F]$. Furthermore, $G_{\F}$ is infinite.
\end{proposition}
\begin{proof}
For every condition $(\sigma,X_n)$ in $\F$, the cylinder $[\sigma, X_n]$ is a non-empty closed subset of $2^{\NN}$. Assume by contradiction that $[\F] = \emptyset$, then as $2^{\NN}$ is compact, there exists a finite subfamily $E$ of $\F$ such that $\bigcap_{(\sigma, X_n) \in E} [\sigma, X_n] = \emptyset$. By compatibility of the conditions in a filter, there is a condition $c$ extending every member of this family, which yields that the cylinder under $c$ is empty, contradiction. Therefore, $[\F]$ is non-empty. 

Let $G$ be an element of $[\F]$. For every $\sigma \prec G$, there exists a condition $c \in \F$ forcing the $\Sigma_1^0$ property $\sigma \prec G$, otherwise, by $1$-genericity of $\F$, there would exist a condition $d$ in $\F$ forcing $\sigma \not \prec G$ and every element in the cylinder under $d$ (including $G$) would therefore not have $\sigma$ as a prefix, contradicting $\sigma \prec G$. All the elements of $[\F]$ share the same prefixes as $G$, hence $G$ is the only element.

By \Cref{lem:core-reservoir-intersection}, $X_n \cap A$ is in $\langle \U_{C_n}^{\M_n} \rangle$ and therefore, $X_n \cap A$ is infinite by hypothesis on $\langle \U_{C_n}^{\M_n} \rangle$. Hence, for every $s \in \NN$, it is not possible to force the property $|G| \leq s$, as for every condition $c = (\sigma,X_n)$, the cylinder under $c$ will contain infinite sets such as $\sigma \cup (X_n \cap A)$. Hence, for every $s \in \NN$, there exists a condition in $\F$ forcing $|G| > s$ and $G_{\F}$ is infinite.
\end{proof}

As mentioned, the semantic definition of the forcing relation states that a condition~$c$ forces a formula~$\varphi(G)$ if $\varphi(G_\F)$ holds for every sufficiently generic filter containing~$c$. The following lemma gives a quantitative refinement of this definition by stating for every $k$-generic filter~$\F$, if a condition $c \in \F$ forces a $\Sigma^0_{k+1}$ or a $\Pi^0_{k+1}$-formula~$\varphi(G)$, then $\varphi(G_\F)$ holds. It shows in particular that the syntactic forcing relation is sound, in that it implies the semantic forcing relation.

\begin{proposition}\label[proposition]{prop:core-forcing-imply-truth}
Let $\F$ be a $k$-generic filter for some $1 \leq k \leq n$. If there exists some $c \in \F$ such that $c \Vdash (\forall x) \phi(G,x)$ for some $\Sigma_k^0$ formula $\phi(G,x)$, then for every $a \in \NN$, there exists some $d \in \F$ such that $d \Vdash \phi(G,a)$. Furthermore, $(\forall x)\phi(G_{\F},x)$ will hold.   
\end{proposition}

\begin{proof}
Proceed by strong induction on $k$:

Let $1 \leq k \leq n$ and assume the property to be true for every $\ell < k$. Let $\F$ be $k$-generic, let $c = (\sigma, X_n) \in \F$ be such that $c \Vdash (\forall x) \phi(G,x)$ for some $\Sigma_{k}^0$ formula $\phi(G,x)$ and let $a \in \NN$. By $k$-genericity of $\F$, there exists some $d = (\tau,Y_n) \in \F$ such that $d \Vdash \phi(G,a)$ or $d \Vdash \neg \phi(G,a)$. By definition of a filter and by downward closure of the forcing relation, we can assume that $d \leq c$, hence $\tau \succeq \sigma$. Assume for contradiction that $d \Vdash \neg \phi(G,a)$. Since $c \Vdash (\forall x)\phi(G,x)$, there are two cases depending on the value of $k$:
\begin{itemize}
    \item If $k = 1$, then, for all $\rho \subseteq X_n$ and all $x \in \NN$, $\sigma \cup \rho \qvdash \phi(G,x)$ holds, hence $\tau \qvdash \phi(G,a)$ and the following class is large (and therefore includes $\langle \U_{C_n}^{\M_n} \rangle \subseteq \langle \U_{C_{0}}^{\M_{0}} \rangle$):
$$\U_{C_{0}}^{\M_{0}} \cap \{X : (\exists \rho \subseteq X )\phi(\tau \cup \rho, a) \}$$
    
    As $Y_n \in \langle \U_{C_n}^{\M_n} \rangle$, there exists some $\rho \subseteq Y_n$ such that $\phi(\tau \cup \rho, a)$ holds, contradicting $d \Vdash \neg \phi(G,a)$. 
    \item If $1 < k \leq n$, then, write $\phi(G,x) = (\exists y)\psi(G,x,y)$ for some $\Pi_{k-1}^0$ formula $\psi(G,x,y)$. Then, for all $\rho \subseteq X_n$ and all $x \in \NN$, $\sigma \cup \rho \qvdash (\exists y)\psi(G,x,y)$ holds, hence $\tau \qvdash (\exists y) \psi(G,a,y)$ holds, and the following class is large (and therefore includes $\langle \U_{C_n}^{\M_n} \rangle \subseteq \langle \U_{C_{k-1}}^{\M_{k-1}} \rangle$):
$$\U_{C_{k-1}}^{\M_{k-1}} \cap \{X : (\exists \rho \subseteq X )(\exists y) \tau \cup \rho \nqvdash \neg \psi(\tau \cup \rho, a,y) \}$$

    As $Y_n \in \langle \U_{C_{n}}^{\M_{n}} \rangle$, there exists some $\rho\subseteq Y_n$ and some $b \in \NN$ such that $\tau \cup \rho \nqvdash \neg \psi(\tau \cup \rho, a,b)$ holds, contradicting $d \Vdash (\forall y) \neg \psi(G,a, y)$.
\end{itemize}
Therefore, $d \Vdash \phi(G,a)$ holds. If $k = 1$, then $\phi(G_{\F},a)$ holds by definition of the forcing relation for $\Sigma_1^0$ formulas, hence $(\forall x)\phi(G_{\F},x)$ will hold. If $k > 1$, then, writing $\phi(G,x)$ as $(\exists y)\psi(G,x,y)$ for some $\Pi_{k-1}^0$ formula $\psi(G,x,y)$ yields that $d \Vdash \psi(G,a,b)$, hence $\psi(G_{\F},a,b)$ holds for some $b \in \NN$ (either by the inductive hypothesis applied with $\ell = k-2$ if $k > 2$, or if $k = 2$, by definition of the forcing relation), thus $(\forall x)\phi(G_{\F},x)$ will hold. 
\end{proof}

\section{Iterated lowness basis}\label{sect:iterated-lowness}

The goal of this section is to prove the following iterated basis theorem, which, besides its intrinsic interest, will serve to separate $\Sub{\Sigma^0_n}$ from $\Sub{\Delta^0_{n+1}}$ over $\omega$-models:

\begin{repmaintheorem}{thm:main-weakly-low-basis}
Fix~$n \geq 1$. For every $\Sigma^0_{n+1}$ set~$A$ and every set~$Q$ of PA degree over~$\emptyset^{(n)}$, there is an infinite set~$H \subseteq A$ or $H \subseteq \overline{A}$ such that $H^{(n)} \leq_T Q$.
\end{repmaintheorem}

Before proving \Cref{thm:main-weakly-low-basis}, we deduce a few immediate consequences.

\begin{corollary}\label{cor:main-low-basis}
Fix~$n \geq 0$. For every $\Sigma^0_{n+1}$ set~$A$, there is an infinite set~$H \subseteq A$ or $H \subseteq \overline{A}$ of low${}_{n+1}$ degree.
\end{corollary}
\begin{proof}
For $n = 0$, every $\Sigma^0_1$ set~$A$ is either infinite, in which case it admits an infinite computable subset, or $A$ is finite, and therefore $\overline{A}$ is an infinite computable solution. In particular, every computable set is low. Suppose $n \geq 1$.
By the low basis theorem relativized to $\emptyset^{(n)}$ (\cite{jockusch1972classes}), there is a set~$Q$ of PA degree over~$\emptyset^{(n)}$ such that $Q' \leq_T \emptyset^{(n+1)}$. By \Cref{thm:main-weakly-low-basis}, there is an infinite set~$H \subseteq A$ or $H \subseteq \overline{A}$ such that $H^{(n)} \leq_T Q$. In particular, $H^{(n+1)} \leq_T Q' \leq_T \emptyset^{(n+1)}$, so $H$ is of low${}_{n+1}$ degree.
\end{proof}

The restriction of \Cref{thm:main-weakly-low-basis} to $\Delta^0_{n+1}$ sets was proven by Cholak, Jockusch and Slaman~\cite{cholak2001strength} for the case~$n = 1$ and in the general case by Monin and Patey~\cite{monin2021weakness}.
The case $n = 1$ of \Cref{cor:main-low-basis} was proven by Benham et al.~\cite{benham2024ginsburg}. \Cref{cor:main-low-basis} can be used to separate $\Sub{\Sigma^0_n}$ from $\Sub{\Delta^0_{n+1}}$ over $\omega$-models.

\begin{corollary}
Fix~$n \geq 1$. Then there exists an $\omega$-model of~$\RCA_0 + \Sub{\Sigma^0_n}$
which is not a model of~$\Sub{\Delta^0_{n+1}}$.
\end{corollary}
\begin{proof}
By Downey, Hirschfeldt, Lempp and Solomon~\cite{downey2001delta2} relativized to $\emptyset^{(n-1)}$,
there exists a $\Delta^0_{n+1}$ set $B$ with no infinite subset~$H \subseteq B$ or~$H \subseteq \overline{B}$ such that $H' \leq_T \emptyset^{(n)}$. In particular, $B$ is a computable instance of $\Sub{\Delta^0_{n+1}}$ with no solution of low${}_n$ degree. 

Consider a chain $\N_0 \subseteq \N_1 \subseteq \dots$ of countable $\omega$-models of $\RCA_0$ such that 
\begin{enumerate} 
    \item $\N_i$ is topped by some set $D_i$ of low${}_n$ degree ; 
    \item For every $\Sigma^0_n(\N_i)$ set $A$, there exists some $j \in \NN$ and some infinite~$H \in \N_j$ such that $H \subseteq A$ or $H \subseteq \overline{A}$.
\end{enumerate}

This chain can be constructed recursively as follows: let $\N_0$ be the $\omega$-model whose second order part are the computable sets and assuming $\N_0, \dots \N_a$ have been defined, let $\langle k,\ell \rangle = a$ and consider $A$ the $k$-th $\Sigma_{n}^0(D_{\ell})$ set. The set $A$ is also $\Sigma_{n}^0(D_{a})$, thus using a relativized version of \Cref{cor:main-low-basis}, there exists an infinite set $H \subseteq A$ or $H \subseteq \overline{A}$ of low${}_{n}$ degree relative to $D_a$. Let $D_{a+1} = D_a \oplus H$ that is therefore of low${}_{n}$ degree and let $\N_{a+1}$ be comprised of all the sets computable by $D_{a+1}$. \\

Since a union of an increasing sequence of Turing ideals is again a Turing ideal, $\N = \bigcup_{n \in \NN} \N_i$ is again a Turing ideal, $\N \models \RCA_0$. By item 2, $\N \models \Sub{\Sigma^0_n}$, as every instance of $\Sub{\Sigma^0_n}$ in $\N$ belongs to some~$\N_i$, and therefore has a solution in some $\N_j \subseteq \N$.
Last, by item 1, $\N$ contains only sets of low${}_n$ degree, hence, $\N$ contains no solution to~$B$ seen as an instance of $\Sub{\Delta^0_{n+1}}$. It follows that $\N \not \models \Sub{\Delta^0_{n+1}}$.
\end{proof}
\bigskip

The remainder of the section is dedicated to the proof of \Cref{thm:main-weakly-low-basis}. Fix~$n \geq 1$.
Let $\M_0, \dots, \M_n$ be Scott ideals with Scott codes $M_0, \dots, M_n$, respectively, and let $C_0, \dots, C_{n-1}$ be forming a largeness tower, that is, for every $i < n$:
\begin{itemize}
    \item $\U_{C_i}^{\M_i}$ is an $\M_i$-cohesive large class containing only infinite sets ;
    \item $C_i,M_i' \in \M_{i+1}$ ;
    \item $\U_{C_{i+1}}^{\M_{i+1}} \subseteq \langle \U_{C_i}^{\M_i} \rangle$.
\end{itemize}
Fix a $\Sigma^0_{n+1}$ set~$A$. 
By partition regularity of $\langle \U_{C_{n-1}}^{\M_{n-1}} \rangle$, either~$A$ or $\overline{A} \in \langle \U_{C_{n-1}}^{\M_{n-1}} \rangle$, and maybe both. 
Depending on which case holds, one will construct an infinite subset of~$A$ or~$\overline{A}$. However, there is some asymmetry in the constructions, as $A$ is $\Sigma^0_{n+1}$, while $\overline{A}$ is $\Pi^0_{n+1}$. Intuitively, it is easier to build subsets of~$A$, as belonging to~$A$ is witnessed by a $\emptyset^{(n)}$-c.e. process. Thus, if both $A$ and $\overline{A}$ belong to $\langle \U_{C_{n-1}}^{\M_{n-1}} \rangle$, we will rather construct an infinite subset of~$A$. It follows that we will construct a subset of~$\overline{A}$ only if $A \not \in \langle \U_{C_{n-1}}^{\M_{n-1}} \rangle$. If so, it is witnessed by a large $\Sigma^0_1(\M_{n-1})$ class $\U \supseteq \langle \U_{C_{n-1}}^{\M_{n-1}} \rangle$ such that $A \not \in \U$. We will then exploit this witness to construct an infinite subset of~$\overline{A}$. 

We now present our two notions of forcing, called \emph{main forcing} and \emph{witness forcing}, to build an infinite subset of~$A$ in the former case, and of $\overline{A}$ in the latter case.


\subsection{Main forcing}\label{sect:main-forcing}

Throughout this section, fix a $\Sigma^0_{n+1}$ set~$A \in \langle \U_{C_{n-1}}^{\M_{n-1}} \rangle$. The notion of forcing is parameterized by the set~$A$.

\begin{definition}[Condition]
Let~$\MM^A_n$ be the $\PP^A_{n-1}$ notion of forcing.
\end{definition}

The $\MM^A_n$-forcing coincides with the $\PP^A_{n-1}$-forcing as a partial order, but one will exploit the $\Sigma^0_{n+1}$ extra-hypothesis on~$A$ to define a $\Sigma^0_{n+1}$-preserving forcing question for~$\Sigma^0_{n+1}$-formulas. Recall that $\PP^A_n$-forcing admits a forcing question for~$\Sigma^0_{n+1}$-formulas with a bad definitional complexity.

The forcing relation for $\MM^A_n$-forcing inherits the forcing relation from $\PP^A_{n-1}$-forcing. However, $\PP^A_{n-1}$-forcing does not define any forcing relation for $\Sigma^0_{n+1}$ and $\Pi^0_{n+1}$-formulas.
We shall actually define a slightly different forcing relation by making $\rho$ range over $X_{n-1} \cap A$ rather than $X_{n-1}$. This difference will be justified by the design of the forcing question for $\Sigma^0_{n+1}$-formulas.

\begin{definition}[Forcing relation]\label[definition]{def:main-forcing-relation}
Let $c = (\sigma,X_{n-1})$ be an $\MM_{n}^A$-condition, we define the forcing relation $\Vdash$ for $\Sigma_{n+1}^0$ and $\Pi_{n+1}^0$ formulas as follows: 
For $\phi(G,x)$ a $\Pi_n^0$ formula:
\begin{itemize}
    \item $c \Vdash (\exists x)\phi(G,x)$ if $c \Vdash \phi(G,a)$ for some~$a \in \NN$;
    \item $c \Vdash (\forall x)\neg\phi(G,x)$ if for every $\rho \subseteq X_{n-1} \cap A$ and every $a \in \NN$, $\sigma \cup \rho \qvdash \neg \phi(G, a)$.
\end{itemize}
\end{definition}

As expected, the forcing relations for $\Sigma^0_{n+1}$ and $\Pi^0_{n+1}$ formulas satisfy the axioms of \Cref{def:forcing-relation}. The following lemma extends \Cref{lem:core-forcing-closure} to $\Sigma^0_{n+1}$ and $\Pi^0_{n+1}$ formulas.

\begin{lemma}\label[lemma]{lem:main-forcing-closed-downwards}
    The forcing relations are closed downwards.
\end{lemma}

\begin{proof}
    Let $c = ( \sigma,X_{n-1})$ and $d = (\tau,Y_{n-1})$ be two $\MM_n^A$-conditions, such that $d \leq c$. Let $\phi(G, x)$ be a $\Pi^0_n$-formula.
      \begin{itemize}
        \item If $c \Vdash (\exists x)\phi(G, x)$, then $c \Vdash \phi(G,a)$ for some $a \in \NN$. By \Cref{lem:core-forcing-closure}, $d \Vdash \phi(G,a)$, and thus $d \Vdash (\exists x)\phi(G, x)$.
         \item If $c \Vdash (\forall x)\neg \phi(G, x)$, then for all $\rho \subseteq X_{n-1} \cap A$ and for all $x \in \NN$, $\sigma \cup \rho \qvdash \neg \phi(G,x)$. This yields in particular, letting $\mu$ be such that $\sigma \cdot \mu = \tau$, for all $\rho \subseteq Y_{n-1} \cap A$ and for all $x \in \NN$, $\sigma \cup \mu \cup \rho \qvdash \neg \phi(G,x)$, hence $d \Vdash (\forall x)\neg \phi(G, x)$.
    \end{itemize}
\end{proof}

A forcing question is \emph{extremal} if it is one of the two canonical forcing questions (see discussion after \Cref{def:abstract-forcing-question}). In other words, a forcing question for $\Sigma^0_{n+1}$-formulas is extremal if either it coincides with the forcing relation for $\Sigma^0_{n+1}$-formulas, or it coincides with the negation of the forcing relation of $\Pi^0_{n+1}$-formulas. Contrary to the lower levels, the forcing question for $\Sigma^0_{n+1}$-formulas is extremal.

\begin{definition}[Forcing question for $\Sigma^0_{n+1}$ formulas]
Let $c = (\sigma,X_{n-1})$ be an $\MM^A_n$-condition.
For $\phi(G,x)$ a $\Pi_n^0$ formula, write $c \qvdash (\exists x)\phi(G,x)$ if $c \not \Vdash (\forall x)\neg \phi(G, x)$, that is, there exist some $\rho \subseteq X_{n-1} \cap A$ and some $a \in \NN$ such that $\sigma \cup \rho \nqvdash \neg \phi(G, a)$.
\end{definition}

The following lemma shows that the forcing question for $\Sigma^0_{n+1}$-formulas has the right definitional property. In particular, if $M_n$ is low over~$\emptyset^{(n)}$, that is, $M_n' \leq_T \emptyset^{(n+1)}$, then one can decide a $\Sigma^0_{n+1}$-formula using~$\emptyset^{(n+1)}$.

\begin{lemma}\label[lemma]{lem:main-complexity-sigma0n+1-question}
The statement $c \qvdash (\exists x)\phi(G,x)$ for $\phi(G)$ a $\Pi_{n}^0$ formula is $\Sigma^0_1(\M_n)$.
\end{lemma}

\begin{proof}
The statement $c \qvdash (\exists x) \phi(G,x)$ is equivalent to $(\exists \rho \in 2^{<\NN})(\exists x \in \NN)(\rho \subseteq X_{n-1} \cap A \wedge \sigma \cup \rho \nqvdash \neg \phi(G,x))$.

The statement $\rho \subseteq X_{n-1} \cap A$ is $\Sigma_1^0(M_{n-1}')$ since $A$ is $\Sigma_{n+1}^0$ (and $M_{n-1}$ contains $\emptyset^{(n - 1)}$) and $X_{n-1} \in \M_{n-1}$ and the statement $\sigma \cup \rho \nqvdash \neg \phi(G,x)$ is $\Sigma^0_1(\M_n)$ by \Cref{lem:core-complexity-question}.
\end{proof}




The following lemma states that the forcing question at the last level meets its specifications, that is, each answer is witnessed by an extension forcing it.

\begin{lemma}\label[lemma]{lem:main-question-find-extension}
Let $c = (\sigma, X_{n-1})$ be an $\MM^A_n$-condition and $\phi(G,x)$ be a $\Pi_n^0$ formula.
\begin{itemize}
    \item If $c \qvdash (\exists x) \phi(G,x)$, then there exists $d \leq c$ such that $d \Vdash (\exists x) \phi(G,x)$.
    \item If $c \nqvdash (\exists x) \phi(G,x)$, then there exists $d \leq c$ such that $d \Vdash (\forall x) \neg \phi(G,x)$.
\end{itemize}
Deciding which case holds and finding an index of the extension can be done uniformly in an index of $c$ using $M_n'$.
\end{lemma}

\begin{proof}\ 
\begin{itemize}
\item If $c \qvdash (\exists x) \phi(G,x)$, there exists some $\rho \subseteq X_{n-1} \cap A$ and some $a \in \NN$ such that $\sigma \cup \rho \nqvdash \neg \phi(G,a)$. Therefore, by \Cref{lem:core-question-find-extension}, there exists some $d \leq (\sigma \cup \rho, X_{n-1} \setminus \{0, \dots, |\rho|-1\}) \leq c$ such that $d \Vdash \phi(G,a)$, hence $d \Vdash (\exists x)\phi(G,x)$.

\item If $c \nqvdash (\exists x) \phi(G,x)$, then already $c \Vdash (\forall x) \neg \phi(G,x)$. \\

\end{itemize}
Regarding the complexity of finding such $d$, deciding whether $c \qvdash (\exists x) \phi(G,x)$ holds can be done using $M_n'$ by \Cref{lem:main-complexity-sigma0n+1-question} and in the case where $c \qvdash (\exists x) \phi(G,x)$, finding the extension $d$ can be done using $A \oplus M_n$ by \Cref{lem:core-question-find-extension}, hence using $M_n'$ since $A$ is $\emptyset^{(n+1)}$ computable and $\M_n$ contains $\emptyset^{(n)}$.

\end{proof}





The following lemma is the counterpart of \Cref{prop:core-forcing-imply-truth} and essentially states that the syntactic forcing relation implies the semantic forcing relation, with an explicit bound to the amount of genericity for it to hold.

\begin{proposition}\label[proposition]{prop:main-forcing-imply-truth}
Let $\F$ be an $n$-generic $\MM^A_n$-filter. If there exists some $c \in \F$ such that $c \Vdash (\forall x) \phi(G,x)$ for some $\Sigma_n^0$ formula $\phi(G,x)$, then for every $a \in \NN$, there exists some $d \in \F$ such that $d \Vdash \phi(G,a)$. Furthermore, $(\forall x)\phi(G_{\F},x)$ will hold.   
\end{proposition}

\begin{proof}
Let $\F$ be $n$-generic, let $c = (\sigma, X_{n-1}) \in \F$ be such that $c \Vdash (\forall x) \phi(G,x)$ for some $\Sigma_n^0$ formula $\phi(G,x)$ and let $a \in \NN$. By $n$-genericity of $\F$, there exists some $d = (\tau,Y_{n-1}) \in \F$ such that $d \Vdash \phi(G,a)$ or $d \Vdash \neg \phi(G,a)$. By definition of a filter and by downward closure of the forcing relation, we can assume that $d \leq c$, hence $\tau \succeq \sigma$. Assume for contradiction that $d \Vdash \neg \phi(G,a)$. Since $c \Vdash (\forall x)\phi(G,x)$, there are two cases depending on the value of $n$:
\begin{itemize}
    \item If $n = 1$, then, for all $\rho \subseteq X_0 \cap A$ and all $x \in \NN$, $\sigma \cup \rho \qvdash \phi(G,x)$ holds, hence $\tau \qvdash \phi(G,a)$ and the following class is large (and therefore includes $\langle \U_{C_0}^{\M_0} \rangle$):
$$\U_{C_{0}}^{\M_{0}} \cap \{X : (\exists \rho \subseteq X )\phi(\tau \cup \rho, a) \}$$
    
    As $Y_0 \in \langle \U_{C_0}^{\M_0} \rangle$, there exists some $\rho \subseteq Y_0$ such that $\phi(\tau \cup \rho, a)$ holds, contradicting $d \Vdash \neg \phi(G,a)$. 
    
    \item If $n > 1$, then, write $\phi(G,x) = (\exists y)\psi(G,x,y)$ for some $\Pi_{n-1}^0$ formula $\psi(G,x,y)$. Then, for all $\rho \subseteq X_{n-1} \cap A$ and all $x \in \NN$, $\sigma \cup \rho \qvdash (\exists y)\psi(G,x,y)$ holds hence $\tau \qvdash (\exists y) \psi(G,a,y)$ holds and the following class is large (and therefore includes $\langle \U_{C_{n-1}}^{\M_{n-1}} \rangle$):
$$\U_{C_{n-1}}^{\M_{n-1}} \cap \{X : (\exists \rho \subseteq X)(\exists y) \tau \cup \rho \nqvdash \neg \psi(\tau \cup \rho, a,y) \}$$

    As $Y_{n-1} \cap A \in \langle \U_{C_{n-1}}^{\M_{n-1}} \rangle$, there exists some $\rho \subseteq Y_{n-1} \cap A$ and some $b \in \NN$ such that $\tau \cup \rho \nqvdash \psi(\tau \cup \rho, a,b)$ holds, contradicting $d \Vdash (\forall y) \neg \psi(G,a)$. 
\end{itemize}

Therefore, $d \Vdash \phi(G,a)$ holds. If $n = 1$, then $\phi(G_{\F},a)$ holds by definition of the forcing relation for $\Sigma_1^0$ formulas, hence $(\forall x)\phi(G_{\F},a)$ will hold. If $n > 1$, then, writing $\phi(G,x)$ as $(\exists y)\psi(G,x,y)$ for some $\Pi_{n-1}^0$ formula $\psi(G,x,y)$ yields that $d \Vdash \psi(G,a,b)$, hence $\psi(G_{\F},a,b)$ holds for some $b \in \NN$ (either by \Cref{prop:core-forcing-imply-truth} if $k > 2$, or if $k = 2$, by definition of the forcing relation), thus $(\forall x)\phi(G_{\F},a)$ will hold. 
\end{proof}

We are now ready to prove our first abstract construction. When considering a set $M_n$ such that $M_n' \leq_T \emptyset^{(n+1)}$, it states the existence of an infinite subset of low${}_{n+1}$ degree.

\begin{proposition}\label[proposition]{prop:main-forcing-lown+1}
There exists an infinite subset $H \subseteq A$ such that $H^{(n+1)} \leq_T M_n'$.
\end{proposition}

\begin{proof}
Note that $\MM^A_n \neq \emptyset$, since $(\epsilon,\NN) \in \MM^A_n$, with $\epsilon$ denoting the empty sequence. We can build effectively in $M_n'$ an $(n+1)$-generic filter $\F$ such that the corresponding set $G_{\F}$ will be a subset of $A$ such that $G_{\F}^{(n+1)} \leq M_n'$. More precisely, we will build a uniformly $M_n'$-computable decreasing sequence of conditions 
$$
c_0 \geq c_1 \geq c_2 \geq \dots
$$
where $c_0 = (\epsilon, \NN)$, and for every~$s \in \NN$, if $s = \langle e,k \rangle$ with $k \leq n$, then $c_{s+1}$ decides $(\exists x) \phi_e^k(G,x)$ for $(\phi_e^k(G,x))_{e \in \NN}$ a computable list of all the $\Pi_k^0$ formulas with parameters $G$ and $x$. Then, the set~$\F = \{ d \in \MM^A_n : (\exists s) d \geq c_s \}$ will be a $(n+1)$-generic filter. \\

Let $s \in \NN$ and assume $c_{s} = (\sigma, X_{n-1})$ has already been defined. 

If $s = \langle e,k \rangle$ for some $k \leq n$, then, thanks to \Cref{lem:core-question-find-extension} in the case where $k < n$ and thanks to \Cref{lem:main-question-find-extension} in the case where $k = n$, uniformly in $M_n'$ we can find an extension $c_{s+1}\leq c_{s}$ such that $c_{s+1} \Vdash (\exists x)\phi_e^k(G,x)$ or $c_{s+1} \Vdash (\forall x)\neg \phi_e^k(G,x)$. \\

Let $s \in \NN$, as $\F$ is $(n+1)$-generic, every property forced by $c_s$ will hold for $G_{\F}$ by \Cref{prop:core-forcing-imply-truth} and \Cref{prop:main-forcing-imply-truth}, hence $M_n'$ decides every $\Sigma_k^0$ property of $G_{\F}$ for $k < n+1$, thus $G_{\F}^{(n+1)} \leq M_n'$. Finally, by \Cref{prop:core-forcing-generic-set}, $G_{\F}$ is infinite, and by definition of a condition, $G_{\F} \subseteq A$.
\end{proof}

\Cref{prop:main-forcing-lown+1} enables to reprove the theorem from Monin and Patey~\cite{monin2021weakness} about $\Sub{\Delta^0_n}$.

\begin{theorem}[Monin and Patey~\cite{monin2021weakness}]
Let $B$ be a $\Delta_{n+1}^0$ set, there exists an infinite set $H$ of low${}_{n+1}$ degree such that $H \subseteq B$ or $H \subseteq \overline{B}$.    
\end{theorem}
\begin{proof}
By \Cref{lem:scott-tower}, there is a Scott tower $\M_0, \dots, \M_n$ of height~$n$ with Scott codes~$M_0, \dots, M_n$, such that for every~$i \leq n$, $M_i$ is of low degree over~$\emptyset^{(i)}$.
By \Cref{lem:largeness-tower}, it can be enriched with some sets~$C_0, \dots, C_{n-1}$ to form a largeness tower of height $n$. We can therefore add the assumption that $M_n' \leq \emptyset^{(n+1)}$ in \Cref{sect:main-forcing}. There are two cases:

Case $1$: $B \in \langle \U_{C_{n-1}}^{\M_{n-1}} \rangle$. In that case, since $B$ is $\Sigma_{n+1}^0$, \Cref{prop:main-forcing-lown+1} will hold for $B$ and there exists some infinite subset $H \subseteq B$ such that $H^{(n+1)} \leq_T M_n' \leq_T \emptyset^{(n+1)}$.

Case $2$: $B \not\in \langle \U_{C_{n-1}}^{\M_{n-1}} \rangle$. In that case, $\overline{B} \in \langle \U_{C_{n-1}}^{\M_{n-1}} \rangle$ by partition regularity of the class. Thus, since $\overline{B}$ is also $\Sigma_{n+1}^0$, \Cref{prop:main-forcing-lown+1} will also hold, this time for $\overline{B}$, yielding an infinite subset $H \subseteq \overline{B}$ such that $H^{(n+1)} \leq_T \emptyset^{(n+1)}$. 
\end{proof}

\subsection{Witness forcing}

Throughout this section, fix a $\Pi^0_{n+1}$ set~$A$ such that $\overline{A} \not\in \langle \U_{C_{n-1}}^{\M_{n-1}} \rangle$. The set~$A$ in this section must be thought of as the complement of the set~$A$ in \Cref{sect:main-forcing}. Contrary to the previous notions of forcing, witness forcing conditions need a second reservoir of higher complexity. 

\begin{definition}[Condition]
Let~$\WW^A_{n}$ be the notion of forcing whose conditions are tuples $(\sigma,X_{n-1}, X_{n})$ where:
\begin{itemize}
    \item $\sigma \subseteq A$ ;
    \item $X_{n-1} \supseteq X_{n}$ ;
    \item $\overline{X_{n}} \not\in \langle \U_{C_{n-1}}^{\M_{n-1}} \rangle$ ;
    \item $X_{n-1} \in \M_{n-1}$ and $X_{n} \in \M_{n}$.
\end{itemize}
\end{definition}

Given a $\WW^A_n$-condition $c = (\sigma, X_{n-1}, X_n)$, $X_n \in \langle \U_{C_{n-1}}^{\M_{n-1}} \rangle$ by partition regularity of the class, and 
$X_{n - 1} \in \langle \U_{C_{n-1}}^{\M_{n-1}} \rangle$ by its upward-closure. Therefore, $c \uh {\PP^A_{n-1}} = (\sigma, X_{n-1})$ is a valid $\PP^A_{n-1}$-condition.
One can think of a $\WW^A_{n}$-condition $(\sigma, X_{n-1}, X_n)$ either as a $\PP^A_{n-1}$-condition $(\sigma, X_{n-1})$, or as a Mathias condition $(\sigma, X_n)$. The notion of extension follows from both approaches:

\begin{definition}
A $\WW^A_{n}$-condition $(\tau, Y_{n-1}, Y_n)$ \emph{extends} a $\WW^A_{n}$-condition $(\sigma, X_{n-1}, \allowbreak X_n)$
if $Y_{n-1} \subseteq X_{n-1}$, $Y_n \subseteq X_n$ and $\sigma \preceq \tau \subseteq \sigma \cup X_n$. 
\end{definition}

The notion of cylinder is naturally defined as follows:

\begin{definition}
The \emph{cylinder} under a $\WW^A_n$-condition $(\sigma, X_{n-1}, X_n)$ is the class
$$
[\sigma, X_{n-1}, X_n] = \{ G : \sigma \subseteq G \subseteq \sigma \cup (X_n \cap A) \}.
$$
\end{definition}

The notion of index of a $\WW^A_{n}$-condition is defined accordingly:

\begin{definition}
An \emph{index} of a $\WW^A_{n}$-condition $(\sigma, X_{n-1}, X_n)$ is a tuple $\langle \sigma, a, b\rangle$ where $a$ is an $M_{n-1}$-index of~$X_{n-1}$ and $b$ is an $M_n$-index of~$X_n$.
\end{definition}

Despite having a second reservoir, $\WW^A_{n}$-forcing inherits abstractly many proper\-ties from $\PP^A_{n-1}$-forcing. The following commutation lemma shows that any density property over $\PP^A_{n-1}$-forcing yields a density property over $\WW^A_{n}$-forcing. It follows in particular that $\WW^A_{n}$-forcing inherits the forcing relation for $\Sigma_k^0$-formulas for $k \leq n$ and that \Cref{lem:core-question-find-extension} also holds with $\WW^A_{n}$-conditions.

\begin{lemma}\label[lemma]{lem:witness-compatible-core}
Let $c = (\sigma, X_{n-1}, X_n)$ be a $\WW^A_n$-condition and $(\tau, Y_{n-1}) \leq c \uh \PP^A_{n-1}$ be a $\PP^A_{n-1}$-extension. Then there is a $\WW^A_n$-extension $d \leq c$ such that $d \uh \PP^A_{n-1} = (\tau, Y_{n-1})$. Furthermore, an index for~$d$ can be found computably uniformly in an index for~$c$ and $(\tau, Y_{n-1})$.
\end{lemma}
\begin{proof}Since $c$ is a $\WW^A_n$-condition, $\overline{X_n} \notin \langle \U_{C_{n-1}}^{\M_{n-1}} \rangle$. Then, $(\tau, Y_{n-1})$ being a $\PP^A_{n-1}$-condition, we have $Y_{n-1} \in \M_{n-1} \cap \langle \U_{C_{n-1}}^{\M_{n-1}} \rangle$, thus $\langle \U_{C_{n-1}}^{\M_{n-1}} \rangle \subseteq \L_{Y_{n-1}}$ and $\overline{Y_{n-1}} \notin \langle \U_{C_{n-1}}^{\M_{n-1}} \rangle$.
Combining those two results yields that $\overline{X_n} \cup \overline{Y_{n-1}} \notin \langle \U_{C_{n-1}}^{\M_{n-1}} \rangle$ by partition regularity of $\langle \U_{C_{n-1}}^{\M_{n-1}} \rangle$.
Hence, $d = (\tau, Y_{n-1}, X_n \cap Y_{n-1})$ is a valid $\WW^A_n$-condition. It is clear that $d \uh \PP^A_{n-1} = (\tau, Y_{n-1})$ and that $d \leq c$. \\

By \Cref{rem:index-translation}, an $M_n$-index for the set $Y_{n-1}$ can be found uniformly computably using an $M_{n-1}$ index for $Y_{n-1}$. Thus, by our assumptions on the encodings of Scott ideals, an index for $X_n \cap Y_{n-1}$ can be found uniformly computably using an $M_{n}$-index for the set $X_n$ and an $M_{n-1}$-index for the set $Y_{n-1}$. Therefore, an index for~$d$ can be found uniformly computably  in an index for~$c$ and $(\tau, Y_{n-1})$.
\end{proof}

The following forcing relation for $\Sigma^0_{n+1}$ and $\Pi^0_{n+1}$ formulas is closer to the one from $\PP^A_n$-forcing than $\MM^A_n$-forcing, in that $\tau$ ranges over~$X_n$ instead of $X_n \cap A$. Contrary to $\MM^A_n$-forcing, the complexity of the forcing relation for $\Pi^0_{n+1}$-formulas does not play any role in the constructions.

\begin{definition}[Forcing relation]
Let $c = (\sigma,X_{n-1}, X_{n})$ be a $\WW^A_n$-condition, we define the forcing relation $\Vdash$ for $\Sigma_{n+1}^0$ and $\Pi_{n+1}^0$ formulas as follows: for $\phi(G,x)$ a $\Pi_{n}^0$ formula,
\begin{itemize}
    \item $c \Vdash (\exists x)\phi(G,x)$ if there exists some $a \in \NN$ such that $c \Vdash \phi(G,a)$.
    \item $c \Vdash (\forall x)\neg \phi(G,x)$ if $(\forall \tau \subseteq X_{n})(\forall x \in \NN)\sigma \cup \tau \qvdash \neg \phi(G,x)$ .
\end{itemize}
\end{definition}

As mentioned earlier in \Cref{sect:iterated-lowness}, if $\overline{A} \not \in \langle \U_{C_{n-1}}^{\M_{n-1}} \rangle$, then this is witnessed by a large $\Sigma^0_1(\M_{n-1})$ class $\U \supseteq \langle \U_{C_{n-1}}^{\M_{n-1}} \rangle$ such that $\overline{A} \not \in \U$. This witness plays an important role in the design of a forcing question with good definitional properties. We therefore define the notion formally and state the complexity of finding such a witness.

\begin{definition}
Let $c = (\sigma, X_{n-1}, X_{n})$ be a $\WW^A_n$-condition. A \emph{witness} for~$c$ is a $\Sigma^0_1(\M_{n-1})$ class~$\U \supseteq \langle \U_{C_{n-1}}^{\M_{n-1}} \rangle$  such that $\overline{X_{n}} \cup \overline{A} \not \in \U$.
\end{definition}


\begin{lemma}\label[lemma]{lem:witness-overapproximate}
    For every $\WW^A_n$-condition $(\sigma,X_{n-1}, X_{n})$, there exists some witness~$\U$.
    An index for such a class $\U$ can uniformly be found using $M_{n}'$.
\end{lemma}

\begin{proof}
By partition regularity of $\langle \U_{C_{n-1}}^{\M_{n-1}} \rangle$, $\overline{X_n} \cup \overline{A} \notin \langle \U_{C_{n-1}}^{\M_{n-1}} \rangle$, otherwise either $\overline{X_n}$ or $\overline{A}$ would be in $\langle \U_{C_{n-1}}^{\M_{n-1}} \rangle$ which is impossible by definition of a condition.

The class $\langle \U_{C_{n-1}}^{\M_{n-1}} \rangle$ being an intersection of $\Sigma_1^0(\M_{n-1})$ large classes, there exists one of those $\Sigma_1^0(\M_{n-1})$ large class $\U$ such that $\U \supseteq \langle \U_{C_{n-1}}^{\M_{n-1}} \rangle$ and such that $\overline{X_n} \cup \overline{A} \not \in \U$.

To find such a witness, one first $M_n'$-computes a set $D_{n-1} \supseteq C_{n-1}$ such that $\U_{D_{n-1}}^{\M_{n-1}} = \langle \U_{C_{n-1}}^{\M_{n-1}} \rangle$. Then, for every $(e, i) \in D_{n-1}$, ask whether for every $\rho \subseteq \overline{X_n} \cup \overline{A}$, $\rho \not \in W_e^{Z_i^{n-1}}$ (where $Z_i^{n-1}$ is the $i$-th element of $\M_{n-1}$). This can be done $M_n'$-computably as $\overline{A}$ is $\Sigma_{n+1}^0$ and $\M_n$ contains $\emptyset^{(n)}$ and $\overline{X_n}$. One must eventually find such a pair $(e, i)$ for which the answer is positive.    
\end{proof}

We now define forcing questions for $\Sigma^0_n$ and $\Sigma^0_{n+1}$-formulas in order to preserve hyperimmunities. Contrary to $\MM^A_n$-forcing, the set~$A$ is $\Pi^0_{n+1}$, hence one cannot ask for some~$\rho \subseteq X_{n-1} \cap A$ to satisfy some property, as the corresponding question would be too complex. We will therefore use an over-approximation by quantifying universally over all sets. Thankfully, given a $\WW^A_n$-condition $c$ and a witness~$\U$, one can restrict the over-approximation to all sets~$B$ such that $\overline{B} \not \in \U$, as this is the case for~$X_{n-1} \cap A$. This refined over-approximation has two benefits: (1) it is still compact, hence yields a forcing question with the appropriate definitional complexity, and (2) for every such set~$B$, since~$\overline{B}$ does not belong to a large class, then $B$ must contain many elements, hence one can always ask for a subset of~$B$.

\begin{definition}[Forcing question]
    Let $c = (\sigma,X_{n-1},X_n)$ be a $\WW^A_n$-condition and $\phi(G,x)$ a $\Pi_k^0$ formula for $k = n - 1$ or $k = n$. Let $\U$ be a $\Sigma_1^0(\M_{n-1})$-class and define $c \qvdash^\U (\exists x)\phi(G,x)$ to hold if for every $B \in 2^\NN$ such that $\overline{B} \not \in \U$, there is a finite $\tau \subseteq B \cap X_{k}$ and some~$x \in \NN$ such that $\phi(\sigma \cup \tau,x)$ holds if $k = 0$ or  $\sigma \cup \tau \nqvdash \neg \phi(G,x)$ holds if $k > 0$.
\end{definition}

The following lemma shows that the forcing question for $\Sigma^0_n$ and $\Sigma^0_{n+1}$ formulas has the appropriate definitional complexity.

\begin{lemma}\label[lemma]{lem:witness-forcing-question-complexity}
    Let $\phi(G,x)$ be a $\Pi_k^0$ formula for $k = n-1$ or $k = n$ and let $\U$ be a $\Sigma_1^0(\M_{n-1})$-class. For every $\WW^A_n$-condition $c$, the relation $c \qvdash^\U (\exists x) \phi(G,x)$ is $\Sigma_1^0(\M_k)$.
\end{lemma}

\begin{proof}
By compactness, the statement $c \qvdash^\U (\exists x) \phi(G,x)$ is equivalent to the following statement when $k > 0$ (the case $k = 0$ is similar), with $\overline{\beta}$ representing the bitwise complement of $\beta$ :
$$(\exists \ell)(\forall \beta \in 2^{\ell})([\overline{\beta}] \subseteq \U \vee (\exists \tau \finsub \beta \cap X_k)(\exists x)\sigma \cup \tau \nqvdash \neg \phi(G,x))$$

The statement $[\overline{\beta}] \subseteq \U$ is $\Sigma_1^0(\M_{n-1})$ as $\U$ is a $\Sigma_1^0(\M_{n-1})$-class. The statement $\sigma \cup \tau \nqvdash \neg \phi(G,x)$ is $\Sigma_1^0(\M_k)$ by \Cref{lem:core-complexity-question}.
\end{proof}

The following lemma states, as usual, that the forcing question meets its specifications.

\begin{lemma}\label[lemma]{lem:witness-question-find-extension}
    Let $c$ be a $\WW^A_n$-condition, $\U$ be a witness for~$c$ and $\phi(G,x)$ a $\Pi_k^0$ formula for $k = n - 1$ or $k = n$.
    \begin{itemize}
        \item If $c \qvdash^\U (\exists x)\phi(G,x)$, then there exists $d \leq c$ such that $d \Vdash (\exists x)\phi(G,x)$.
        
        \item If $c \nqvdash^\U (\exists x)\phi(G,x)$, then there exists $d \leq c$ such that $d \Vdash (\forall x)\neg \phi(G,x)$.
    \end{itemize}
    Deciding whether $c \qvdash^\U (\exists x)\phi(G,x)$ holds or not, and finding the appropriate extension can be done uniformly in an index of $\U$ and~$c$ using $M_n'$.
\end{lemma}

\begin{proof}
    Say $c = (\sigma,X_{n-1},X_{n})$ and let $\phi(G,x)$ be a $\Pi_k^0$ formula for $k = n -1$ or $k = n$. There are two cases:
    \begin{itemize}
        \item If $c \qvdash^\U (\exists x)\phi(G,x)$. By definition of a witness, $\overline{A \cap X_{n}} \notin \U$, hence there exists a finite $\tau \finsub A \cap X_{n}$ and some $a \in \NN$ such that $\sigma \cup \tau \nqvdash \neg \phi(G,a)$. By \Cref{lem:core-question-find-extension}, there exists an extension $(\rho,Y_{n-1}) \leq (\sigma \cup \tau, X_{n-1} \setminus \{0, \dots, |\tau|\})$ forcing $\phi(G,a)$. Therefore, $d = (\rho, Y_{n-1}, X_{n} \cap Y_{n-1})$ extends $c$ and forces $(\exists x)\phi(G,x)$.
        
        \item  If $c \nqvdash^\U (\exists x)\phi(G,x)$. The following class is $\Pi_1^0(\M_k)$ and non-empty:
        $$\C = \{B \in 2^{\NN} : \overline{B} \notin \U \wedge (\forall \tau \finsub B \cap X_k)(\forall x \in \NN) \sigma \cup \tau \qvdash \neg \phi(G,x) \}$$
        Therefore, as $\M_k$ is a Scott ideal, there exists some $B \in \C \cap \M_k$. Since $\overline{B} \notin \U$, $\overline{B} \notin \langle \U_{C_{n-1}}^{\M_{n-1}} \rangle$. As $\overline{B}$ and $\overline{X_k}$ are not in $\langle \U_{C_{n-1}}^{\M_{n-1}} \rangle$, $\overline{B \cap X_k} \notin \langle \U_{C_{n-1}}^{\M_{n-1}} \rangle$. Hence, for $k = n-1$, $d = (\sigma, B \cap X_{n-1}, B \cap X_{n})$ is a valid condition extending $c$ such that $d \Vdash (\forall x)\neg \phi(G,x)$ and for $k = n$, $d = (\sigma, X_{n-1}, B \cap X_{n})$ is a valid condition extending $c$ such that $d \Vdash (\forall x)\neg \phi(G,x)$.
    \end{itemize}
By \Cref{lem:witness-forcing-question-complexity}, the relation $c \qvdash^\U (\exists x)\phi(G,x)$ is $\Sigma^0_1(\M_k)$, hence can be decided using $M_n'$. In the first case,  $\tau$, $a$ and the extension $(\rho, Y_{n-1})$ can be found $(A \oplus M_n)$-computably, hence $M_n'$-computably since $A$ is $\Pi^0_{n+1}$ and $\emptyset^{(n)} \in \M_n$.
In the second case, an $M_k$-index of a tree~$T \subseteq \bstr$ such that $[T] = \C$ can be found computably, and and since $M_k$ is a Scott code, an $M_k$-index of $B$ can be found computably in the $M_k$-index of~$T$. Thus an index of the extension witnessing the second case is found computably in an index of~$c$.
\end{proof}

Given a $\WW^A_n$-filter $\F$, we write $[\F]$ for $\bigcap_{(\sigma,X_{n-1},X_{n}) \in \F} [\sigma,X_{n-1},X_{n}]$
and $\F \uh \PP^A_{n-1} = \{ c \uh \PP^A_{n-1} : c \in \F \}$. Note that $\F \uh \PP^A_{n-1}$ is a $\PP^A_{n-1}$-filter and that $[\F] \subseteq [\F \uh \PP^A_{n-1}]$.
The following proposition is the counterpart of \Cref{prop:core-forcing-generic-set} for the witness forcing. 

\begin{proposition}\label[proposition]{prop:witness-forcing-generic-set}
Let $\F$ be a $1$-generic $\WW^A_n$-filter. There exists a unique set $G_{\F}$ belonging to $[\F]$. Furthermore, $G_{\F}$ is infinite.
\end{proposition}
\begin{proof}
For every condition $(\sigma,X_{n-1}, X_n)$ in $\F$, the cylinder $[\sigma, X_{n-1}, X_n]$ is a non-empty closed subset of $2^{\NN}$. Assume by contradiction that $[\F] = \emptyset$, then as $2^{\NN}$ is compact, there exists a finite subfamily $E$ of $\F$ such that 
$$\bigcap_{(\sigma, X_{n-1}, X_n) \in E} [\sigma, X_{n-1}, X_n] = \emptyset$$ 
By compatibility of a filter, there is a condition $c$ extending every member of this family, which yields that the cylinder under $c$ is empty, contradiction. Therefore, $[\F]$ is non-empty. 

For every condition $(\sigma,X_{n-1}, X_n)$ in $\F$, $[\sigma, X_{n-1}, X_n] \subseteq [\sigma, X_{n-1}]$,
hence by \Cref{prop:core-forcing-generic-set}, $[\F]$ is a singleton~$G_\F$, which is infinite as a set.
\end{proof}

\begin{proposition}\label[proposition]{prop:witness-forcing-imply-truth}
Let $\F$ be a $n$-generic $\WW^A_n$-filter. If there exists some $c \in \F$ such that $c \Vdash (\forall x) \phi(G,x)$ for some $\Sigma_n^0$ formula $\phi(G,x)$, then for every $a \in \NN$, there exists some $d \in \F$ such that $d \Vdash \phi(G,a)$. Furthermore, $(\forall x)\phi(G_{\F},x)$ will hold.   
\end{proposition}

\begin{proof}
A similar proof as the one of \Cref{prop:core-forcing-imply-truth} can be used to prove this result, by discarding the reservoir $X_{n-1}$ of a condition $(\sigma, X_{n-1}, X_n)$.   
\end{proof}

We are now ready to prove our second abstract construction. When considering a set~$M_n$ such that $M_n' \leq_T \emptyset^{(n+1)}$, it states the existence of an infinite subset of~$A$ of low${}_{n+1}$ degree. Note that \Cref{prop:witness-forcing-lown+1} is the same statement as \Cref{prop:main-forcing-lown+1}.

\begin{proposition}\label[proposition]{prop:witness-forcing-lown+1}
There exists an infinite subset $H \subseteq A$ such that $H^{(n+1)} \leq_T M_n' $.
\end{proposition}

\begin{proof}
The proof is similar to that of \Cref{prop:main-forcing-lown+1}:

An $(n+1)$-generic $\WW^A_n$-filter $\F$ can be build effectively in $M_n'$ using \Cref{lem:witness-overapproximate} to find suitable over-approximations $\U$ of the class $\langle \U_{C_{n-1}}^{\M_{n-1}} \rangle$, and using \Cref{lem:witness-question-find-extension} and \Cref{lem:core-question-find-extension} to build extensions deciding every $\Sigma_k^0$ formulas for $k \leq n+1$.

By \Cref{prop:witness-forcing-generic-set}, there exists an infinite $M_n'$-computable set $G_{\F} \subseteq A$. Furthermore, by \Cref{prop:witness-forcing-imply-truth} and \Cref{prop:core-forcing-imply-truth}, every property forced by $\F$ will hold for $G_{\F}$. Hence, $G_{\F}^{(n+1)} \leq M_{n}'$.
\end{proof}


\subsection{Applications}

First-jump control is often considered as simpler than second-jump control. This is why the former technique preferable, when available. In their seminar paper, Cholak, Jockusch and Slaman~\cite{cholak2001strength} constructed solutions to $\Sub{\Delta^0_2}$ using both techniques, to show that it admits weakly low and a low${}_2$ basis.
Very recently, Benham et al.~\cite{benham2024ginsburg} showed that the second-jump control of Cholak, Jockusch and Slaman~\cite{cholak2001strength} could be extended to $\Sub{\Sigma^0_2}$ to produce low${}_2$ solutions and asked whether it was also the case for first-jump control.

We now prove that $\Sub{\Sigma^0_n}$ admits solutions both with $(n+1)$th jump control and $n$th jump control,
to prove a low${}_{n+1}$ basis and a weakly low${}_n$ basis, respectively. We first start with a direct $(n+1)$th jump control construction.

\begin{theorem}
Let $n \geq 1$ and let~$B$ be a $\Sigma_{n+1}^0$ set, then there exists an infinite set $H$ of low${}_{n+1}$ degree such that $H \subseteq B$ or $H \subseteq \overline{B}$.   
\end{theorem}

\begin{proof}
By \Cref{lem:scott-tower}, there is a Scott tower $\M_0, \dots, \M_n$ of height~$n$ with Scott codes~$M_0, \dots, M_n$, such that for every~$i \leq n$, $M_i$ is of low degree over~$\emptyset^{(i)}$.
By \Cref{lem:largeness-tower}, it can be enriched with some sets~$C_0, \dots, C_{n-1}$ to form a largeness tower of height $n$. There are two cases:

Case 1: $B \in \langle \U_{C_{n-1}}^{\M_{n-1}} \rangle$. Then \Cref{prop:main-forcing-lown+1} yields a set $H$ such that $H \subseteq B$ and $H^{(n+1)} \leq_T M_n'$.

Case 2: $B \notin \langle \U_{C_{n+1}}^{\M_{n+1}} \rangle$. Then \Cref{prop:witness-forcing-lown+1} applied on the set $\overline{B}$ yields a set $H \subseteq \overline{B}$ such that $H^{(n+1)} \leq_T M_n'$.

Since $M_n$ is low over $\emptyset^{(n)}$, $M_n' \leq_T \emptyset^{(n+1)}$ and $H$ is low$_{n+1}$.
\end{proof}

We now turn to the $n$th jump control construction to prove that $\Sub{\Sigma^0_n}$ admits a weakly low${}_n$ basis. Due to the asymmetry of $\Sigma^0_{n+1}$ sets, contrary to $\Delta^0_{n+1}$ sets, the construction is more complicated, involving some new combinatorics. Note that the proof of \Cref{thm:main-weakly-low-basis} involves only the core forcing on both sides.

\begin{repmaintheorem}{thm:main-weakly-low-basis}
Fix~$n \geq 1$. For every $\Sigma^0_{n+1}$ set~$B$ and every set~$P$ of PA degree over~$\emptyset^{(n)}$, there is an infinite set~$H \subseteq B$ or $H \subseteq \overline{B}$ such that $H^{(n)} \leq_T P$.
\end{repmaintheorem}
\begin{proof}
By \Cref{lem:scott-tower}, there is a Scott tower $\M_0, \dots, \M_{n-1}$ of height~$n-1$ with Scott codes~$M_0, \dots, M_{n-1}$, such that for every~$i < n$, $M_i$ is of low degree over~$\emptyset^{(i)}$.
Moreover, by \Cref{prop:pa-to-scott}, $P$ computes a Scott code~$M_n$ of a Scott ideal~$\M_n$ containing~$\emptyset^{(n)}$. Thus, $\M_0, \dots, \M_{n-1}, \M_n$ forms a Scott tower of height~$n$.
By \Cref{lem:largeness-tower}, it can be enriched with some sets~$C_0, \dots, C_{n-1}$ to form a largeness tower of height $n$. There are two cases: \\

\textbf{Case 1}: $B \in \langle \U_{C_{n-1}}^{\M_{n-1}} \rangle$. 

\begin{lemma}\label[lemma]{lem:main-weak-decide}
    Let $c$ be a $\PP^B_{n-1}$-condition and $\phi(G,x)$ a $\Pi_{k}^0$ formula for some $k < n$. An index of an extension deciding $(\exists x)\phi(G,x)$ can be found $P$-computably uniformly in an index of~$c$ and $\phi$.
\end{lemma}

\begin{proof}
Say $c = (\sigma, X_{n-1})$.
Let~$\V = \{Z : (\exists \rho \subseteq Z)(\exists x \in \NN) \sigma \cup \rho \nqvdash \neg \phi(G,x)\}$ if $k > 0$ and otherwise, if $k = 0$, let~$\V = \{Z : (\exists \rho \subseteq Z)(\exists x \in \NN) \phi(\sigma \cup \rho,x)\}$.
We claim that one of the following two statements is true:
\begin{enumerate}
    \item[(1.a)] there is some~$\rho \subseteq X_{n-1} \cap B$ and some~$x \in \NN$ such that $\sigma \cup \rho \nqvdash \neg \phi(G,x)$ holds if $k > 0$ or $\phi(\sigma \cup \rho,x)$ if $k = 0$;
    \item[(1.b)] the class $\U_{C_{n-1}}^{\M_{n-1}} \cap \V$ is not large.
\end{enumerate}
Indeed, if (1.b) fails, then by $\M_{n-1}$-cohesiveness of $\U_{C_{n-1}}^{\M_{n-1}}$, $\langle \U_{C_{n-1}}^{\M_{n-1}} \rangle \subseteq \V$. Since $X_{n-1} \cap B \in \langle \U_{C_{n-1}}^{\M_{n-1}} \rangle$, there exists some $\rho \subseteq X_{n-1} \cap B$ and some $a \in \NN$ such that either $\sigma \cup \rho \nqvdash \neg \phi(G,a)$ holds if $k > 0$ or $\phi(\sigma \cup \rho,a)$ holds if $k = 0$, so we are in case (1.a).

Note that the first statement is $\Sigma^0_{n+1}$ as $\sigma \cup \rho \nqvdash \neg \phi(G,x)$ is $\Sigma^0_1(\M_k)$ by \Cref{lem:core-complexity-question}, and the second statement is $\Sigma^0_1(\M_n)$ by \Cref{lem:complexity-largeness}, so since~$P \geq_T M_n$, both events are $P$-c.e. One can then wait $P$-computably for one of the two facts to be witnessed.

In Case 1.a, by \Cref{lem:core-question-find-extension}, there exists an extension $d \leq (\sigma \cup \rho, X_{n-1} \setminus \{0,\dots,|\rho|\}) \leq c$ forcing $(\exists x)\phi(G,x)$.

In Case 1.b, by \Cref{lem:decreasing-sequence-large}, there is a finite set~$F \finsub C_{n-1}$, such that $\U_{F}^{\M_{n-1}} \cap \V$ is not large. Given~$k \in \NN$, let $\P_k$ be the $\Pi^0_1(\M_{n-1})$ class of all $k$-partitions $Y_0 \sqcup \dots \sqcup Y_{k-1} = \NN$ such that for every~$i < k$, either $Y_i \not \in \U_F^{\M_{n-1}}$ or $Y_i \not \in \V$. By assumption, there is some~$k$ such that $\P_k \neq \emptyset$. Since $\M_{n-1}$ is a Scott ideal, there is a $k$-partition $Y_0 \sqcup \dots \sqcup Y_{k-1} = \NN$ in $\P_k \cap \M_{n-1}$. By construction and  \Cref{lem:complexity-finding-among-partition}, since the $k$-partition belongs to~$\M_{n-1}$, $C_{n-1}\oplus M_{n-1}'$ can computably find some~$i < k$ such that $Y_i \in \U^{\M_{n-1}}_{C_{n-1}}$, uniformly in an $M_{n-1}$-index of the $k$-partition. By $\M_{n-1}$-cohesiveness of~$\U^{\M_{n-1}}_{C_{n-1}}$, since $X_{n-1},Y_i \in \M_{n-1}$ and $X_{n-1}, Y_i \in \U^{\M_{n-1}}_{C_{n-1}}$, then $\U^{\M_{n-1}}_{C_{n-1}} \subseteq \L_{X_{n-1}} \cap \L_{Y_i}$, so $X_{n-1} \cap Y_i \in \langle \U^{\M_{n-1}}_{C_{n-1}} \rangle$. In particular, $Y_i \in \U^{\M_{n-1}}_F$ so $Y_i \not \in \V$. It follows that $(\sigma, X_{n-1} \cap Y_i)$ forces $(\forall x) \neg\phi(G,x)$.
\end{proof}

Using \Cref{lem:main-weak-decide}, we can construct, computably in $P$, a sequence $c_0 \geq c_1 \geq \dots$ of $\PP^B_{n-1}$ conditions such that the set $\F = \{c \in \PP^B_{n-1} : (\exists s) c \geq c_s \}$ is a $n$-generic filter. By \Cref{prop:core-forcing-generic-set}, the corresponding set $G_{\F}$ is infinite, $P$ computable and satisfies $G_{\F} \subseteq B$. Furthermore, by \Cref{prop:core-forcing-imply-truth}, every property forced will hold for $G_{\F}$, hence $G_{\F}^{(n)} \leq P$.  \\

\textbf{Case 2}: $B \notin \langle \U_{C_{n-1}}^{\M_{n-1}} \rangle$. Then, there exists a $\Sigma_1^0(\M_{n-1})$ class $\U$ such that $\langle \U_{C_{n-1}}^{\M_{n-1}} \rangle \subseteq \U$ and $B \notin \U$. 

Let $W_\U$ be the $\Sigma_1^0(\M_{n-1})$ set of chains associated with $\U$. Since $B \notin \U$, for every $\rho \in W_\U$, $\rho \cap \overline{B} \neq \emptyset$.

\begin{lemma}\label[lemma]{lem:witness-weak-decide}
    Let $c$ be a $\PP^{\overline{B}}_{n-1}$-condition and $\phi(G,x)$ a $\Pi_k^0$ formula for some $k < n$. An index of an extension deciding $(\exists x)\phi(G,x)$ can be found $P$-computably uniformly in an index of $c$ and $\phi$. 
\end{lemma} 

\begin{proof}
Say $c =(\sigma, X_{n-1})$.
Consider $\V$ the same class as in \Cref{lem:main-weak-decide} and let $W_\V$ be the $\Sigma_1^0(\M_{n-1})$ set of chains associated with $\V$.



We claim that one of the following two cases holds:

\begin{enumerate}
    \item[(2.a)] There is some $\ell \in \NN$ and some $\rho_0, \dots, \rho_{\ell-1} \in W_\U$ such that $\rho_i \subseteq X_{n-1}$ for all $i < \ell$ and for every $\mu \in \rho_0 \times \dots \times \rho_{\ell-1}$ (seen as finite sets of integers), there exists some $\rho \subseteq \mu$ and some $x \in \NN$ such that $\sigma \cup \rho \nqvdash \neg \phi(G,x)$ holds if $k > 0$ or $\phi(\sigma \cup \rho,x)$ holds if $k = 0$ ;
    \item[(2.b)] $\U_{C_{n-1}}^{\M_{n-1}} \cap \V$ is not large.
\end{enumerate}

Indeed, if $(2.b)$ fails, then by \Cref{lem:cohesive-compatibility}, $\U_{C_{n-1}}^{\M_{n-1}} \cap \V \cap \U$ is large. It follows that $\langle \U_{C_{n-1}}^{\M_{n-1}} \rangle \subseteq \V \cap \U$. In particular, as $X_{n-1} \in \langle \U_{C_{n-1}}^{\M_{n-1}} \rangle$ and $\langle \U_{C_{n-1}}^{\M_{n-1}} \rangle$ is partition regular, there exists a depth $p$ such that for every partition $Y_0 \cup Y_1 = \{0,\dots,p-1\} \cap X_{n-1}$, there exists some side $i < 2$ and some $\rho \in W_\U$ and $\rho' \in W_\V$ such that $\rho, \rho' \subseteq Y_i$. Fix $\rho_0, \dots, \rho_{\ell - 1}$ be all such $\rho \in W_\U$. We claim that this family satisfies $(2.a)$. Indeed, let $\mu \in \rho_0 \times \dots \times \rho_{\ell-1}$, and consider the partition $\mu \cup (\overline{\mu} \cap X_{n-1} \cap \{0, \dots, d-1\}) =  X_{n-1} \cap \{0, \dots, d-1\}$, by definition of the $(\rho_k)_{k < \ell}$, there exists some $k < \ell$ and some $\rho' \in W_\V$ such that $\rho_k, \rho' \subseteq \mu$ or $\rho_k, \rho' \subseteq \overline{\mu}$. As $\mu \in \rho_0 \times \dots \times \rho_{\ell-1}$, it is not possible for $\rho_k$ to be a subset of $\overline{\mu}$ for any $k < \ell$. Hence, we are in the first case and, as $\rho' \in W_\V$, we have found some $\rho' \subseteq \mu$ and some $x \in \NN$ such that $\sigma \cup \rho' \nqvdash \neg \phi(G,x)$ holds if $k > 0$ or $\phi(\sigma \cup \rho', x)$ holds if $k = 0$, so we are in case $(2.a)$.

Note that the first case is $\Sigma^0_1(\M_{n-1})$ and the second is $\Sigma^0_1(\M_n)$, so both events are $P$-c.e. since~$P \geq M_n$. One can then $P$-computably wait for one of the two facts to be witnessed.

In Case 2.a, since $\rho \cap \overline{B} \neq \emptyset$ for every $\rho \in W_\U$, there exists some $\mu \in \rho_0 \times \dots \times \rho_{\ell-1}$ such that $\mu \subseteq \overline{B} \cap X_{n-1}$. Hence, there exists some $\rho \subseteq \mu$ such that $\sigma \cup \rho \qvdash (\exists x)\phi(G,x)$, and by \Cref{lem:core-question-find-extension}, there exists an extension $d \leq (\sigma \cup \rho, X_{n-1} \setminus \{0,\dots, |\rho|) \leq c$ forcing $(\exists x)\phi(G,x)$. Such a $\mu$ can be found uniformly in $P$: there are only finitely many elements in $\rho_0 \times \dots \times \rho_{\ell-1}$, hence a PA over $\emptyset'$ can find one satisfying the $\Pi_{n+1}^0$ property $\mu \subseteq \overline{B} \cap X_{n-1}$.

In Case 2.b, as in Case 1.b, there is a set $Y_i \in \M_{n-1}$ such that the extension $(\sigma, X_{n-1} \cap Y_i)$ forces $(\forall x)\neg \phi(G,x)$ and such an extension can be found computably in~$P$.

\end{proof}

Similarly, we can $P$ computably construct an infinite set $G \subseteq \overline{B}$ such that $G^{(n)} \leq P$ using \Cref{lem:witness-weak-decide}. Thus, the theorem is proved.

\end{proof}

\begin{corollary}
Let $n \geq 1$ and let~$A$ be a $\Sigma_{n+1}^0$ set, then there exists an infinite set $H$ of low${}_{n+1}$ degree such that $H \subseteq A$ or $H \subseteq \overline{A}$.
\end{corollary}

Recall that the Ginsburg-Sands theorem~\cite{ginsburg1979minimal} restricted to $T_1$-spaces states that every infinite topological space has an infinite subspace homeomorphic to either the discrete or the cofinite topology on~$\NN$. This theorem was studied by Benham et al.~\cite{benham2024ginsburg}, who proved that it is equivalent over~$\RCA_0$ to $\COH + \Sub{\Sigma^0_2}$. Let $\GST$ be the Ginsburg-Sands theorem for $T_1$-spaces. Benham et al.~\cite{benham2024ginsburg} asked whether $\GST$ admits a weakly low basis. We answer positively through the following theorems:

\begin{proposition}\label{prop:weakly-low-jump-model}
Fix~$n \geq 1$. Let $\Psf_0, \dots, \Psf_{k-1}$ be $\Pi^1_2$-problems admitting a weakly low${}_n$ basis.
Then for every set~$Q$ of PA degree over~$\emptyset^{(n)}$, there is an $\omega$-model $\M$ of~$\RCA_0 + \Psf_0 + \dots + \Psf_{k-1}$
such that for every set~$X \in \M$, $X^{(n)} \leq_T Q$.
\end{proposition}
\begin{proof}
Fix~$Q$ of PA degree over~$\emptyset^{(n)}$. By Scott~\cite{scott1962algebra}, $Q$ computes a Scott code of a Scott ideal~$\N$ containing~$\emptyset^{(n)}$. Since~$\Psf_0, \dots, \Psf_{k-1}$ admit a weakly low${}_n$ basis, construct a chain~$\M_0 \subseteq \M_1 \subseteq \dots$ of countable $\omega$-models of~$\RCA_0$ such that
\begin{itemize}
    \item $\M_0$ is the $\omega$-model whose second-order part are the computable sets;
    \item $\M_i$ is topped by a set~$D_i$ such that $D_i^{(n)} \in \N$;
    \item For every $s < k$ and every instance $X$ of~$\Psf_s$ in~$\M_i$, there is some~$j \in \NN$ such that $\M_j$
    contains a $\Psf_s$-solution to~$X$.
\end{itemize}
Then $\M = \bigcup_{i \in \NN} \M_i$ is an $\omega$-model of~$\RCA_0 + \Psf_0 + \dots + \Psf_{k-1}$
such that for every set~$X \in \M$, $X^{(n)} \in \N$, hence $X^{(n)} \leq_T Q$.
\end{proof}

\begin{corollary}\label{cor:weakly-low-provability}
Fix~$n \geq 1$. Let $\Psf_0, \dots, \Psf_{k-1}$ and $\Qsf$ be $\Pi^1_2$ problems such that
each $\Psf_i$ admits a weakly low${}_n$ basis for~$i < k$ and $\RCA_0 + \Psf_0 + \dots + \Psf_{k-1} \vdash \Qsf$.
Then $\Qsf$ admits a weakly low${}_n$ basis.
\end{corollary}
\begin{proof}
Let~$X$ be a computable instance of~$\Qsf$ and let~$R$ be of PA degree over~$\emptyset^{(n)}$.
By \Cref{prop:weakly-low-jump-model}, there is an $\omega$-model~$\M$ of $\RCA_0 + \Psf_0 + \dots + \Psf_{k-1}$
such that for every set~$Y \in \M$, $Y^{(n)} \leq_T R$. In particular, $\M \models \Qsf$, so there is a $\Qsf$-solution~$Y$ to~$X$ in~$\M$. In particular, $Y^{(n)} \leq_T R$.
\end{proof}

Note that \Cref{cor:weakly-low-provability} also holds when the implication $\RCA_0 + \Psf_0 + \dots + \Psf_{k-1} \vdash \Qsf$ is restricted to $\omega$-models, also known as \emph{computable entailment}.
The following corollary states that $\GST$ admits a weakly low basis, answering the question of Benham et al.~\cite{benham2024ginsburg}. We refer to Dorais~\cite{dorais2011reverse} for the formalization of countable second-countable spaces in second-order arithmetic.

\begin{corollary}
For every set~$Q$ of PA degree over~$\emptyset'$ and every infinite computable $T_1$ CSC space $\langle X, \U, k\rangle$, $X$ has an infinite subspace~$Y$ which is either discrete or has the cofinite topology, and such that $Y' \leq_T Q$.
\end{corollary}
\begin{proof}
By Cholak, Jockusch and Slaman~\cite{cholak2001strength}, $\COH$ admits a weakly low basis.
By the \Cref{thm:main-weakly-low-basis}, so does $\Sub{\Sigma^0_2}$. Moreover, by Benham et al.~\cite{benham2024ginsburg}, 
$\RCA_0 \vdash \GST \leftrightarrow (\COH + \Sub{\Sigma^0_2})$. Thus, by \Cref{cor:weakly-low-provability}, $\GST$ admits a weakly low basis.
\end{proof}

\section{Preservation of hyperimmunities}\label{sect:preservation-hyps}

The goal of this section is to prove the following main theorem:

\begin{repmaintheorem}{thm:main-preservation-hyps}
Fix~$n \geq 1$. For every $\Delta^0_{n+1}$ set~$A$, every $\emptyset^{(n-1)}$-hyperimmune function $f : \NN \to \NN$ and every $\emptyset^{(n)}$-hyperimmune function $g : \NN \to \NN$, there is an infinite set~$H \subseteq A$ or $H \subseteq \overline{A}$ such that $f$ is $H^{(n-1)}$-hyperimmune and $g$ is $H^{(n)}$-hyperimmune.
\end{repmaintheorem}

The preservation of multiple hyperimmunities is motivated by the separation of $\Sub{\Delta^0_{n+1}}$ from $\Sub{\Sigma^0_{n+1}}$ over $\omega$-models.
Before proving \Cref{thm:main-preservation-hyps}, we prove this separation assuming the theorem holds. 

\begin{proposition}[Benham et al.~\cite{benham2024ginsburg}]\label{prop:sigman-hyp-deltan-hyp}
Let $n \geq 1$. There exists a $\Sigma^0_{n+1}$ set~$B$ which is $\emptyset^{(n-1)}$-hyperimmune and whose complement is $\emptyset^{(n)}$-hyperimmune.
\end{proposition}

\begin{proof}[Proof sketch]
By relativizing to~$\emptyset^{(n)}$ the finite injury priority construction of a c.e. set whose complement is hyperimmune (see Post~\cite[Section 5]{post1944recursively}), and interleaving steps to make~$B$ $\emptyset^{(n-1)}$-hyperimmune.
\end{proof}

The case $n = 1$ of the following corollary was proven by Benham et al.~\cite{benham2024ginsburg}.

\begin{corollary}
Let $n \geq 1$. There exists an $\omega$-model of~$\RCA_0 + \Sub{\Delta^0_{n+1}}$ which is not a model of~$\Sub{\Sigma^0_{n+1}}$.
\end{corollary}
\begin{proof}
Let~$B$ be a $\Sigma^0_{n+1}$ set which is $\emptyset^{(n-1)}$-hyperimmune and whose complement is $\emptyset^{(n)}$-hyperimmune. Such as set exists by \Cref{prop:sigman-hyp-deltan-hyp}. Recall that the principal function of an infinite set $X = \{ x_0 < x_1 < \dots \}$ is the function $k \mapsto x_k$. Let~$f_{n-1}$ and $f_n$ be the principal functions of~$B$ and $\overline{B}$, respectively. 

Using \Cref{thm:main-preservation-hyps}, construct a chain $\N_0 \subseteq \N_1 \subseteq \dots$ of countable $\omega$-models of $\RCA_0$ such that 
\begin{enumerate}
    \item $\N_0$ is the $\omega$-model whose second order part are the computable sets ;
    \item $\N_i$ is topped by a set $D_i$ such that $f_{n-1}$ is $D_i^{(n-1)}$-hyperimmune and $f_n$ is $D_i^{(n)}$-hyperimmune ;
    \item For every $\Delta_{n+1}^0(\N_i)$ set $X$, there exists some $j \in \NN$ such that $\N_j$ contains an infinite subset of $X$ or $\overline{X}$.
\end{enumerate}

Then $\N = \bigcup_{i \in \NN} \N_i$ is an $\omega$-model of $\RCA_0 + \Sub{\Delta^0_{n+1}}$ such that for every set $X$ of $\N$, $f_{n-1}$ is $X^{(n-1)}$-hyperimmune and $f_n$ is $X^{(n)}$-hyperimmune. Hence, $\N$ contains no subset of $B$ or $\overline{B}$ and therefore is not a model of $\Sub{\Sigma^0_{n+1}}$.
\end{proof}

The remainder of this section is devoted to the proof of \Cref{thm:main-preservation-hyps}. Fix some~$n \geq 1$. Let $\M_0, \dots, \M_n$ be Scott ideals with Scott codes $M_0, \dots, M_n$, and let $C_0, \dots, C_{n-1}$ be sets forming a largeness tower, that is,  for every $i < n$:
\begin{itemize}
    \item $\U_{C_i}^{\M_i}$ is an $\M_i$-cohesive large class containing only infinite sets ;
    \item $C_i,M_i' \in \M_{i+1}$ ;
    \item $\U_{C_{i+1}}^{\M_{i+1}} \subseteq \langle \U_{C_i}^{\M_i} \rangle$ if $i < n-1$.
\end{itemize}
Let $A$ be a $\Delta^0_{n+1}$ set. In particular, $A$ and $\overline{A}$ are both $\Sigma^0_{n+1}$, and by partition regularity of $\langle \U_{C_{n-1}}^{\M_{n-1}} \rangle$, either $\MM^A_n$ or $\MM^{\overline{A}}_n$ is a valid notion of forcing. Say $\MM^A_n$ is valid. To preserve hyperimmunities at levels $\{n-1,n\}$, one needs to have a $\Sigma^0_{k+1}$-preserving $\Sigma^0_{k+1}$-compact forcing question for $\Sigma^0_{k+1}$-formulas, for~$k \in \{n-1, n\}$. This is the case for $\MM^A_n$-forcing for $\Sigma^0_{n+1}$-formulas, but not for $\Sigma^0_n$-formulas. Indeed, the $\MM^A_n$-forcing question for $\Sigma^0_n$-formulas is neither $\Sigma^0_n$-preserving, nor $\Sigma^0_n$-compact.

We will therefore design a new forcing question for $\Sigma^0_n$-formulas, not formulated in terms of largeness, but by over-approximation of the set~$A$. This over-approximation comes at one cost: this yields a disjunctive forcing question. We must therefore work with a new notion of forcing, whose conditions are of the form $(\sigma_0, \sigma_1, X_{n-1})$ where $(\sigma_0, X_{n-1})$ is an $\MM^A_n$-condition and $(\sigma_1, X_{n-1})$ is an $\MM^{\overline{A}}_n$-condition. This notion of forcing can be defined only in the case $A$ and $\overline{A}$ both belong to $\langle \U_{C_{n-1}}^{\M_{n-1}} \rangle$.
If this is not the case, then we shall use witness forcing $\WW^A_n$ or $\WW^{\overline{A}}_n$, depending on the case.

In what remains, we first introduce the disjunctive notion of forcing in \Cref{sect:disjunctive-forcing}, then we study in \Cref{sect:compact-forcing-questions} which forcing questions are compact among the various notions of forcing. Last, we combine the various frameworks to prove \Cref{thm:main-preservation-hyps} in \Cref{sect:preservation-apps}.


\subsection{Disjunctive forcing}\label{sect:disjunctive-forcing}

The disjunctive notion of forcing is a variant of Mathias forcing commonly used in the reverse mathematics of Ramsey-type theorems, due to the intrinsic disjunctive nature of the statements. Fix a $\Delta^0_{n+1}$ set $A$ such that $A$ and $\overline{A}$ both belong to $\langle \U_{C_{n-1}}^{\M_{n-1}} \rangle$.

We shall see thanks to \Cref{lem:disjunctive-compatible-core} that a disjunctive condition is nothing but two main forcing conditions sharing a reservoir. Thus, disjunctive forcing inherits many properties of the main forcing, and in particular, each side admits a non-disjunctive forcing question for $\Sigma^0_{n+1}$ with the good definitional properties. In particular, if $g : \NN \to \NN$ is $M_n$-hyperimmune, then every sufficiently generic filter~$\F$ will yield two infinite sets~$G_0, G_1$ such that $g$ will be both $G_0^{(n)}$-hyperimmune and $G_1^{(n)}$-hyperimmune. The only purpose of this disjunctive notion of forcing is then to design a disjunctive forcing question for $\Sigma^0_n$-formulas with the good definitional properties. Then, if $f : \NN \to \NN$ is $M_{n-1}$-hyperimmune, it will be either $G_0^{(n-1)}$-hyperimmune, or $G_1^{(n-1)}$-hyperimmune.

\begin{definition}[Condition]
Let $\DD_{n}^A$ be the notion of forcing whose conditions are tuples $(\sigma_0, \sigma_1, X_{n-1})$ such that $(\sigma_0, X_{n-1}) \in \MM_n^A$ and $(\sigma_1, X_{n-1}) \in \MM_n^{\overline{A}}$.
\end{definition}

For $c = (\sigma_0, \sigma_1, X_{n-1}) \in \DD_{n}^A$ and $i < 2$, we write $c^{[i]}$ for the $\MM_n^{A^i}$-condition $(\sigma_i, X_{n-1})$, where $A^0 = A$ and $A^1 = \overline{A}$. 
The notion of $\DD_{n}^A$-extension is naturally induced by the notion of $\MM^A_n$-extension on each side:

\begin{definition}
A $\DD_{n}^A$-condition $d = (\tau_0, \tau_1, Y_{n-1})$ \emph{extends} a $\DD_{n}^A$-condition $c = (\sigma_0, \sigma_1, X_{n-1})$ (and we write $d \leq c$) if $d^{[i]} \leq c^{[i]}$ for both~$i < 2$, that is, $Y_{n-1} \subseteq X_{n-1}$ and $\sigma_i \preceq \tau_i \subseteq \sigma_i \cup X_{n-1}$.
\end{definition}

Since a reservoir forces only negative information, having two Mathias condi\-tions share a common reservoir does not impact their ability to force properties. This is made clear in the following lemma, which plays the same role as \Cref{lem:witness-compatible-core} for witness forcing.

\begin{lemma}\label[lemma]{lem:disjunctive-compatible-core}
Let $c = (\sigma_0, \sigma_1, X_{n-1})$ be a $\DD^A_n$-condition and $(\tau, Y_{n-1}) \in \PP^{A^i}_{n-1}$ such that $(\tau, Y_{n-1}) \leq c^{[i]}$ for some $i < 2$. Then there is a $\DD^A_n$-extension $d \leq c$ such that $d^{[i]} = (\tau, Y_{n-1})$. Furthermore, an index for~$d$ can be found computably uniformly in an index for~$c$ and $(\tau, Y_{n-1})$.
\end{lemma}

\begin{proof}
Simply take $d = (\tau,\sigma_1,Y_{n-1})$ if $i = 0$ and $d = (\sigma_0,\tau,Y_{n-1})$ if $i = 1$. It is clear that $d \leq c$ and that $d^{[i]} = (\tau, Y_{n-1})$.
\end{proof}

We now define the disjunctive forcing question to decide $\Sigma^0_n$-formulas with better definitional properties.
Since the set $A$ is too complex with respect to the forcing question, one uses an over-approximation by quantifying universally over all possible sets. By compactness, this universal second-order quantification can be translated into an existential first-order quantification, yielding the appropriate complexity.

\begin{definition}[Forcing question]
Let $c = (\sigma_0, \sigma_1,X_{n-1})$ be a $\DD_{n}^A$-condition and $\phi_{0}(G,x), \phi_1(G,x)$ two $\Pi_{n-1}^0$ formulas, the relation $$c \qvdash (\exists x)\phi_0(G_0,x) \vee (\exists x)\phi_1(G_1,x)$$
holds if for every partition $Z_0,Z_1$ of $X_{n-1}$, there exist some side $i < 2$, some $\rho \subseteq Z_i$ and some $x \in \NN$ such that either $\phi_i(\sigma_i \cup \tau,x)$ holds in the case where $n = 1$ or $\sigma_i \cup \rho \nqvdash \neg \phi_{i}(G_i,x)$ holds if $n > 1$.
\end{definition}

The following lemma shows that the forcing question for $\Sigma^0_n$-formulas has the good definitional complexity.
In particular, if $M_{n-1}$ is low over $\emptyset^{(n-1)}$, that is, $(M_{n-1} \oplus \emptyset^{(n-1)})' \leq_T \emptyset^{(n)}$, then the forcing question is $\emptyset^{(n)}$-decidable.

\begin{lemma}\label[lemma]{lem:disjunctive-forcing-question-complexity}
    The relation $c \qvdash (\exists x)\phi_0(G_0,x) \vee (\exists x)\phi_1(G_1,x)$ is $\Sigma_1^0(\M_{n-1})$.
\end{lemma}

\begin{proof}
By compactness, the statement $c \qvdash (\exists x)\phi_0(G_0,x) \vee (\exists x)\phi_1(G_1,x)$ is equivalent (in the case where $n > 1$) to the following:
$$(\exists \ell)(\forall Z_0 \cup Z_1 = X_{n-1} \uh \ell)(\exists i < 2)(\exists \rho \subseteq Z_i)(\exists x)\sigma_i \cup \rho \nqvdash \neg \phi_i(G_i,x)$$

which is $\Sigma^0_1(\M_{n-1})$ by \Cref{lem:core-complexity-question}. The case $n = 1$ is similar.
\end{proof}

The following lemma states that the disjunctive forcing meets its specifications.
Note that in the negative case, the negation is forced only on one side of the condition. The disjunction of the question should therefore not be considered as a single formula, but rather as two separate questions.

\begin{lemma}\label[lemma]{lem:disjunctive-question-find-extension}
Let $c = (\sigma_0, \sigma_1, X_{n-1})$ be a $\DD^A_{n}$-condition and let $\phi_0(G_0,x)$ and $\phi_1(G_1,x)$ two $\Pi_{n-1}^0$ formulas. 
\begin{itemize}
    \item If $c \qvdash (\exists x)\phi_0(G_0,x) \vee (\exists x)\phi_1(G_1,x)$ holds, then there exist some side $i < 2$ and some condition $d \leq c$ such that $d^{[i]} \Vdash (\exists x)\phi_{i}(G_i,x)$.
    \item If $c \nqvdash (\exists x)\phi_0(G_0,x) \vee (\exists x)\phi_1(G_1,x)$ holds, then there exist some side $i < 2$ and some condition $d \leq c$ such that $d^{[i]} \Vdash (\forall x)\neg \phi_{i}(G_i,x)$.
\end{itemize}
\end{lemma}
\begin{proof}

Assume $c \qvdash (\exists x)\phi_0(G_0,x) \vee (\exists x)\phi_1(G_1,x)$ holds. Let $Y_0 = A \cap X_{n-1}$ and $Y_1 = \overline{A} \cap X_{n-1}$ and consider the partition $X_{n-1} = Y_0 \cup Y_1$. There exists some side $i < 2$, some $\tau \subseteq Y_i$ and some $a \in \NN$ such that $\phi_i(\sigma_i \cup \tau,a)$ holds if $n = 1$ or $\sigma_i \cup \tau \nqvdash \neg \phi_{i}(G_i,a)$ holds if $n > 1$. As $\tau \subseteq A^i \cap X_{n-1}$, the condition $e_i = (\sigma_i \cup \tau, X_{n-1} \setminus \{0, \dots, |\tau|\}) \leq c^{[i]}$ is in $\PP_{n-1}^{A^i}$, hence, if $n = 1$, $e_i$ already forces $\phi_i(G_i,a)$ and if $n > 1$, using \Cref{lem:core-question-find-extension}, there exists some $d_i \leq e_i$ such that $d_i \Vdash \phi_{i}(G_i,a)$.

Therefore, in both case, using \Cref{lem:disjunctive-compatible-core} we have found some extension $d \leq c$ and some $i < 2$ such that $d^{[i]} \Vdash \phi_{i}(G_i,a)$, hence $d^{[i]} \Vdash (\exists x)\phi_i(G_i,x)$. \\

Assume $c \nqvdash (\exists x)\phi_0(G_0,x) \vee (\exists x)\phi_1(G_1,x)$ holds. The class of all partitions $X_{n-1} = Y_0 \cup Y_1$, such that for every side $i < 2$, every $\tau \subseteq Y_i$ and every $x \in \NN$, $\neg \phi_i(\sigma_i \cup \tau, x)$ holds if $n = 1$ or $\sigma_i \cup \tau \qvdash \neg \phi_{i}(G_i,x)$ holds if $n > 1$ is $\Pi_1^0(\M_{n-1})$ and non-empty. Therefore, as $\M_{n-1}$ is a Scott ideal, there exists such a partition $X_{n-1} = Y_0 \cup Y_1$ in $\M_{n-1}$. As $X_{n-1} \in \langle \U_{C_{n-1}}^{\M_{n-1}} \rangle$, there exists some $i < 2$ such that $Y_i \in \langle \U_{C_{n-1}}^{\M_{n-1}} \rangle$, then $d = (\sigma_0, \sigma_1, Y_i) \leq c$ is an extension duch that $d^{[i]} \Vdash (\forall x) \neg \phi_i(G_i,x)$.
\end{proof}

\subsection{Compact forcing questions}\label{sect:compact-forcing-questions}

As explained in \Cref{sect:iterated-jump-control}, there is a close relationship between preservation of hyperimmunities and the existence of compact forcing questions with good definitional properties.
For the purpose of \Cref{thm:main-preservation-hyps}, one needs to use notions of forcing whose forcing questions for $\Sigma_{k+1}^0$-formulas
are $\Sigma^0_{k+1}$-preserving and $\Sigma^0_{k+1}$-compact for $k \in \{n-1, n\}$.
The proof of \Cref{thm:main-preservation-hyps} uses the disjunctive forcing ($\DD^A_n$) if $A$ and $\overline{A}$ both belong to the appropriate partition regular class, and the witness forcing ($\WW^A_n$ or $\WW^{\overline{A}}_n$) if either fails. The disjunctive forcing inherits its forcing question for $\Sigma^0_{n+1}$-formulas from the main forcing ($\MM^A_n$ and $\MM^{\overline{A}}_n$) and defines a new disjunctive forcing question for $\Sigma^0_n$-formulas. The witness forcing defines new forcing questions for both $\Sigma^0_n$-formulas and $\Sigma^0_{n+1}$-formulas.

In this section, we study the forcing questions of the main forcing, the disjunctive forcing, and the witness forcing, and prove that $\MM^A_n$-forcing question for $\Sigma^0_{n+1}$-formulas, the disjunctive $\DD^A_n$-forcing question for $\Sigma^0_n$-formulas, and the $\WW^A_n$-forcing questions for both $\Sigma^0_n$-formulas and $\Sigma^0_{n+1}$-formulas are all compact.

\begin{remark}\label[remark]{rem:bounded-quantification-question}
$\Sigma^0_{n+1}$-compactness of a forcing question states that if $c \qvdash (\exists x)\phi(G, x)$ for some~$\Pi^0_n$-formula $\phi$, then there is some bound~$k \in \NN$ such that $c \qvdash (\exists x \leq k)\phi(G, x)$. Traditionally, bounded quantification is treated as syntactic sugar, where $(\exists x \leq y)\phi(x)$ is translated into $(\exists x)[x \leq y \wedge \phi(x)]$ and $(\forall x \leq y)\phi(x)$ becomes $(\forall x)[x \leq y \rightarrow \phi(x)]$. However, due to the lack of robustness of the forcing question, it might be that two logically equivalent formulas do not yield the same answer. 

To circumvent this issue, in the case of of $\Sigma^0_{n+1}$-compactness, bounded existentials will be treated natively, requiring to define the forcing question for formulas of the form $(\exists x \leq k)\phi(G, x)$ where $\phi$ is $\Pi^0_n$. The definition of $c \qvdash (\exists x \leq k)\phi(G, x)$ will be very similar to that of $c \qvdash (\exists x)\phi(G, x)$, except that the existential quantifier for~$x$ will be bounded by $k$ accordingly. The lemma stating that the forcing question meets its specifications remains true when working with bounded formulas.
\end{remark}

We start with the main forcing.

\begin{lemma}\label[lemma]{lem:main-question-compact}
In $\MM^A_n$-forcing, the $\qvdash$ relation for $\Sigma_{n+1}^0$ formulas is compact, i.e., if $c \qvdash (\exists x)  \phi(G,x)$ for some $\Pi_{n}^0$ formula $\phi(G,x)$, then there exists some bound $\ell \in \NN$ such that  $c \qvdash (\exists x \leq \ell)\phi(G,x)$.
\end{lemma}

\begin{proof}
Let $c = (\sigma, X_{n-1})$ and assume $c \qvdash (\exists x)  \phi(G,x)$ holds for some $\Pi_{n}^0$ formula $\phi(G,x)$. There exists some $a \in \NN$ and some $\rho \subseteq X_{n-1} \cap A$ such that $\sigma \cup \rho \nqvdash \neg \phi(G,a)$, therefore  $c \qvdash (\exists x \leq a)\phi(G,x)$ . 
\end{proof}

\begin{lemma}\label[lemma]{lem:main-diagonalization}
For every $\MM^A_{n}$-condition~$c$, every $M_n$-hyperimmune function~$f_n$ and every Turing index~$e \in \NN$, there is an extension~$d \leq c$ forcing $\Phi^{G^{(n)}}_e$ not to dominate~$f_n$.
\end{lemma}

\begin{proof}
If there exists some $x \in \NN$ such that $c \nqvdash \Phi_e^{G^{(n)}}(x) \converge$, then by \Cref{lem:main-question-find-extension}, there exists some extension $d \leq c$ forcing $\Phi_e^{G^{(n)}}(x) \diverge$, hence forcing $\Phi_e^{G^{(n)}}$ to be partial.

If for all $x \in \NN$, $c \qvdash \Phi_e^{G^{(n)}}(x)\converge$, then by \Cref{lem:main-question-compact}, for every~$x$, there is a bound~$\ell_x$ such that $c \qvdash \Phi_e^{G^{(n)}}(x) \converge \leq \ell_x$. By assumption, the function $x \mapsto \ell_x$ is total, and by \Cref{lem:main-complexity-sigma0n+1-question}, it is $M_n$-computable, and therefore does not dominate $f_n$. Take some $x \in \NN$ such that $f_n(x) > \ell_x$ and by \Cref{lem:main-question-find-extension} let $d \leq c$ be an extension forcing $\Phi_e^{G^{(n)}}(x) \converge \leq  \ell_x$. Therefore, $d$ forces $\Phi_e^{G^{(n)}}(x) \converge < f_n(x)$, hence forces $\Phi_e^{G^{(n)}}$ not to dominate $f_n$.    
\end{proof}

We now turn to the study of the disjunctive forcing question for $\Sigma^0_n$-formulas in $\DD^A_n$.

\begin{lemma}\label[lemma]{lem:disjunctive-question-compact}
    Let $c = (\sigma_0, \sigma_1, X_{n-1})$ be a $\DD^A_{n}$-condition and let $\phi_0(G,x)$ and $\phi_1(G,x)$ be two $\Pi_{n-1}^0$ formulas. If 
    $$c \qvdash (\exists x)\phi_0(G_0,x) \vee (\exists x)\phi_1(G_1,x)$$
    then there exists some bound $\ell \in \NN$ such that 
    $$c \qvdash (\exists x \leq \ell)\phi_0(G,x) \vee (\exists x \leq \ell)\phi_1(G_1,x)$$
\end{lemma}

\begin{proof}
If $c \qvdash (\exists x)\phi_{0}(G_0,x) \vee (\exists x)\phi_{1}(G_1,x)$, then, by compactness, there exists some bound $\ell \in \NN$ such that every partition $X_{n-1} \uh \{0,\dots \ell\} = Y_0 \cup Y_1$ there exists some side $i < 2$, some $\tau \subseteq Y_i$ and some $x \in \NN$ such that $\phi_{i}(\sigma_i \cup \tau,x)$ holds if $n = 1$ or $\sigma_i \cup \rho \nqvdash \neg \phi_{i}(G_i,x)$ holds if $n > 1$. There are only finitely many such partitions, hence finitely many such witnesses $x$. Therefore, there exists some bound $\ell$ such that $c \qvdash (\exists x \leq \ell)\phi_0(G,x) \vee (\exists x \leq \ell)\phi_1(G_1,x)$.
\end{proof}

\begin{lemma}\label[lemma]{lem:disjunctive-diagonalization}
For every $\DD^A_{n}$-condition $c = (\sigma_0, \sigma_1, X_{n-1})$, every $M_{n-1}$-hyper\-immune function $f_{n-1}$ and every pair of Turing indices $e_0,e_1 \in \NN$, there exist a side $i < 2$ and an extension $d \leq c$ forcing $\Phi_{e_i}^{G_i^{(n-1)}}$ not to dominate $f_{n-1}$.    
\end{lemma}

\begin{proof}
If there exists some $x \in \NN$ such that $c \nqvdash \Phi_{e_0}^{G_0^{(n-1)}}(x) \converge \vee \Phi_{e_1}^{G_1^{(n-1)}}(x) \converge$, then by \Cref{lem:disjunctive-question-find-extension}, there exist some side $i < 2$ and some extension $d \leq c$ such that $d \Vdash \Phi_{e_i}^{G_i^{(n-1)}}(x) \diverge$, hence, forces $\Phi_{e_i}^{G_i^{(n-1)}}$ to be partial. 

If for all $x \in \NN$, $c \qvdash \Phi_{e_0}^{G_0^{(n-1)}}(x) \converge \vee \Phi_{e_1}^{G_1^{(n-1)}}(x) \converge$, then by \Cref{lem:disjunctive-question-compact}, for every~$x$, there is a bound $\ell_x \in \NN$ such that $c \qvdash \Phi_{e_0}^{G_0^{(n-1)}}(x) \converge \leq \ell_x \vee \Phi_{e_1}^{G_1^{(n-1)}}(x) \converge \leq \ell_x$. The function $x \mapsto \ell_x$ is total and by \Cref{lem:disjunctive-forcing-question-complexity}, it is $M_{n-1}$-computable, and therefore does not dominate $f_{n-1}$. Take some $x \in \NN$ such that $f_{n-1}(x) > \ell_x$ and by \Cref{lem:disjunctive-question-find-extension} let $i < 2$ and let $d \leq c$ be an extension forcing $ \Phi_e^{G_i^{(n-1)}}(x)\converge \leq \ell_x$. Then, $d$ forces $\Phi_{e_i}^{G_i^{(n-1)}}(x) \converge < f_{n-1}(x)$, hence forces $\Phi_{e_i}^{G_i^{(n-1)}}$ not to dominate $f_{n-1}$.  
\end{proof}

Last, we study the forcing questions for $\Sigma^0_n$ and $\Sigma^0_{n+1}$ formulas in the witness forcing.

\begin{lemma}\label[lemma]{lem:witness-question-compact}
In $\WW^A_n$-forcing, the $\qvdash^{\U}$ relation for $\Sigma_{n}^0$ and $\Sigma_{n+1}^0$ formulas is compact, i.e., for $k = n-1$ or $k = n$, for every $\WW^A_n$-condition $c = (\sigma,X_{n-1},X_{n})$, every $\Pi_{k}^0$ formula $\phi(G,x)$, and every $\U$ witness for $c$
, if $c \qvdash^{\U} (\exists x) \phi(G,x)$, then there exists some bound $\ell \in \NN$ such that $c \qvdash^{\U} (\exists x \leq \ell) \phi(G,x)$.
\end{lemma}

\begin{proof}
Assume $c \qvdash^{\U} (\exists x) \phi(G,x)$. By compactness, there exists some bound $\ell \in \NN$ such that for every $\beta \in 2^{\ell}$, letting $\overline{\beta}$ be the bitwise complement of~$\beta$, either we already have  $[\,\overline{\beta}\,] \subseteq \U$ or there is some $x \leq \ell$ and $\rho \finsub \beta \cap X_{k}$ such that $\phi(\sigma \cup \rho,x)$ holds if $k = 1$ or $\sigma \cup \rho \nqvdash \neg \phi(G,x)$ holds if $k > 1$.
Then $c \qvdash^{\U} (\exists x \leq \ell) \phi(G,x)$.
\end{proof}


\begin{lemma}\label[lemma]{lem:witness-diagonalization}
Let $k = n - 1$ or $k = n$, for every $\WW^A_n$-condition $c = (\sigma,X_{n-1},X_n)$ in $\WW^A_n$, every $M_{k}$-hyperimmune function $f_{k}$ and every Turing index $e \in \NN$, there exists an extension $d \leq c$ forcing $\Phi_{e}^{G^{(k)}}$ not to dominate $f_{k}$.    
\end{lemma}

\begin{proof}
If there exists some $x \in \NN$ such that $c \nqvdash \Phi_e^{G^{(k)}}(x) \converge$, then by \Cref{lem:witness-question-find-extension}, there exists some extension $d \leq c$ forcing $\Phi_e^{G^{(k)}}(x) \diverge$, hence forcing $\Phi_e^{G^{(k)}}$ to be partial.

If for all $x \in \NN$, $c \qvdash \Phi_e^{G^{(k)}}(x)\converge$, then by \Cref{lem:witness-question-compact}, for every~$x$, there is a bound $\ell_x \in \NN$ such that $c \qvdash \Phi_e^{G^{(k)}}(x) \converge \leq \ell_x$. By assumption, the function $x \mapsto \ell_x$ is total and by \Cref{lem:witness-forcing-question-complexity}, it is $M_{k}$-computable, and therefore does not dominate $f_{k}$. Take some $x \in \NN$ such that $f_{k}(x) > \ell_x$ and by \Cref{lem:witness-question-find-extension} let $d \leq c$ be an extension forcing $\Phi_e^{G^{(k)}}(x) \converge \leq \ell_x$. Therefore, $d$ forces $\Phi_e^{G^{(k)}}(x) \converge < f_{k}(x)$, hence forces $\Phi_e^{G^{(k)}}$ not to dominate $f_{k}$.      
\end{proof}


\subsection{Applications}\label{sect:preservation-apps}

We now have all the necessary ingredients to prove our next main theorem about simultaneous preservation of hyperimmunities, enabling to separate $\Sub{\Delta^0_{n+1}}$ from $\Sub{\Sigma^0_{n+1}}$ over $\omega$-models.

\begin{repmaintheorem}{thm:main-preservation-hyps}
Fix~$n \geq 1$. For every $\Delta^0_{n+1}$ set~$B$, every $\emptyset^{(n-1)}$-hyperimmune function $f : \NN \to \NN$ and every $\emptyset^{(n)}$-hyperimmune function $g : \NN \to \NN$, there is an infinite set~$H \subseteq B$ or $H \subseteq \overline{B}$ such that $f$ is $H^{(n-1)}$-hyperimmune and $g$ is $H^{(n)}$-hyperimmune.
\end{repmaintheorem}
\begin{proof}
For every $X$, let $\C(X)$ be the $\Pi^0_1(X)$ class of \Cref{prop:pa-to-scott}.
By \Cref{lem:scott-tower}, there is a Scott tower $\M_0, \dots, \M_{n-2}$ of height~$n-2$ with Scott codes~$M_0, \dots, M_{n-2}$, such that for every~$i \leq n-2$, $M_i$ is of low degree over~$\emptyset^{(i)}$.
By relativizing \Cref{thm:wkl-hyperimmunity-preservation} to $\emptyset^{(n-1)}$ and applying it on the family $\{(f,0),(g,1)\}$, there is a Scott code~$M_{n-1} \in \C(\emptyset^{(n-1)})$ of a Scott ideal~$\M_{n-1}$ such that $f$ is $M_{n-1}$-hyperimmune and $g$ is $M_{n-1}'$-hyperimmune. In particular, $M_{n-2}' \leq_T \emptyset^{(n-1)} \in \M_{n-1}$.
By the computably dominated basis theorem~\cite{jockusch1972classes} relativized to $M_{n-1}'$, there exists some Scott code $M_n \in C(M_{n-1}')$ of a Scott ideal~$\M_n$ containing $M_{n-1}'$, and such that every $M_n$-computable function is dominated by an $M_{n-1}'$-computable function. In particular, since $g$ is $M_{n-1}'$-hyperimmune, $g$ is $M_n$-hyperimmune. The Scott ideals $\M_0, \dots, \M_n$ therefore form a Scott tower such that $f$ is $M_{n-1}$-hyperimmune, and $g$ is $M_n$-hyperimmune.

By \Cref{lem:largeness-tower}, it can be enriched with some sets~$C_0, \dots, C_{n-1}$ to form a largeness tower of height $n$. There are two cases:
\bigskip

\textbf{Case 1:} Both $B$ and $\overline{B}$ are in $\langle \U_{C_{n-1}}^{\M_{n-1}} \rangle$. We will build an infinite subset of~$B$ or $\overline{B}$ using $\DD_n^B$-forcing.
Given $e_0,e_1 \in \NN$, let $\D_{e_0,e_1}$ be the set 
$$\{c \in \DD_n^B : (\exists i < 2) c \Vdash \Phi_{e_i}^{G_i^{(n-1)}} \mbox{ does not dominate } f_{n-1} \}$$
By \Cref{lem:disjunctive-diagonalization}, $\D_{e_0,e_1}$ is dense for every~$e_0, e_1 \in \NN$.
Therefore, for $\F$ a sufficiently generic $\DD_n^B$-filter, for every $e_0,e_1 \in \NN$, either $\Phi_{e_0}^{G_{0,\F}^{(n-1)}}$ or $\Phi_{e_1}^{G_{1,\F}^{(n-1)}}$ does not dominate $f_{n-1}$. By a standard pairing argument, there exists some side $i < 2$ such that $\Phi_{e}^{G_{i,\F}^{(n-1)}}$ does not dominate $f_{n-1}$ for any $e \in \NN$, hence $f_{n-1}$ is $G_{i,\F}^{(n-1)}$-hyperimmune. 

For $e \in \NN$, let $\D_e$ be the set
$$\{c \in \MM^{B^i}_n : c \Vdash \Phi_{e}^{G_i^{(n)}} \mbox{ do not dominate } f_n \}$$

By \Cref{lem:main-diagonalization}, the set $\D_e$ is dense for every $e \in \NN$. Hence, $\F$ being a sufficiently generic $\DD^B_n$-filter, the $\MM_n^{B^i}$-filter $\F^{[i]} = \{c^{[i]} : c \in \F\}$ is sufficiently generic by \Cref{lem:disjunctive-compatible-core}. Then, by \Cref{lem:core-question-find-extension} and \Cref{lem:main-question-find-extension}, $\F^{[i]}$ is $n$-generic. In particular, by \Cref{prop:core-forcing-generic-set}, $G_{i,\F}$ exists and is an infinite subset of~$B^i$, where~$B^0 = B$ and $B^1 = \overline{B}$. By \Cref{prop:core-forcing-imply-truth} and \Cref{prop:main-forcing-imply-truth}, every property forced for the set $G_{i,\F}$ is true. Thus, $H = G_{i,\F}$ is an infinite subset of $B^i$ such that $f_{n-1}$ is $H^{(n-1)}$-hyperimmune and $f_n$ is $H^{(n)}$-hyperimmune.\\

\textbf{Case 2:} $\overline{B}$ is not in $\langle \U_{C_{n-1}}^{\M_{n-1}} \rangle$ (the case $B \notin \langle \U_{C_{n-1}}^{\M_{n-1}} \rangle$ is symmetrical). We will build an infinite subset of~$B$ using $\WW^B_n$-forcing. For $e \in \NN$, let $\C_e$ be the set
$$\{c \in \WW^B_n : c \Vdash \Phi_{e}^{G^{(n-1)}} \mbox{ do not dominate } f_{n-1} \}$$

And let $\D_e$ be the set
$$\{c \in \WW^B_n : c \Vdash \Phi_{e}^{G^{(n)}} \mbox{ do not dominate } f_{n} \}$$

By \Cref{lem:witness-diagonalization}, $\C_e$ and $\D_e$ are dense for every~$e \in \NN$.
Let $\F$ be a sufficiently generic filter, so $\F$ intersect every $\C_e$ and $\D_e$. By \Cref{lem:core-question-find-extension} (which can be applied to $\WW_n^A$-conditions thanks to \Cref{lem:witness-compatible-core})  and \Cref{lem:witness-question-find-extension}, $\F$ will be $n$-generic. Thus, by \Cref{prop:core-forcing-imply-truth} and \Cref{prop:witness-forcing-imply-truth}, every property forced for the set $G_{\F}$ will hold. Moreover, $G_\F$ is an infinite subset of $B$ by \Cref{prop:core-forcing-generic-set}. Thus, we have found an infinite subset $H = G_{\F}$ of $B$ such that $f_{n-1}$ is $H^{(n-1)}$-hyperimmune and $f_n$ is $H^{(n)}$-hyperimmune.

\end{proof}

\section{Conservation theorems}\label{sect:conservation}

We now turn to the last main application of the previous notions of forcing: conservation theorems.
The goal of this section is to prove the following main theorem:

\begin{repmaintheorem}{thm:main-conservation}
Let $n \geq 1$. $\RCA_0 + \ISig_{n+1} + \Sub{\Sigma^0_{n+1}}$ is $\Pi_1^1$ conservative over $\RCA_0 + \ISig_{n+1}$.
\end{repmaintheorem}

The techniques will be very similar to the standard realm, but working with a formalized version of the notions of forcing in models of weak arithmetic. We start with a short introduction to the techniques of $\Pi^1_1$-conservation over~$\RCA_0 + \ISig_{n+1}$, then prove that the previous notions of forcing satisfy the necessary combinatorial features to preserve induction, and finally prove \Cref{thm:main-conservation}.

\subsection{Conservation over~$\RCA_0 + \ISig_{n+1}$}

The standard model-theoretic approach for proving that a theory $T_2$ is $\Gamma$-conservative over a theory $T_1$ is the following:

\begin{itemize}
    \item First, assume by contrapositive that $T_1$ does not prove some formula $\phi \in \Gamma$.
    \item Using Gödel completeness theorem, there exists some model $\M$ of $T_1 + \neg \phi$.
    \item From this model, construct another model $\N$, this time of $T_2 + \neg \phi$.
    \item Therefore, $T_2$ does not prove $\phi$ either.
\end{itemize} 

The heart of these proofs lies in the construction of the model of $T_2 + \neg \phi$. In this article, $\Gamma$ will be the set of all the $\Pi_1^1$-formulas. This allows to easily ensure that the model constructed will be a model of $\neg \phi$: this came free if the model constructed is an $\omega$-extension of the initial model.

\begin{definition}
A model $\N = (N, T, +^\N, \times^\N, <^\N,0^{\N},1^\N)$ is an \emph{$\omega$-extension}\footnote{The terminology \qt{$\omega$-extension} should not be confused with the notion of \qt{$\omega$-model}. Indeed, if $\N$ is an $\omega$-extension of a non-standard model~$\M$, then neither $\N$ nor $\M$ are $\omega$-models.} of a model $\M = (M, S, +^\M, \times^\M, <^\M,0^\M,1^\M)$ if $N = M$, $+^\N = +^\M$, $\times^\N = \times^\M$, $<^\N = <^\M$, $0^\N = 0^\M$, $1^\N = 1^\M$ and $T \supseteq S$.
\end{definition}

In other words, a model~$\N$ is an $\omega$-extension of~$\M$ if $\N$ is obtained from~$\M$ by adding new sets to the second-order part, and leaving the first-order part unchanged. The following lemma states that $\Sigma^1_1$-formulas are left unchanged by considering $\omega$-extensions.

\begin{lemma}\label{lem:omega-extension-absolute}
    Let $\phi$ be a $\Pi_1^1$-formula and $\M$ of model of $\neg \phi$. If $\N$ is an $\omega$-extension of $\M$, then $\N \models \neg \phi$.      
\end{lemma}

\begin{proof}
Write $\phi = (\forall X)\theta(X)$ for $\theta(X)$ an arithmetical formula and write $\M = (M,S)$. By assumption, $\M \models (\exists X)\neg \theta(X)$, thus there exists some set $A \in S$ such that $\M \models \neg \theta(A)$.

An $\omega$-extension $\N$ of $\M$ can be written as $\N = (M,\hat S)$ with $\hat S \supseteq S$, hence $A \in \hat S$. As $\theta(A)$ is an arithmetical formula, its truth value only depends on the first order part of the model, therefore, we also have $\N \models \neg \theta(A)$ and $\N \models \neg (\forall X) \theta(X)$.
\end{proof}

When the theory $T_1$ is included in the theory $T_2$, we say that $T_2$ is a \emph{conservative extension} of $T_1$. In our case, the theory $T_1$ will be $\RCA_0 + \ISig_n$ and the theory $T_2$ will be equal to $T_1 + \Psf$ for some $\Pi_2^1$-problem~$\Psf$. Having to preserve a $\Pi_2^1$ problem allows us to further break down the construction step.

\begin{proposition}\label{prop:pi12-problem-conservation}
    Let $\Psf$ be a $\Pi_2^1$-problem and $T$ be a theory composed solely of $\Pi_2^1$ axioms. 
    
    Assume that for every countable model $\M = (M,S)$ of $T$ and every instance $X \in S$ of $\Psf$, there exists some countable $\omega$-extension $\M'$ of $\M$ containing a solution to $X$ and such that $\M' \models T$. Then there exists a countable $\omega$-extension $\N$ of $\M$ such that $\N \models T + \Psf$.
\end{proposition}

\begin{proof}
The model $\N$ will be defined as $(M,\bigcup_{n \in \omega} S_n)$ for $\M_0 = (M,S_0) \subseteq \M_1 = (M,S_1) \subseteq \dots$ a sequence of $\omega$-extensions obtained recursively using the assumption and having the following properties:
\begin{enumerate}
    \item $\M_0 = \M$ and  $\M_k \models T$ for every $k \in \omega$. 
    \item $\M_k$ is countable for every $k \in \omega$. 
    \item For every instance $X$ of $\Psf$ contained in some $\M_k$, there exists some $\ell \in \omega$ such that $\M_{\ell}$ contains a solution to $X$.
\end{enumerate}
 

\emph{Claim 1: $\N \models \Psf$.} Let $X$ be an instance of $\Psf$ belonging to $\N$. By definition of $\N$, $X$ belongs to some $\M_k$, thus there exists some index $\ell$ such that $\M_{\ell}$, contains a solution to $X$, therefore, $\N$ also contains it. Hence, $\N \models \Psf$.

\emph{Claim 2: $\N \models T$.} Let $\phi \in T$ be a $\Pi_2^1$ formula of the form $(\forall X)(\exists Y)\theta(X,Y)$ with $\theta(X,Y)$ an arithmetical formula. Let $A \in \M$ be a set, $A$ belongs to some $S_k$ for $k \in \omega$. Since $\M_k \models T$, there exists some $B \in S_k$ such that $\M_k \models \theta(A,B)$, thus $\N \models \theta(A,B)$ as $\theta(X,Y)$ is an arithmetical formula and $\N$ is an $\omega$-extension of $\M_k$. The set $A$ being arbitrary, $\N \models T$.
\end{proof}

The theory $\RCA_0 + \ISig_n$ is composed of $\Pi_2^1$ axioms, so the result may be applied. The assumption that the initial model has to be countable is not a problem thanks to the downward L\"owenheim-Skolem theorem. 

So, the problem of showing that a $\Pi_2^1$-problem $\Psf$ is $\Pi_1^1$-conservative over the base theory $\RCA_0 + \ISig_n$ has been reduced to showing that for any countable model $\M = (M,S)$ of $\RCA_0 + \ISig_n$ and any instance $X \in S$ of $\Psf$, there is a countable $\omega$-extension $\M' = (M,S')$ of $\M$ such that $\M' \models \RCA_0 + \ISig_n$ and $S'$ contains a solution $Y$ to $X$. \\

The only axioms of $\RCA_0 + \ISig_n$ stating the existence of sets are the one from the $\Delta^0_1$-comprehension scheme, so when adding a solution $Y$ to the structure $\M$, the only other subsets that need to be added are the one computable using $Y$ and the elements of $\M$.

\begin{definition}
    For $\M = (M,S)$ a structure and $G \subseteq M$, write $\M \cup \{G\}$ for the $\omega$-extension of~$\M$ whose second-order part is $S \cup \{G\}$, and $\M[G]$ for the $\omega$-extension of $\M$ containing all the sets $\Delta_1^0$-definable using parameters in $\M \cup \{G\}$.
\end{definition}

For $Y \subseteq M$ a solution to the instance, the structure $\M \cup \{G\}$ does not necessarily satisfy the $\Delta^0_1$-comprehension scheme, as it may not contain all the $Y$-computable sets. This is not the case for the structure $\M[Y]$, as shown by the following lemma. 

\begin{lemma}[Folklore]\label{lem:formula-translation}
    Let $\M = (M,S) \models \RCA_0$ and $G \subseteq M$. Then $\M[G]$ satisfies the $\Delta^0_1$-comprehension scheme.
\end{lemma}
\begin{proof}[Proof sketch]
It suffices to prove that every $\Sigma^0_1(\M[G])$-formula can be translated into a $\Sigma^0_1(\M \cup \{G\})$-formula. Then, any $\Delta^0_1(\M[G])$ predicate is $\Delta^0_1(\M \cup \{G\})$, hence belongs to~$\M[G]$.

Let $\phi(x)$ be a $\Sigma^0_1(\M[G])$-formula. Let $X$ be a second-order parameter in~$\phi$ belonging to~$\M[G]$. By definition of $\M[G]$, $X$ is $\Delta^0_1(\M \cup \{G\})$-definable, so there are two $\Delta^0_0(\M \cup \{G\})$-formulas $\theta(x, y)$ and $\zeta(x, y)$ such that $\M \cup \{G\} \models \forall x(\exists y\theta(x,y) \leftrightarrow \forall y\zeta(x, y))$. Every occurrence of the atomic formula~$x \in X$ in $\phi$ can be either replaced by $\exists y \theta(x, y)$ or $\forall y \zeta(x, y)$, so that the resulting formula $\hat\phi$ is a again $\Sigma^0_1$. One can iterate the operation for every second-order parameter in $\M[G] \setminus \M \cup \{G\}$ appearing in~$\phi$.
\end{proof}

So, the $\omega$-extension $\M'$ can always be assumed to be of the form $\M[Y]$ for some solution $Y$. \\

The only thing that remains is to find $Y$ such that $\M[Y]$ satisfies $\ISig_n$. To do so, some care is needed, as choosing an inadequate $Y \subseteq M$ may break induction: for example, if $Y$ computes some $\Sigma_1^0$-cut of $M$, no $\omega$-extension of $\M$ can contain $Y$ and be a model of $\ISig_1$. The following lemma shows that the problem of preserving induction for $\M[Y]$ can be reduced to the problem of preserving induction for $\M \cup \{Y\}$.

\begin{lemma}[Folklore]\label{lem:induction-to-generated}
    Let $\M = (M,S) \models \RCA_0 + \ISig_n$ and $G \subseteq M$. If $\M \cup \{G\} \models \ISig_n$, then $\M[G] \models \ISig_n$.
\end{lemma}

\begin{proof}[Proof sketch]
Let $\phi(x)$ be a $\Sigma_n^0$ formula with set parameters from $\M[G]$. A similar proof to that of \Cref{lem:formula-translation} yields that $\phi(x)$ is equivalent in $\M$ to a formula $\psi(x)$ with parameters from $\M \cup \{G\}$.
Then, as $\M \cup \{G\} \models \ISig_n$, $\psi(x)$ will satisfy that 
$$\psi(0) \wedge (\forall x \psi(x) \to \psi(x+1)) \to (\forall y \psi(y))$$
thus this is also the case for $\phi(x)$ and $\M[G] \models \ISig_n$.
\end{proof}

Using \Cref{thm:abstract-preservation-isig1}, finding such a $Y$ can be done with a $\Sigma^0_n$-preserving $(\Sigma^0_n,\Pi^0_n)$-merging forcing question. \\

Thanks to the following theorem amalgamation theorem of Yokoyama~\cite{yokoyama2009conservativity}, the proofs that some $\Pi_2^1$-theorem $\Psf$ is a $\Pi^1_1$-conservative extensions of $\RCA_0 + \ISig_n$ are quite modular and can easily be combined.

\begin{theorem}[\cite{yokoyama2010conservativity}]\label{thm:amalgamation-pi11}
Let~$T_0, T_1, T_2$ be $\Pi^1_2$ theories such that $T_0 \supseteq \RCA_0$ and $T_1$ and $T_2$ are both $\Pi^1_1$-conservative extensions of~$T_0$. Then $T_1 + T_2$ is a $\Pi^1_1$-conservative extension of~$T_0$.
\end{theorem}

\subsection{Merging forcing questions}

As mentioned in \Cref{sect:iterated-jump-control}, preservation of $\Sigma^0_n$-induction is closely related to the existence of a $\Sigma^0_n$-preserving, $(\Sigma^0_n, \Pi^0_n)$-merging forcing question. A forcing question is \emph{left $\Sigma^0_n$-extremal} if for every condition~$c$ and every $\Sigma^0_n$-formula~$\varphi(G)$, $c \qvdash \varphi(G)$ if and only if $c$ forces $\varphi(G)$. It is \emph{right $\Sigma^0_n$-extremal} if for every condition~$c$ and every $\Sigma^0_n$-formula~$\varphi(G)$, $c \qvdash \varphi(G)$ if and only if $c$ does not force $\neg \varphi(G)$.
Any left or right $\Sigma^0_n$-extremal forcing question is $(\Sigma^0_n, \Pi^0_n)$-merging. 

The goal being to prove that $\RCA_0 + \ISig_{n+1} + \Sub{\Sigma^0_{n+1}}$ is $\Pi^1_1$-conservative over~$\RCA_0 + \ISig_{n+1}$, one needs to consider the merging properties of the forcing question for $\Sigma^0_{n+1}$-formulas for the main and witness forcing.
The forcing question for $\Sigma^0_{n+1}$-formulas in the $\MM^A_n$-forcing is right $\Sigma^0_{n+1}$-extremal, hence is $(\Sigma^0_{n+1}, \Pi^0_{n+1})$-merging. On the other hand, the forcing question for $\Sigma^0_{n+1}$-formulas in the $\WW^A_n$-forcing is not $\Sigma^0_{n+1}$-extremal. The following lemma shows that it is however somehow $(\Sigma^0_{n+1}, \Pi^0_{n+1})$-merging, by considering the appropriate witnesses.

\begin{definition}
    Given a class~$\A \subseteq \cs$, let~$\L_2(\A)$ be the class 
$$\{X : (\forall X_0 \cup X_1 \supseteq X)(\exists i < 2) X_i \in \A\}.$$
\end{definition}

Note that if the class $\A$ is open, then so is $\L_2(\A)$.
Moreover, there is a computable function $g : \NN \to \NN$ such that $\U_{g(e)}^X = \L_2(\U_e^X)$ for every index~$e \in \NN$ and every oracle~$X$. Also note that if $\A$ is large, then so is $\L_2(\A)$. Indeed, $\L(\A) \subseteq \L_2(\A)$.

\begin{lemma}\label[lemma]{lem:witness-question-pair}
Let $c = (\sigma, X_{n-1},X_{n})$ be a $\WW^A_n$-condition and let $\U$ be a witness for $c$.
Let $\varphi(G), \psi(G)$ be two $\Sigma^0_{n+1}$ formulas such that
$c \qvdash^{\L_2(\U)} \varphi(G)$ and $c \nqvdash^{\U} \psi(G)$. Then there exists an extension~$d \leq c$
such that $d \Vdash \varphi(G)$ and $d \Vdash \neg \psi(G)$.
\end{lemma}

\begin{proof}
Since $c \nqvdash^{\U} \psi(G)$, the proof of \Cref{lem:witness-question-find-extension} gives us a set $B \in \M_n$ such that $\overline{B} \notin \U$ and such that $e = (\sigma,X_{n-1},X_n \cap B)$ is a valid condition extending $c$ and forcing $\neg \psi(G)$. \\

Write $\varphi(G) = (\exists x)\theta(G,x)$ for some $\Pi_n^0$ formula $\theta(G,x)$. We claim that $e \qvdash^{\U} \varphi(G)$: let $D \in 2^{\NN}$ be such that $\overline{D} \notin \U$. For every such $D$, $\overline{D} \cup \overline{B} \notin \L_2(\U)$ (otherwise, either $\overline{D}$ or $\overline{B}$ would be in $\U$), hence, as $c \qvdash^{\L_2(\U)} \varphi(G)$, there exist some finite $\tau \subseteq D \cap (X_n \cap B)$ and some $x \in \NN$ such that $\sigma \cup \tau \nqvdash \neg \theta(G,x)$. Therefore, $e \qvdash^{\U} \varphi(G)$ and by \Cref{lem:witness-question-find-extension}, there exists some extension $d \leq e$ forcing $\varphi(G)$.
\end{proof}

\subsection{Applications}

We are now ready to prove \Cref{thm:main-conservation}. As explained above, thanks to \Cref{prop:pi12-problem-conservation}, it is reduced to proving that any countable model of $\RCA_0 + \ISig_{n+1}$ can be $\omega$-extended into another model of~$\RCA_0 + \ISig_{n+1}$ containing a solution to a fixed instance of $\Sub{\Sigma^0_{n+1}}$.

\begin{proposition}\label[proposition]{prop:cohesive-class-conservation}
Let $n \geq 1$ and consider a countable model $\M = (M,S) \models \RCA_0 + \ISig_{n+1}$ topped by a set $Y \in S$. There exist some $M_0,M_1, \dots, M_n \subseteq M$ and some $C_0, \dots, C_{n-1} \subseteq M$ such that:
    \begin{itemize}
        \item $\M[M_i] \models \RCA_0 + \ISig_{n+1-i}$ for $i \leq n$ ;
        \item $M_i$ is a Scott code of $\M_i \models \WKL_0 + \ISig_{n+1-i}$ for $i \leq n$ ;
        \item $Y \in \M_0$ and $M_i' \oplus C_i \in \M_{i+1}$ for $i < n$ ;
        \item $\U_{C_i}^{\M_i}$ is $\M_i$-cohesive for $i < n$ ;
        \item $\U_{C_{i+1}}^{\M_{i+1}} \subseteq \langle \U_{C_i}^{\M_i} \rangle$ for $i < n-1$.
    \end{itemize}
\end{proposition}
\begin{proof}
By a formalization of \Cref{prop:pa-to-scott} in second-order arithmetic (see for example Fiori-Carones et al.~\cite[Lemma 3.2]{fiori2021isomorphism}), for every set~$X$, there is an infinite $\Delta^0_1(X)$-definable tree~$T(X) \subseteq 2^{<M}$ such that if~$M_0 \in [T(X)]$, then, $M_0$ is a Scott code of an ideal containing~$X$.

By H\'ajek~\cite{hajek1993interpretability},
there is a path~$M_0 \in [T(Y)]$ such that $\M[M_0] \models \RCA_0 + \ISig_{n+1}$. Let $\M_0$ be the Scott ideal coded by~$M_0$, since $\M[M_0] \models \ISig_{n+1}$, then $\M_0 \models \WKL_0 + \ISig_{n+1}$. 
Assume $M_i$ has been defined for $i < n$. Note that $\M[M_i'] \models \ISig_{n-i}$.
Again by H\'ajek~\cite{hajek1993interpretability} if $i < n-1$ and by Harrington (see \cite[Lemma 8.2]{cholak2001strength} if $i = n-1$, there is a set~$M_{i+1} \in [T(M_i')]$ coding a Scott ideal~$\M_{i+1}$
such that $\M[M_{i+1}] \models \ISig_{n-i}$. In particular, $\M_{i+1} \models \WKL_0 + \ISig_{n-i}$ and $M_i' \in \M_{i+1}$.

By a formalization of \Cref{lem:largeness-tower}, there are sets $C_0, \dots, C_{n_1}$ such that
$C_i \in \M_{i+1}$ and $\U_{C_i}^{\M_i}$ is $\M_i$-cohesive for every~$i < n$, and $\U_{C_{i+1}}^{\M_{i+1}} \subseteq \langle \U_{C_i}^{\M_i} \rangle$ for $i < n-1$.





\end{proof}

Given a condition~$c$ and a formula~$\varphi(G, x)$, we say that $c$ \emph{forces $\varphi(G, x)$ to satisfy induction} if either $c$ forces $\forall x \varphi(G, x)$, or $c$ forces $\neg \varphi(G, 0)$, or there is some~$a > 0$ such that $c$ forces $\neg \varphi(G, a)$ and forces $\varphi(G, a-1)$.

\begin{proposition}\label[proposition]{lem:sigman+1-conservation-step}
    Let $n > 0$ and consider a countable model $\M = (M,S) \models \RCA_0 + \ISig_{n+1}$ topped by a set $Y$, and let $A \subseteq M$ be $\Sigma_{n+1}^0$ in $\M$. Then there exists some infinite set $H \subseteq M$ such that $H \subseteq A$ or $H \subseteq \overline{A}$ and such that $\M[H] \models \RCA_0 + \ISig_{n+1}$.
\end{proposition}

\begin{proof}
Let $M_0, \dots, M_n \subseteq M$ and $C_0, \dots, C_{n-1} \subseteq M$ be the sets obtained from \Cref{prop:cohesive-class-conservation}. There are two cases:
\bigskip

\textbf{Case 1:} $A \in \langle \U_{C_{n-1}}^{\M_{n-1}} \rangle$. Let $\MM^A_n(\M)$ be a formal version of $\MM^A_n$, where conditions $(\sigma, X_{n-1})$ are such that $\sigma$ is $M$-coded, and $X \in \M_{n-1}$. We need the following density lemma.

\begin{lemma}\label[lemma]{lem:main-forcing-induction}
For every $\Sigma_{n+1}^0$ formula $\phi(G,x)$ and every condition~$c \in \MM^A_n(\M)$, there is an extension~$d \leq c$
forcing $\neg \phi(G, x)$ to satisfy induction.
\end{lemma}
\begin{proof}
Let $c \in \MM^A_n(\M)$ be a condition. If $c \nqvdash (\exists x)\phi(G,x)$, then by \Cref{lem:main-question-find-extension}, there exists some extension $d \leq c$ forcing $(\forall x) \neg \phi(G,x)$, and we are done. So suppose  $c \qvdash (\exists x)\phi(G,x)$ and let $e \leq c$ forcing $\phi(G,b)$ for some $b \in M$. Let $I = \{ x \leq b : e \qvdash \phi(G, x) \}$. The set is $\Sigma_1^0(\M_n)$ and $b \in I$. Since $\M_n \models \ISig_1$, there exists some minimal element $a$ of $I$. If $a = 0$, then by \Cref{lem:main-question-find-extension}, there exists some extension $d \leq e$ forcing $\phi(G,0)$. If $a \neq 0$, then $e \nqvdash \phi(G,a-1)$ hence $e \Vdash  \neg \phi(G,a-1)$ and $e \qvdash \phi(G,a)$. By \Cref{lem:main-question-find-extension}, there exists an extension $d \leq e$ forcing $\phi(G,a)$. By \Cref{lem:main-forcing-closed-downwards}, $d$ still forces $\neg \phi(G,a-1)$.
\end{proof}

Let $\F$ be a sufficiently generic $\MM^A_n(\M)$-filter, and let $H = G_\F$. Every sufficiently generic filter is $n$-generic by \Cref{lem:core-question-find-extension} and \Cref{lem:main-question-find-extension}, thus $\F$ is $n$-generic.
By construction, $H \subseteq A$, and by \Cref{prop:core-forcing-generic-set}, the set $H$ is $\M$-infinite.
By \Cref{lem:main-forcing-induction}, $\F$ being sufficiently generic, it forces every $\Pi_{n+1}^0$ formula to satisfy induction.
Finally, by \Cref{prop:core-forcing-imply-truth}, every $\Sigma_{n+1}^0$ or $\Pi_{n+1}^0$ property forced by $\F$ will hold for $H$, thus $\M \cup \{H\} \models \IPi_{n+1}$, which is equivalent to~$\ISig_{n+1}$. Thus, by \Cref{lem:induction-to-generated}, $\M[H] \models \ISig_{n+1}$. Last, by \Cref{lem:formula-translation}, $\M[H] \models \RCA_0$.
\bigskip

\textbf{Case 2:} $A \notin \langle \U_{C_{n+1}}^{\M_{n+1}} \rangle$. Let $\WW^{\overline{A}}_n(\M)$ be a formal version of $\WW^{\overline{A}}_n$ where conditions $(\sigma,X_{n-1},X_n)$ are such that $\sigma$ is $M$-coded, $X_{n-1} \in \M_{n-1}$ and $X_n \in \M_n$. 

\begin{lemma}\label[lemma]{lem:witness-forcing-induction}
For every $\Sigma_{n+1}^0$ formula $\phi(G,x)$ and every condition~$c \in \WW^{\overline{A}}_n(\M)$, there is an extension~$d \leq c$
forcing $\neg \phi(G, x)$ to satisfy induction.
\end{lemma}
\begin{proof}
Let $c = (\sigma,X_{n-1},X_n) \in \WW^{\overline{A}}_n(\M)$ be a condition and let $\U$ be a witness for~$c$. If $c \nqvdash^{\U} (\exists x)\phi(G,x)$, then by \Cref{lem:witness-question-find-extension}, there exists some extension $d \leq c$ forcing $(\forall x) \neg \phi(G,x)$, otherwise, if $c \qvdash^{\U} (\exists x)\phi(G,x)$, there exists some extension $e \leq c$ forcing $\phi(G,b)$ to hold for some $b \in M$. Note that $e$ can be chosen to have the same reservoirs as $c$ up to finite changes, hence $\U$ is also a witness for~$c$. Define inductively the following sequence of $\Sigma_1^0(\M_{n-1})$ classes $(\U_n)_{n \leq b}$ by letting $\U_0 = \U$ and $\U_{n+1} = \L_2(\U_n)$ and let $I = \{ x \leq b : e \qvdash^{\U_x} \phi(G, x) \}$. The set $I$ is $\Sigma_1^0(\M_n)$, and non-empty (it contains $b$). As $\M_n \models \ISig_1$, there exists some minimal element $a$ of $I$. If $a = 0$, then by \Cref{lem:witness-question-find-extension}, there exists some extension $d \leq e$ forcing $\phi(G,0)$ and if $a \neq 0$, then $e \nqvdash^{\U_{a-1}} \phi(G,a-1)$ and $e \qvdash^{\U_a} \phi(G,a)$. By \Cref{lem:witness-question-pair}, there exists an extension $d \leq e$ such that $d \Vdash \phi(G,a)$ and $d \Vdash \neg \phi(G,a-1)$.
\end{proof}

Let $\F$ be a sufficiently generic filter for the $\WW^{\overline{A}}_n(\M)$ forcing, and let $H = G_{\F}$. $H$ is an $M$-infinite subset of $\overline{A}$ and, just as in case $1$, $\M[H] \models \RCA_0 + \ISig_{n+1}$
\end{proof}

\begin{repmaintheorem}{thm:main-conservation}
Let $n \geq 1$. $\RCA_0 + \ISig_{n+1} + \Sub{\Sigma^0_{n+1}}$ is $\Pi_1^1$ conservative over $\RCA_0 + \ISig_{n+1}$.
\end{repmaintheorem}

\begin{proof}
Assume $\RCA_0 + \ISig_{n+1} \not \vdash \forall X \phi(X)$ for $\phi(X)$ an arithmetic formula. By completeness, and the L\"owenheim-Skolem theorem, there exists some countable model $\M = (M,S) \models \RCA_0 + \ISig_{n+1} + \neg \phi(B)$ for some $B \in S$.

\Cref{prop:pi12-problem-conservation} cannot be applied on its current form to \Cref{lem:sigman+1-conservation-step} in order to build an $\omega$-extension $\N$ of $\M$ such that $\N \models \RCA_0 + \ISig_{n+1} + \Sub{\Sigma^0_{n+1}}$, this is because \Cref{lem:sigman+1-conservation-step} has the added assumption that the model considered is topped. 

This added assumption is not a problem: the initial model $\M$ can be assumed to be topped by $B$ (by restricting it to keep only the elements $\Delta_1^0$-definable using $B$), and the property of being topped is preserved by every application of \Cref{lem:sigman+1-conservation-step} (If a model $\hat{\M}$ is topped by a set $Y$, then $\hat{\M}[H]$ is topped by $Y \oplus H$). Therefore, in the proof of \Cref{prop:pi12-problem-conservation}, all the models can be assumed to be topped. \\

Therefore, there exists an $\omega$-extension $\N$ of $\M$ such that $\N \models \RCA_0 + \ISig_{n+1} + \Sub{\Sigma^0_{n+1}}$. By \Cref{lem:omega-extension-absolute}, $\N \models (\exists X) \neg \phi(X)$. Hence, $\RCA_0 + \ISig_{n+1} + \Sub{\Sigma^0_{n+1}} \not \vdash (\forall X)\phi(X)$.
\end{proof}

Thanks to the characterization of the Ginsburg-Sands theorem for $T_1$-spaces in terms of $\Sub{\Sigma^0_2} + \COH$ by Benham et al.~\cite{benham2024ginsburg} and the amalgamation theorem from Yokoyama~\cite{yokoyama2010conservativity}, we deduce the following corollary.

\begin{corollary}
$\RCA_0 + \ISig_2 + \GST$ is $\Pi^1_1$-conservative over~$\RCA_0 + \ISig_2$.
\end{corollary}
\begin{proof}
By Cholak, Jockusch and Slaman~\cite{cholak2001strength}, $\RCA_0 + \ISig_2 + \COH$ is $\Pi^1_1$-conservative over~$\RCA_0 + \ISig_2$. By \Cref{thm:main-conservation}, so is $\RCA_0 + \ISig_2 + \Sub{\Sigma^0_2}$.
It follows by the amalgamation theorem (\Cref{thm:amalgamation-pi11}) that $\RCA_0 + \ISig_2 + \Sub{\Sigma^0_2} + \COH$ is $\Pi^1_1$-conservative over~$\RCA_0 + \ISig_2$. We conclude as $\RCA_0 \vdash \GST \leftrightarrow (\Sub{\Sigma^0_2} + \COH)$ by Benham et al.~\cite{benham2024ginsburg}.
\end{proof}

It is currently unknown whether $\RCA_0 + \Sub{\Sigma^0_2}$ is $\Pi^1_1$-conservative over~$\RCA_0 + \BSig_2$. We now prove that this is not the case for $\Sub{\Delta^0_3}$.

\begin{proposition}\label{prop:delta3-bsig2-conservation}
    Let $n \geq 1$. $\RCA_0 + \Sub{\Delta^0_{n+2}}$ is not $\Pi_1^1$ conservative over $\RCA_0 + \BSig_{n+1}$.
\end{proposition}

\begin{proof}
Given a family of formulas $\Gamma$, let $\mathsf{C}\Gamma$ be the statement that no $\Gamma$-formula defines a total injection with bounded range. This scheme was introduced by Seetapun et Slaman~\cite{seetapun1995strength}.
The principle $\CSig_{n+2}$ satisfies the following properties:

\begin{enumerate}
    \item $\ISig_{n+1} \not \vdash \CSig_{n+2}$, see Groszek and Slaman \cite{groszek1994turing}
    \item $\BSig_{n+2} \vdash \CSig_{n+2}$. Indeed, by \cite[Theorem 2.23]{hajek2017metamathematics}, $\BSig_{n+2}$ is equivalent to the pigeonhole principle for $\Sigma^0_{n+2}$ ($\mathsf{PHP}(\Sigma^0_{n+2})$) which immediately implies $\CSig_{n+2}$.
    \item $\BSig_{n+1} \wedge \neg \ISig_{n+1} \vdash \CSig_{n+2}$, see  Ko\l odziejczyk et al. \cite{kolodziejczyk2023strong} 
\end{enumerate}

The first item yields that $\RCA_0 + \ISig_{n+1} \not \vdash \CSig_{n+2}$, the first-order part of $\RCA_0 + \ISig_{n+1}$ being $\mathsf{I}\Sigma_{n+1}$, hence $\RCA_0 + \BSig_{n+1} \not \vdash \CSig_{n+2}$. 
The two others items yield that $\BSig_{n+2} \vee (\BSig_{n+1} \wedge \neg \ISig_{n+1}) \vdash \CSig_{n+2}$, hence $\BSig_{n+1} \wedge (\ISig_{n+1} \to \BSig_{n+2}) \vdash \CSig_{n+2}$. By Chong, Lempp and Yang~\cite{chong2010role}, $\RCA_0 + \Sub{\Delta^0_{n+2}} \vdash \ISig_{n+1} \to \BSig_{n+2}$, so $\RCA_0 + \BSig_{n+1} + \Sub{\Delta^0_{n+2}} \vdash \BSig_{n+1} \wedge (\ISig_{n+1} \to \BSig_{n+2})$, hence $\RCA_0 + \BSig_{n+1} + \Sub{\Delta^0_{n+2}} \vdash \CSig_{n+2}$.
Therefore, $\Sub{\Delta^0_{n+2}}$ is not $\Pi_1^1$ conservative over $\RCA_0 + \BSig_{n+1}$.
\end{proof}

\section{Open questions}\label{sect:open-questions}

Many questions remain open concerning the pigeonhole hierarchy in reverse mathematics. Due to its connections with Ramsey's theorem for pairs, $\Sub{\Delta^0_2}$ was significantly more studied than the other levels. It was proven not to imply $\COH$ over non-standard models by Chong, Slaman and Yang~\cite{chong2014metamathematics}, and more recently over $\omega$-models by Monin and Patey~\cite{monin2021srt}. The question is open for higher levels of the hierarchy:

\begin{question}
Does $\Sub{\Sigma^0_n}$ imply~$\COH$ over~$\RCA_0$ for some~$n \in \NN$?
\end{question}

The first-order part of $\Sub{\Delta^0_2}$ received a lot of attention. It is known to follow strictly from~$\ISig_2$ and to imply $\BSig_2$. However, the following question is one of the most important questions of modern reverse mathematics:

\begin{question}
Is $\RCA_0 + \Sub{\Delta^0_2}$ $\Pi^1_1$-conservative over~$\RCA_0 + \BSig_2$?
\end{question}

The answer is conjectured to be negative. By \Cref{prop:delta3-bsig2-conservation}, $\RCA_0 + \Sub{\Delta^0_{3}}$ is not $\Pi_1^1$ conservative over $\RCA_0 + \BSig_2$, but it might still be the case for $\Sub{\Sigma^0_2}$.

\begin{question}\label{quest:sigma2-conservative-bsig2}
Is $\RCA_0 + \Sub{\Sigma^0_2}$ $\Pi^1_1$-conservative over~$\RCA_0 + \BSig_2$?
\end{question}

In \Cref{sect:conservation}, we proved that $\RCA_0 + \ISig_n + \Sub{\Sigma^0_n}$ is $\Pi^1_1$-conservative over~$\RCA_0 + \ISig_n$, for~$n \geq 2$. It is however still unknown whether any sufficiently high level of the pigeonhole hierarchy even implies $\ISig_2$.

\begin{question}\label{quest:sigman-imply-isig2}
Does $\RCA_0 + \Sub{\Sigma^0_n}$ imply $\ISig_2$ for some~$n \in \NN$?
\end{question}

The well-foundedness principle $\WF(\alpha)$ states that there is no infinite decreasing sequence of ordinals smaller than~$\alpha$. In particular, $\WF(\omega^\omega)$ admits several characterizations, among which the statement of the totality of Ackermann's function (see Kreuzer and Yokoyama~\cite{kreuzer2016principles}). The principle $\WF(\omega^\omega)$ is known to follow strictly from $\ISig_2$ and be incomparable with~$\BSig_2$ (see Simpson~\cite{simpson2015comparing}).
Patey and Kokoyama~\cite{patey2018proof} proved that $\RCA_0 + \Sub{\Delta^0_2}$ is $\forall \Pi^0_3$-conservative over $\RCA_0$, hence does not imply $\WF(\omega^\omega)$ over~$\RCA_0$.

\begin{question}\label{quest:sigman-imply-wfomegaomega}
Does $\RCA_0 + \Sub{\Sigma^0_n}$ imply $\WF(\omega^\omega)$ for some~$n \in \NN$?
\end{question}

A positive answer to \Cref{quest:sigman-imply-isig2} would yield a positive answer to \Cref{quest:sigman-imply-wfomegaomega}, while a positive answer to \Cref{quest:sigman-imply-wfomegaomega} for $n = 2$ would yield a negative answer to \Cref{quest:sigma2-conservative-bsig2}.

\section*{Acknowledgment}

The authors are thankful to the anonymous referee for his improvement suggestions.

\bibliographystyle{plain}
\bibliography{biblio}

\end{document}